\documentclass[11pt,a4paper,reqno]{amsart}

\usepackage{enumerate}
\usepackage{aliascnt}
\usepackage{hyperref}
\usepackage{amsfonts}
\usepackage{mathrsfs}
\usepackage{amsmath}
\usepackage{amsmath,graphics}
\usepackage{amssymb}
\usepackage{indentfirst,latexsym,bm,amsmath,pstricks,amssymb,amsthm,graphicx,fancyhdr,float,color}
\usepackage[all]{xy}
\usepackage{amsthm}
\usepackage{amscd}
\usepackage[all]{hypcap}
\usepackage{tikz,diagbox}

\newtheorem{theorem}{Theorem}[section]

\newaliascnt{lemma}{theorem}
\newtheorem{lemma}[lemma]{Lemma}
\aliascntresetthe{lemma}

\newaliascnt{conjecture}{theorem}

\aliascntresetthe{conjecture}

\newaliascnt{proposition}{theorem}
\newtheorem{proposition}[proposition]{Proposition}
\aliascntresetthe{proposition}

\newaliascnt{corollary}{theorem}
\newtheorem{corollary}[corollary]{Corollary}
\aliascntresetthe{corollary}

\newaliascnt{problem}{theorem}

\aliascntresetthe{problem}

\newaliascnt{question}{theorem}

\aliascntresetthe{question}

\newaliascnt{claim}{theorem}
\newtheorem{claim}[claim]{Claim}
\aliascntresetthe{claim}

\theoremstyle{definition}

\newaliascnt{definition}{theorem}
\newtheorem{definition}[definition]{Definition}
\aliascntresetthe{definition}

\newaliascnt{example}{theorem}
\newtheorem{example}[example]{Example}
\aliascntresetthe{example}

\newaliascnt{assumption}{theorem}

\aliascntresetthe{assumption}

\theoremstyle{remark}

\newaliascnt{remark}{theorem}
\newtheorem{remark}[remark]{Remark}
\aliascntresetthe{remark}

\newaliascnt{remarks}{theorem}

\aliascntresetthe{remarks}

\numberwithin{equation}{section}
\renewcommand{\theequation}{\arabic{section}.\arabic{equation}}

\numberwithin{figure}{section}

\def\wt{\widetilde}
\def\ol{\overline}

\def\lra{\longrightarrow}

\def\div{\text{\rm{div\,}}}

\def\({$($}
\def\){$)$}

\def\bbp{\mathbb P}

\def\call{\mathcal L}

\def\calo{\mathcal O}

\DeclareMathOperator*{\vol}{\ensuremath{vol}}

\newcommand{\df}{\mathrm{d}f}
\newcommand{\tang}{\mathrm{tang}}
\newcommand{\sF}{\mathcal{F}}
\newcommand{\cs}{\mathrm{CS}}

\def\kod{\mathrm{Kod}}

\begin{document}

	\title{The canonical map of a foliated surface of general type}	
	
	\author{Xin L\"u}
	
	\address{School of Mathematical Sciences,  Key Laboratory of MEA(Ministry of Education) \& Shanghai Key Laboratory of PMMP,  East China Normal University, Shanghai 200241, China}
	
	\email{xlv@math.ecnu.edu.cn}
	
	\thanks{This work is supported by Shanghai Pilot Program for Basic Research,
		National Natural Science Foundation of China,  Fundamental Research Funds for central Universities, and Science and Technology Commission of Shanghai Municipality (No. 22DZ2229014)}
	
	
	
	
	
	
	\maketitle
	
	\begin{abstract}
Let $(S,\mathcal{F})$ be a foliated surface over the complex number of general type, i.e., the Kodaira dimension $\mathrm{Kod}(\mathcal{F})=2$.
We study the geometry of the canonical map $\varphi$ of the foliated surface $(S,\mathcal{F})$, and prove several boundedness results on the canonical map $\varphi$,
generalizing Beauville's beautiful work \cite{bea-79}
on the canonical maps of algebraic surfaces to foliated surfaces. As an application,
we prove three Noether type inequalities for $(S,\mathcal{F})$ depending on the Kodaira dimension of the surface $S$.
	\end{abstract}

		\maketitle

\setcounter{tocdepth}{2}
\tableofcontents

\section{Introduction}\label{sec-intro}
The main purpose of the paper is to study the canonical map of a
 foliated surface of general type over the complex number $\mathbb{C}$.
 
Let $X$ be a smooth projective variety of general type over the complex number $\mathbb{C}$,
and $K_X$ be its canonical divisor.
The $n$-th pluri-geometric genus of $X$ is defined to be
$$p_n(X):=\dim H^0\big(X, nK_X\big).$$
When $n=1$, it is called the geometric genus of $X$ and also denoted by $p_g(X)=p_1(X)$.
If $p_n(X)\geq 2$, one considers the rational map
$$\varphi_n:=\varphi_{|nK_X|}:\,X \dashrightarrow \bbp^{p_n(X)-1},$$
defined by the complete linear system $|nK_X|$.
The map $\varphi_n$ is called the $n$-th pluri-canonical map of $X$, and $\varphi_1$ is also called the canonical map.
The canonical and pluri-canonical maps are very important
in the classification of projective varieties of general type, cf. \cite{acgh,bhpv,bea-79,bom-73,hm-06,tak-06}.

The picture is clear when $X=C$ is a smooth projective curve:
the pluri-canonical map $\varphi_n$ is always birational if $n\geq 3$;
and the canonical map $\varphi_1$ is birational if and only if $C$ is non-hyperelliptic.
If $C$ is indeed a hyperelliptic curve, then $\varphi_1$ gives a double cover onto a projective line $\bbp^1$.

When $X=S$ is a smooth projective surface of general type,
Bombieri \cite{bom-73} proved that the pluri-canonical map $\varphi_n$ is birational once $n\geq 5$.
Later, Beauville \cite{bea-79} provided a systematical study on the canonical map $\varphi_1$.
Among others the following significant boundedness results are proved.
Denote by $\vol(S)$ the volume of $S$, which is defined to be
$$\vol(S)=\limsup_{n\to +\infty} \frac{\dim H^0(S, nK_S)}{n^2/2}.$$
\begin{theorem}[Beauville]\label{thm-1-0}
	Let $\varphi_1:\,S \dashrightarrow \Sigma \hookrightarrow \bbp^{p_g(S)-1}$ be canonical map of $S$.
	\begin{enumerate}[$(i)$]
		\item If the image $\Sigma$ is of dimension two, then
		\begin{equation}\label{eqn-1-1}
		\deg(\varphi_1)\leq \frac{\vol(S)}{p_g(S)-2}.
		\end{equation}
		By the Miyaoka-Yau inequality, it follows that
		\begin{equation}\label{eqn-1-01}
			\deg(\varphi_1)\leq 36; \qquad\text{and}\qquad
			\deg(\varphi_1)\leq 9,~\text{if\,}~ p_g(S)\geq 30.
		\end{equation}
		
		\item If the image $\Sigma$ is of dimension one, i.e., it induces a fibration
		$f:\, \wt{S} \to B$ of genus $g$ after possible blowing-ups and the Stein factorization, then
		\begin{equation}\label{eqn-1-2}
		2\leq g \leq \frac{\vol(S)}{2(p_g(S)-1)}+1.
		\end{equation}
		By the Miyaoka-Yau inequality, it follows that
		\begin{equation}\label{eqn-1-02}
			2\leq g\leq 5,\quad\text{if}\quad p_g(S)\geq 20.
		\end{equation} 
	\end{enumerate}
\end{theorem}
Remark that the lower bound $g\geq 2$ in \eqref{eqn-1-2} and \eqref{eqn-1-02} follows from the assumption that $S$ is of general type.
Beauville's breakthrough work serves as a gateway to the study on the canonical map of surfaces of general type, and hence inspires a series of subsequent noteworthy works;
see for instance \cite{xiao-85-1,xiao-86,sun-94,chen-17,yeung-17,lu-20,rito-22}.

In recent years, the foliation theory has garnered significant interest in the realm of algebraic geometry, particularly in birational geometry.
Specifically, foliations whose canonical divisor $K_{\sF}$ is not pseudo-effective have been thoroughly characterized in \cite{miy-87,bm-16,cp-19}.
Notably, in the remarkable paper \cite{cas-21},
the (conjectural) minimal model program for foliated varieties was introduced.
This innovation sparked a novel project in birational geometry aimed at extending these concepts to foliated varieties.
Substantial noteworthy progress has already been made in this direction. For instance, Miyaoka \cite{miy-87} and Campana-P\u{a}un \cite{cp-19}
delves into foliated varieties with non-pseudoeffective canonical divisors.
Brunella and McQuillan, among others, have explored the minimal model theory for foliated surfaces \cite{bru-99,mcq-08}.
Furthermore, the deformation invariance of pluri-genera \cite{cf-18,liu-19} adds another remarkable progress to this research.
Recently, there are also many significant progress on the study of the adjoint canonical divisor, cf. \cite{ss-23,chlmssx}.
These studies offer strong evidence that the canonical divisor $K_{\sF}$
of a foliation plays an analogous role to the canonical divisor $K_X$ of an algebraic variety.
The canonical divisor $K_{\sF}$ can be utilized to study not only the geometry of the foliation but also the geometry of the variety itself.
This has also inspired our study on the canonical map of a foliated surface of general type.

Let $\sF$ be a foliation on a smooth projective surface $S$.
The geometric genus $p_g(\sF)$ and the volume $\vol(\sF)$ of $\sF$
are defined over some reduced model $(S',\sF')$:
$$\begin{aligned}
	p_g({\sF})&\,=\dim H^0(S', K_{\sF'});\\
	\vol({\sF})&\,=\limsup_{n\to +\infty} \frac{\dim H^0(S', nK_{\sF'})}{n^2/2},
\end{aligned}$$
where $K_{\sF'}$ is the canonical divisor of $\sF'$.
These are two birational invariants, and independent on the choice
of the reduced model;
see \autoref{sec-canonical-map} for more details.
The foliation $\sF$ is said to be of general type if its Kodaira dimension is $2$, or equivalently if $\vol(\sF)>0$.
When $p_g(\sF) \geq 2$, one defines the canonical map
$\varphi$ as the rational map given by the complete linear system $|K_{\sF'}|$:
\begin{equation*}
	\varphi=\varphi_{|K_{\sF'}|}:\,S' \dashrightarrow \Sigma \subseteq \mathbb{P}^{p_g(\sF)-1}.
\end{equation*}
Equivalently, composed with the birational map $S \dashrightarrow S'$, one may also view the canonical map $\varphi$ as a rational map mapping the surface $S$ to the projective space: 
$$\varphi:\,S \dashrightarrow \Sigma \subseteq \mathbb{P}^{p_g(\sF)-1}.$$

Consider first the case when the image $\Sigma$ is of dimension two, or equivalently, the canonical map is generically finite.
In this case, $p_g(\sF) \geq 3$.
\begin{theorem}\label{thm-1-1}
	Let $\sF$ be a foliation of general type on a smooth projective surface $S$, such that its canonical map $\varphi$ induces a generically finite map as above.
	Then
	\begin{equation}\label{eqn-1-3}
		\deg(\varphi) \leq \frac{\vol(\sF)}{p_g(\sF)-2}.
	\end{equation}
	Moreover, if the equality holds in \eqref{eqn-1-3}, then the image $\Sigma$ is a surface of minimal degree in $\bbp^{p_g(\sF)-1}$.
	In particular,
	\begin{equation}\label{eqn-1-3'}
		\deg(\varphi)\leq \vol(\sF),
	\end{equation}
	and if the equality holds in \eqref{eqn-1-3'}, then $p_g(\sF)=3$, i.e., the image $\Sigma\cong \bbp^2$.
\end{theorem}

We will construct examples reaching the equalities in both \eqref{eqn-1-3} and \eqref{eqn-1-3'}; see \autoref{exam-3-1} and \autoref{exam-3-2}.
A special case is when the canonical map $\varphi$ is birational, i.e., $\deg\varphi=1$,
in which situation we would like to call
 $(S,\sF)$ a {\it canonical foliated surface}.
\begin{theorem}\label{thm-1-2}
	Let $(S,\sF)$ be a canonical foliated surface of general type, i.e., the canonical map $\varphi$ is birational. Then
	\begin{enumerate}[$(i)$]
		\item The following inequality holds.
	\begin{equation}\label{eqn-1-4}
		\vol(\sF) \geq p_g(\sF)-2.
	\end{equation}
The equality in \eqref{eqn-1-4} implies that
the image $\Sigma$ is a surface of minimal degree (equal to $p_g(\sF)-2$) in $\bbp^{p_g(\sF)-1}$.
In particular, $S$ is rational.

\item Suppose that the surface $S$ is not ruled. Then
\begin{equation}\label{eqn-1-5}
\vol(\sF) \geq 2p_g(\sF)-4,
\end{equation}
Moreover, if the equality holds in \eqref{eqn-1-5}, then $S$ is birational to a K3 surface.

\item Suppose that the surface $S$ is also of general type. Then
\begin{equation}\label{eqn-1-5-1}
	\vol(\sF) \geq 2p_g(\sF)-4+\frac{\sqrt{8p_g(\sF)-31}-7}{2},
\end{equation}
Moreover, if the equality holds in \eqref{eqn-1-5-1}, then $S$ is birational to a  $(1,0)$-surface.
If $S$ is not birational to a  $(1,0)$-surface, then
\begin{equation}\label{eqn-1-5-2}
\vol(\sF) \geq 2p_g(\sF)-4+\frac{\sqrt{8p_g(\sF)-23}-3}{2},
\end{equation}
Moreover, if the equality holds in \eqref{eqn-1-5-2}, then $S$ is birational to a  $(1,2)$-surface.
Here by an $(a,b)$-surface we mean a minimal surface $Y$ of general type
with $\vol(Y)=a$ and $p_g(Y)=b$.
\end{enumerate}	
\end{theorem}

We refer to \cite[Example\,9.1]{lt-24} for a canonical foliated surface of general type reaching the equality in \eqref{eqn-1-4}.
By \eqref{eqn-1-3}, it follows that the foliated surface $(S_n,\sF_n)$ in \cite[Example\,9.1]{lt-24} is canonical.
We will also construct examples in \autoref{sec-exam-3}
reaching the equalities in both \eqref{eqn-1-5} and \eqref{eqn-1-5-2}.

\vspace{2mm}
Consider next the case when the image $\Sigma$ is of dimension one.
In this case, after possible blowing-ups and taking the Stein factorization,
the canonical map $\varphi$ induces a fibration $f:\,\wt{S} \to B$.
	$$\xymatrix{&& \wt{S} \ar[dll]_-{f} \ar[d]^-{\phi} \ar[rrr]^-{\sigma}  &&& S \ar@{-->}[d]^-{\varphi}\\
	B \ar[rr]^-{\pi} && Y \ar[rrr]^-{\rho}_-{\text{desingularization}} &&&\Sigma \, \ar@{^(->}[r] & \bbp^{p_g(\sF)-1}}$$
A prior, There are two foliations on the surface $\wt{S}$:
the first is $\wt\sF=\sigma^*\sF$ by pulling-back the foliation $\sF$ by the birational map $\sigma$;
the other one is $\mathcal{G}$ defined by taking the saturation of $\ker(\df:\,T_{\wt S} \to f^*T_{B})$ in $T_{\wt S}$.

\begin{theorem}\label{thm-1-3}
	Let $\sF$ be a foliation of general type on a smooth projective surface $S$, such that its canonical map $\varphi$ induces a fibration $f:\, \wt{S} \to B$ as above, and let $F$ be a general fiber of $f$.
	Suppose that  $\wt\sF=\sigma^*\sF$ is the same as the foliation defined by taking the saturation of $\ker(\df:\,T_{\wt S} \to f^*T_{B})$ in $T_{\wt S}$. Then $d:=K_{\wt \sF}\cdot F=2g(F)-2\geq 2$ and
	\begin{equation}\label{eqn-1-6}
	\vol(\sF) \geq \frac{d(d+2)}{d+1}\big(p_g(\sF)-1\big)
	=\frac{4\big(g(F)-1\big)g(F)}{2g(F)-1}\big(p_g(\sF)-1\big).
	\end{equation}
\end{theorem}

\begin{theorem}\label{thm-1-4}
	Let $\sF$ be a foliation of general type on a smooth projective surface $S$, such that its canonical map $\varphi$ induces a fibration $f:\, \wt{S} \to B$ as above, and let $F$ be a general fiber of $f$.
	Suppose that  $\wt\sF=\sigma^*\sF$ is different from the foliation defined by taking the saturation of $\ker(\df:\,T_{\wt S} \to f^*T_{B})$ in $T_{\wt S}$.
	\begin{enumerate}[$(i)$]
		\item It holds that $d=K_{\wt \sF} \cdot F\geq 1$, and the canonical divisor $K_{\wt \sF}$ can be written as
		\begin{equation}\label{eqn-1-7}
		K_{\wt \sF} = f^*K_B+\Delta,
		\end{equation}
		where $\Delta$ is effective. In particular, $p_g(\sF)=p_g(\wt \sF) \geq g(B)$. Moreover, the following inequality holds
		\begin{equation}\label{eqn-1-8}
		\vol(\sF) \geq \frac{d\,(p_g(\sF)-1)^2}{p_g(\sF)}.
		\end{equation}
		
		\item Suppose that $g(B)\geq 1$. Then
		\begin{equation}\label{eqn-1-9}
		\vol(\sF) \geq \frac{d(d+2)}{d+1}\big(p_g(\sF)-1+g(B)\big).
		\end{equation}
		
		\item Suppose that $d=K_{\wt \sF} \cdot F <2g(F)-2$. Then either $g(B)=0$, or $p_g(\sF)=g(B)\geq 2$.
		Moreover, if $p_g(\sF)=g(B)\geq 2$, then
		\begin{equation}\label{eqn-1-10}
			\vol(\sF) \geq \frac{2d(d+2)}{d+1}\big(p_g(\sF)-1\big).
		\end{equation}
	\end{enumerate}
\end{theorem}

\begin{remark}
	(i) In the above two theorems,
	if the foliation $\wt\sF=\sigma^*\sF$ is the same as the one $\mathcal{G}$
	defined by taking the saturation of $\ker(\df:\,T_{\wt S} \to f^*T_{B})$ in $T_{\wt S}$, then $\sF$ is algebraically integrable by definition.
	The converse is not true in general. Namely, there exist examples of
	algebraically integrable foliations whose canonical maps induce fibrations as above,
	but $\wt \sF \neq \mathcal{G}$; see \autoref{exam-6-5} for such a foliation,
	in which case it even holds that $d=K_{\wt \sF} \cdot F=1$.

	(ii)
	Due to the lack of the Miyaoka-Yau type inequality for foliated surfaces of general type,
	there is no absolute upper bound on the degree of the canonical map $\varphi$ for foliated surfaces in \autoref{thm-1-1} if it is generically finite;
	and there is no absolute upper bound either on the genus $g(F)$ (and the degree $d=K_{\sF}\cdot F$)
	of a general fiber $F$ in \autoref{thm-1-3} and \autoref{thm-1-4} when the canonical map $\varphi$ 
	induces a fibration $f:\, \wt S \to B$.
	This is different from Beauville's result in \autoref{thm-1-0}.
\end{remark}

As a byproduct of the study on the canonical map, we obtain the following Noether type inequalities for foliations of general type.
\begin{theorem}\label{thm-1-5}
	Let $\sF$ be a foliation of general type on a smooth projective surface $S$.
	\begin{enumerate}[$(i)$]
		\item The following Noether inequality holds.
		\begin{equation}\label{eqn-1-12}
		\vol(\sF) \geq p_g(\sF)-2.
		\end{equation}
		Moreover, if the equality in \eqref{eqn-1-12} holds,
		then the canonical map $\varphi$ induces a birational map whose image $\Sigma$ is a surface is a surface of minimal degree $($equal to $p_g(\sF)-2)$ in $\bbp^{p_{\sF}-1}$. In particular, $S$ is rational.
		\item Suppose that the surface $S$ is not ruled. Then
		\begin{equation}\label{eqn-1-13}
		\vol(\sF) \geq \min\left\{\frac{3}{2}\big(p_g(\sF)-1\big),~2\big(p_g(\sF)-2\big)\right\}.
		\end{equation}
		\item Suppose that the surface $S$ is of general type. Then
		\begin{equation}\label{eqn-1-14}
		\vol(\sF) \geq 2\big(p_g(\sF)-2\big).
		\end{equation}
		Moreover, if the equality in \eqref{eqn-1-14} holds, then the canonical map $\varphi$ induces a two-to-one map whose image $\Sigma$ is a surface of minimal degree $($equal to $p_g(\sF)-2)$ in $\bbp^{p_g(\sF)-1}$.
	\end{enumerate} 
\end{theorem}
\begin{remark}\label{rem-1-1}
	(i) The classical Noether inequality for surfaces of general type was established by Noether \cite{noether} more than one hundred years ago.
	Chen-Chen-Jiang \cite{ccj-201,ccj-202} obtained the sharp Noether type inequality for $3$-folds of general type.
	Remark also that Chen-Jiang \cite{cj-17} proved the Noether type inequality in any dimension.
	
	(ii) In \cite{lt-24}, we applied the Noether type inequalities to study the Poincar\'e problem on the algebraic integrability of foliated surfaces.
	In fact, the first Noether inequality \eqref{eqn-1-12} has been already proved in \cite[Theorem\,1.9(i)]{lt-24}, and it is sharp.
	\autoref{exam-3-1} with $d=2$ shows that the third Noether inequality \eqref{eqn-1-14} is also sharp; see also \autoref{rem-5-1}.
	However, it is not clear whether the second Noether inequality \eqref{eqn-1-13} is sharp or not;
	see \autoref{prop-5-1} for an improvement when the foliation $\sF$ is algebraically integrable.
\end{remark}

\vspace{2mm}
The paper is organized as follows.
In \autoref{sec-pre}, we review some basic facts about foliated surfaces.
In \autoref{sec-gen-finite}, we consider the case when the canonical map $\varphi$ is generically finite.
We will first study the lower bound on the degree of a non-degenerate surface in a projective space in \autoref{sec-non-degenerate};
the proofs of \autoref{thm-1-1} and \autoref{thm-1-2} will be given in \autoref{sec-gen-finte-proof};
and we will construct several examples in \autoref{sec-exam-3} showing the sharpness of the inequalities in both \autoref{thm-1-1} and \autoref{thm-1-2}.
In \autoref{sec-fibration}, we consider the case when the image of the canonical map $\varphi$ is of dimension one, and hence it induces a fibration $f:\,\wt S \to B$.
Let $\mathcal{G}$ be the foliation defined by taking the saturation of
$\ker(\df:\,T_{\wt S} \to f^*T_{B})$ in $T_{\wt S}$, and $\wt \sF:=\sigma^*(\sF)$ the pulling-back foliation,
where $\sigma:\,\wt S \to S$ is the induced birational morphism.
In \autoref{sec-foliation-same}, we study the case when $\wt \sF=\mathcal{G}$ and prove \autoref{thm-1-3};
in \autoref{sec-foliation-different}, we study the case when $\wt \sF\neq \mathcal{G}$ and prove \autoref{thm-1-4};
lastly in \autoref{sec-exam-4}, we will construct examples showing both cases can indeed happen.
In \autoref{sec-noether}, we will apply our study on the canonical maps of foliated surfaces of general type
to prove the three Noether type inequalities in \autoref{thm-1-5}.
Finally in \autoref{sec-appendix}, we present the Riemann-Hurwitz formula for (reduced) foliated surfaces.

\vspace{2mm}
\noindent
{\bf Acknowledgements.}
The author would like to thank Professor Yifan Chen for a communication on the triple-canonical
maps of surfaces of general type. In particular, the proof of \autoref{lem-3-2}\,(ii) owes to him.

\section{Preliminary}\label{sec-pre}
In this section,
we recall some basic facts about the foliated surfaces. For more details we refer to \cite{bru-04,mcq-08}.
\subsection{Basic definitions}
We work over the complex number $\mathbb{C}$.
By a foliated surface $(S,\sF)$ we mean a foliation $\sF$ on a smooth projective surface $S$ over the complex number $\mathbb{C}$.
\begin{definition}\label{def-foliation}
	A foliation $\mathcal{F}$ on a smooth projective surfacce $S$ over the complex number $\mathbb{C}$ is defined by a saturated invertible subsheaf $T_{\mathcal{F}} \subseteq T_S$ in the tangent sheaf $T_{S}$,
	i.e.,  $T_S/T_{\mathcal{F}}$ is torsion free.
\end{definition}
The singular locus of $\mathcal{F}$ consists of the points $x\in S$
where the quotient sheaf $T_X/T_{\mathcal{F}}$
fails to be locally free at $x$.
The torsion-freeness of $T_X/T_{\mathcal{F}}$ ensures that the singular locus of $\sF$ in $S$ has codimension of at least two,
i.e., it consists of finitely many points.
In other words, a foliation $\sF$ on $S$ can be expressed by an exact sequence
$$0\lra T_{\sF} \lra T_S  \lra \mathcal{I}(N_{\sF}) \lra 0,$$
where $T_{\sF}$ and $N_{\sF}$ are respectively called the {\it tangent bundle} and {\it normal bundle} of $\sF$,
and $\mathcal{I}$ is an ideal sheaf supported on the singular locus of $\sF$.
Equivalently, a foliation on $S$ can be described by the data $\{(U_i, v_i)\}_{i\in I}$,
where $\{U_i\}_{i\in I}$ is an open covering of $S$,
$v_i$ is a holomorphic vector field on $U_i$ with at most isolated zeros,
and there exists $g_{ij} \in \mathcal{O}^*_S(U_i\cap U_j)$ whenever $U_i\cap U_j \neq \emptyset$ such that
\begin{equation}\label{eqn-2-8}
v_i|_{U_i\cap U_j}=g_{ij}v_j|_{U_i\cap U_j}.
\end{equation}
The cocycle $\{g_{ij}\}$ defines a line bundle which is precisely the dual of the tangent bundle $T_{\sF}^*$.
The bundle $T_{\sF}^*$ will be called the {\it canonical divisor} of $\mathcal{F}$, and we denote it by $K_{\mathcal{F}}:=T_{\sF}^*$.

Alternatively, one can also define $\sF$ using one-forms instead of vector fields.
A foliation on $S$ is characterized by a collection of one-forms
$\omega_i\in \Omega^1_S(U_i)$ that posses only isolated zeros.
Moreover, there exists a function $f_{ij} \in \mathcal{O}^*_S(U_i\cap U_j)$ whenever $U_i\cap U_j \neq \emptyset$ satisfying the relation:
\begin{equation}\label{eqn-2-9}
\omega_i|_{U_i\cap U_j}=f_{ij}\omega_j|_{U_i\cap U_j}.
\end{equation}
The cocycle $\{f_{ij}\}$ gives rise to a line bundle which is the conormal bundle $N_{\sF}^*$.
One can also translate this into an exact sequence as follows.
$$0\lra N_{\sF}^* \lra \Omega_S^1  \lra \mathcal{I}(K_{\sF}) \lra 0.$$

For any surjective morphism $\Pi:\,\wt{S} \to S$ from another smooth projective surface $\wt{S}$,
there exists a natural foliation $\Pi^*\sF$ induced by pulling-back $\sF$ on $\wt{S}$.
Let $\{(U_i,\omega_i)\}_{i\in I}$ be a collection 
of local one-forms defining the foliation $\sF$.
Initially, one may assume that $\wt{\sF}$ is simply given by $\{(\Pi^{-1}U_i, \Pi^*\omega_i)\}_{i\in I}$.
However, this is not always the case, as $\Pi^*\omega_i$ may posses one-dimensional zeros.
Instead, the foliation $\Pi^*\sF$ is defined by the data
$$\{(V_{ij}, \wt\omega_{ij})\},~i\in I, j\in J,$$
where $V_{ij}$ is an open covering of $\wt{S}$ such that $\Pi(V_{ij}) \subseteq U_i$,
and $\wt\omega_{ij}=\frac{\Pi^*(\omega_i)}{h_{ij}}$ with $h_{ij}$ being a holomorphic function over $V_{ij}$ satisfying $\div(h_{ij})=\div\big(\Pi^*(\omega_i)|_{V_{ij}}\big)$.
In particular, for any blowing-up $\sigma:\,\wt{S} \to S$ centered at some point $p\in S$,
there is an induced foliation $\sigma^*\sF$ on $\wt{S}$,
such that the restricted foliations $(\sigma^*\sF)|_{\wt{S}\setminus E} \cong \sF|_{S\setminus p}$ under the isomorphism $\wt{S} \setminus E \cong S \setminus\{p\}$,
where $E$ is the exceptional curve.

\vspace{2mm}
Suppose that $p$ is a singular point of $\sF$,
and that $v=A(x,y)\frac{\partial}{\partial x}+B(x,y)\frac{\partial}{\partial y}$ is a local vector field defining $\sF$.
The two eigenvalues $\lambda_1,\lambda_2$ of the differential matrix
$$Dv:=\left(\begin{aligned}
	A_x~&~A_y \\
	B_x~&~B_y
\end{aligned}\right)$$
at the point $p$ are uniquely determined up to multiplication by a non-zero constant,
where $A_x=\frac{\partial A(x,y)}{\partial x}$, and similar for other notations in the above matrix.
\begin{definition}\label{def-2-1}
	A singularity $p$ of $\sF$ is called a {\it reduced singularity} if at least one of the two eigenvalues (say, $\lambda_2$) is not zero and the quotient $\lambda=\frac{\lambda_1}{\lambda_2}$ is not a positive rational number.
	The foliation $\sF$ is said to be reduced if all its singularities are reduced.
\end{definition}
It is worth noting that the quotient $\lambda=\frac{\lambda_1}{\lambda_2}$ remains when $v$ is multiplied by a nonvanishing holomorphic function.
Naturally, if $\lambda_1\neq 0$, one could consider the reciprocal quotient
$\lambda^{-1}=\frac{\lambda_2}{\lambda_1}$ instead,
but the condition $\lambda \not\in \mathbb{Q}^{+}$ is equivalent to $\lambda^{-1} \not\in \mathbb{Q}^{+}$.
With a slight abuse of notation due to the exchange $\lambda \leftrightarrow \lambda^{-1}$, 
the complex number $\lambda=\frac{\lambda_1}{\lambda_2}$ is referred to as the {\it eigenvalue} of $\sF$ at $p$, following \cite{bru-04}.
A reduced singularity is called {\it non-degenerate} if both of the two eigenvalues $\lambda_1$ and $\lambda_2$ are non-zero; otherwise it is called a {\it saddle node}.
Given any foliation, one can obtain a reduced one through a sequence of blowing-ups:
\begin{theorem}[{Seidenberg, \cite[Theorem\,1.1]{bru-04}}]\label{thm-seidenberg}
	For any foliated surface $(S,\sF)$, there exists a sequence of blowing-ups $\pi:\,\wt{S} \to S$,
	such that the induced foliation $\wt\sF=\pi^*\sF$ on $\wt{S}$ is reduced.
\end{theorem}


	\begin{definition}
	Let $\sF$ be reduced foliation on a smooth projective $S$.
	An irreducible curve $C\subseteq S$ is $\sF$-exceptional if
	\begin{enumerate}[(i).]
		\item $C$ is an exceptional curve of first kind on $S$, i.e., it is a smooth rational curve with $C^2=-1$;
		\item the contraction of $C$ to a point produces a new foliation $(S_0,\sF_0)$ which is still reduced.
	\end{enumerate}
\end{definition}

\begin{definition}\label{def-min}
	A foliated surface $(S,\sF)$ is called relatively minimal if the foliation $\sF$ is reduced and there is no $\sF$-exceptional curve on $S$.
\end{definition}

It is established that any foliated surface $(S,\sF)$ has a relatively minimal model, cf. \cite[Proposition\,5.1]{bru-04}.
We assume in the following that $\sF$ is a relatively minimal foliation on a smooth projective surface $S$
such that $K_{\sF}$ is pseudo-effective.
In fact, the canonical divisor $K_{\sF}$ is pseudo-effective if and only if $\sF$ is not induced by a $\bbp^1$-fibration, cf. \cite{miy-87}.
Denote the Zariski decomposition of $K_{\sF}$ by
\begin{equation}\label{eqn-zariski}
K_{\sF}=P+N,
\end{equation}
where $P$ is the nef part and $N$ is the negative one.
By the Riemann-Roch theorem, we have
$$\vol(\sF)=P^2.$$
Moreover, McQuillan proved that the support of the negative part $N$ is a disjoint union of maximal $\sF$-chains.
\begin{definition}
	Let $\sF$ be a relatively minimal foliation on a smooth projective surface $S$.
	We say a curve $C\subseteq S$ is an $\sF$-chain if
	\begin{enumerate}[(i).]
		\item the curve $C$ is a Hirzebruch-Jung string, i.e., $C=\cup_{j=1}^{r} C_j$, each $C_j$ is a smooth rational curve with $C_j^2\leq -2$, $C_j\cdot C_i=1$ if $|i-j|=1$ and $0$ if $|i-j|\geq 2$;
		\item each irreducible component $C_j$ is $\sF$-invariant;
		\item $\mathrm{Sing}(\sF)\cap C$ are all reduced and non-degenerate;
		\item $Z(\sF,C_1)=1$, and $Z(\sF,C_j)=2$ for any $2\leq j\leq r$.
	\end{enumerate}
\end{definition}

\begin{theorem}[{\cite[Theorem\,8.1]{bru-04}}]\label{thm-3-4}
	Let $\sF$ be a relatively minimal foliation on a smooth projective surface $S$.
	Suppose that $K_{\sF}$ is pseudo-effective with the Zariski decomposition as in \eqref{eqn-zariski}.
	Then the support $\mathrm{Supp}(N)$ is a disjoint union of maximal $\sF$-chains,
	and $\lfloor N \rfloor=0$.
\end{theorem}

An important way to produce foliations comes from fibrations on smooth projective surfaces.
\begin{definition}\label{def-algebraic}
	A foliation $\sF$ is {\it algebraically integrable}, if it comes from a fibration of curves.
\end{definition}
A fibration  $f:\,S \to B$ of curves
is given by a proper surjective morphism $f$ from a smooth projective surface $S$ onto a smooth projective curve $B$ with connected fibers.
Given a fibration of curves $f:\,S \to B$,
it defines a foliation $\sF$ on $S$ by taking the saturation of $\ker(\df:\,T_{S} \to f^*T_{B})$ in $T_{S}$.
The canonical divisor $K_{\sF}$ is simple:
\begin{equation}\label{eqn-2-3}
K_{\sF}=K_{S/B} \otimes \mathcal{O}_{S}\Big(\sum(1-a_i)C_i\Big),
\end{equation}
where $K_{S/B}=K_S-f^*(K_B)$ is the relative canonical divisor,
the sum is taken over all components $C_i$'s in fibers of $f$,
and $a_i$ is the multiplicity of $C_i$ in fibers of $f$.
The foliation $\sF$ is reduced if and only if every possible singular fiber of $f$ is normal crossing.
In particular, if the fibration is semi-stable,
i.e., any possible singular fiber of $f$ is a reduced node curve,
and any possible smooth rational component in such a singular fiber intersects other components at least two points,
then $\sF$ is relatively minimal and $K_{\sF}=K_{S/B}$ by \eqref{eqn-2-3}.

\vspace{2mm}
An irreducible curve $C\subseteq S$ is said to be {\it $\sF$-invariant} if the inclusion
$T_{\sF}|_{C} \hookrightarrow T_X|_{C}$ factors through $T_C$.
By a curve $C\subseteq S$ we mean a reduced and compact algebraic curve.
So it might be singular and reducible.
Suppose that $C$ is not $\sF$-invariant, or more precisely every irreducible component of $C$ is not $\sF$-invariant.
Then one defines the tangency of $\sF$ to $C$ as follows.
Let $p\in C$ be any point.
Around $p$, let $\{f=0\}$ be a local equation of $C$,
and $v$ be a local holomorphic vector field defining $\sF$ around $p$.
Then the tangency of $\sF$ to $C$ at $p$ is defined to be
$$\tang(\sF,C,p)=I_p\big(f,v(f)\big):=\dim_{\mathbb{C}} \frac{\mathcal{O}_{S,p}}{\langle f, v(f) \rangle}.$$
As $C$ is not $\sF$-invariant, $\tang(\sF,C,p)<+\infty$ and  $\tang(\sF,C,p)=0$ except for finitely many points.
Hence one defines the tangency of $\sF$ to $C$.
$$\tang(\sF,C) = \sum_{p\in C} \tang(\sF,C,p).$$
\begin{proposition}[{\cite[Proposition\,2.2]{bru-04}}]
	Let $C$ be a curve on $S$ which is not $\sF$-invariant. Then
	$$\tang(\sF,C)=K_{\sF}C+C^2 \geq 0.$$
\end{proposition}

We now suppose that $C$ is $\sF$-invariant, or more precisely every irreducible component of $C$ is $\sF$-invariant.
Given any point $p\in C$, let $\{f=0\}$ be a local equation of $C$,
and $\omega$ be a local holomorphic one-form defining $\sF$ around $p$.
Because $C$ is $\sF$-invariant, we may write
$$g\omega=hdf+f\eta,$$
for some holomorphic one-form $\eta$ and holomorphic functions $g,h$ around $p$,
such that $h$ and $f$ are coprime.
We define
$$\begin{aligned}
Z(\sF,C,p)&\,=\text{vanishing order of $~\frac{h}{g}\Big|_C$ at $p$},\\
\cs(\sF,C,p)&\,=\text{residue of $~-\frac{\eta}{h}\Big|_C$ at $p$}.
\end{aligned}$$
By definition, both $Z(\sF,C,p)$ and $\cs(\sF,C,p)$ are zero if $p$ is not a singular point of $\sF$.
If $\sF$ is reduced, then $Z(\sF,C,p)\geq 0$ for any $p\in C$ \cite{bru-97}.
Let
$$\begin{aligned}
Z(\sF,C) &\,= \sum_{p\in C} Z(\sF,C,p)=\sum_{p\in C\,\cap\, \text{Sing}(\sF)} Z(\sF,C,p),\\
\cs(\sF,C) &\,= \sum_{p\in C} \cs(\sF,C,p)=\sum_{p\in C\,\cap\, \text{Sing}(\sF)} \cs(\sF,C,p).
\end{aligned}$$
\begin{proposition}[{\cite[Proposition\,2.2]{bru-04}}]\label{prop-2-2}
	Let $C$ be a curve on $S$ which is $\sF$-invariant. Then
	$$\begin{aligned}
	Z(\sF,C)&\,=K_{\sF}C+\chi(C), \quad \text{where~}\chi(C)=-K_SC-C^2;\\
	\cs(\sF,C) &\,=C^2.
	\end{aligned}$$
	In particular, if $\sF$ is reduced and $C$ is an $\sF$-invariant curve, then
	\begin{equation*} 
	0\leq K_{\sF}C+\chi(C)=K_{\sF}C-K_SC-C^2=N_{\sF}C-C^2.
	\end{equation*}
\end{proposition}

\subsection{The birational invariants and the canonical map of a foliated surface}\label{sec-canonical-map}
In this subsection, we review the birational invariants and the canonical map for a foliated surface $(S,\sF)$.
\begin{definition}\label{def-birational}
	Let $(S,\sF)$ be a foliated surface, and let $(S',\sF')$ be any reduced model, i.e., $(S',\sF')$ is birational to $(S,\sF)$ and $\sF'$ is reduced.
	\begin{enumerate}[(i)]
		\item For any integer $n\geq 1$, the $n$-th pluri-genus $p_n(\sF)$ of $(S,\sF)$ is defined to be
		$$p_n(\sF):=\dim H^0(S',nK_{\sF'}).$$
		When $n=1$, it is the geometric genus, and also denoted by $p_g(\sF)=p_1(\sF)$. \vspace{2mm}
		
		\item The volume $\vol(\sF)$ and the Kodaira dimension $\kod(\sF)$ are defined as follows.
		$$\left\{\begin{aligned}
		\vol({\sF})&\,=\limsup_{n\to +\infty} \frac{\,p_n(\sF)}{n^2/2}=\limsup_{n\to +\infty} \frac{\,\dim H^0(S', nK_{\sF'})}{n^2/2};\\[3pt]
		\kod(\sF)&\,=\left\{\begin{aligned}
		&-\infty, &\quad&\text{if~$p_n(\sF)=0$ for any $n\geq 1$},\\
		&\limsup_{n\to +\infty} \frac{\log p_n(\sF)}{\log n}, &&\text{otherwise}.
		\end{aligned}\right.
		\end{aligned}\right.$$
	\end{enumerate}
\end{definition}
By Seidenberg's \autoref{thm-seidenberg},
every foliated surface admits a reduced model by a sequence of blowing-ups.
Moreover, the above definitions are independent on the choice of the reduced model,
and they are all birational invariants for a foliated surface $(S,\sF)$.
Indeed, suppose that there are two reduced models $(S',\sF')$ and $(S'',\sF'')$.
Since $(S',\sF')$ is birational to $(S'',\sF'')$, there is a foliated surface $(\wt S,\wt \sF)$ with two birational morphisms $\rho':\,\wt S \to S'$ 
and $\rho'':\,\wt S \to S''$ such that
\begin{enumerate}[(i)]
	\item the foliated surface $(\wt S, \wt \sF)$ is reduced;
	\item we have the isomorphism $(\rho')^*(\sF')\cong \wt \sF  \cong (\rho'')^*(\sF'')$.
\end{enumerate}
$$\xymatrix{  &\wt S \ar[dl]_-{\rho'} \ar[dr]^-{\rho''}&\\
	S' && S''}$$
Hence
$$K_{\wt\sF}=(\rho')^*(K_{\sF'})+\mathcal{E}'=(\rho'')^*(K_{\sF''})+\mathcal{E}'',$$
where $\mathcal{E}'$ (resp. $\mathcal{E}''$) is effective supported over curves contracted by $\rho'$ (resp. $\rho''$).
In particular,
$$\dim H^0(S',nK_{\sF'})=\dim H^0(\wt S,nK_{\wt \sF})=\dim H^0(S'',nK_{\sF''}).$$

\begin{remark}
	Historically, one defined the pluri-genera using the divisor $K_{\sF}$ on $S$ directly in some literatures, i.e., defined them to be $\dim H^0(S,nK_{\sF})$.
	However, it turns out that such definitions behave not well in birational geometry;
	they might change under the birational morphisms between foliated surfaces.
	This is one of the reasons for us to define the pluri-genera (as well as the volume and the Kodaira dimension) over a reduced model.
\end{remark}

Let $(S,\sF)$ be a foliated surface, and let $(S',\sF')$ be any reduced model.
The foliated surface $(S,\sF)$ is of general type if the $\sF$ is of general type, i.e., the Kodaira dimension $\kod(\sF)=2$.
Our main purpose is to investigate the canonical map $\varphi=\varphi_{|K_{\sF'}|}$ defined by the linear system $|K_{\sF'}|$
for a foliation $\sF$ of general type
if the geometric genus $p_g(\sF)=h^0(S',K_{\sF'})\geq 2$.
	\begin{equation*}
	\varphi:\,S' \dashrightarrow \Sigma \subseteq \mathbb{P}^{p_g(\sF)-1}.
\end{equation*}
Composed with the birational map $S \dashrightarrow S'$, one may also view the canonical map $\varphi$ as a rational map mapping the surface $S$ to the projective space: 
\begin{equation}\label{eqn-2-21}
	\varphi:\,S \dashrightarrow \Sigma \subseteq \mathbb{P}^{p_g(\sF)-1}.
\end{equation}
It is easy to see that the canonical map $\varphi$ is independent on the choice of the reduced model $(S',\sF')$.
Equivalently, to study the birational geometry of the canonical map $\varphi$, we may assume that the foliated surface $(S,\sF)$ is already reduced,
in which case, the canonical map $\varphi=\varphi_{|K_{\sF}|}$ is just the rational map defined by the complete linear system $|K_{\sF}|$.
By contracting possible $\sF$-exceptional curves based on \cite[Proposition\,5.1]{bru-04},
we may even assume that
$\sF$ is relatively minimal as well.
Let
\begin{equation}\label{eqn-2-22}
	|K_{\sF}|=|M|+Z,
\end{equation}
be the decomposition of the complete linear system $|K_{\sF}|$, where $M$ (resp. $Z$) is the moving part (resp. fixed part).
Let $\rho:\,Y\to \Sigma$ be the minimal desingularization of $\Sigma$,
where $\Sigma=\varphi(S)$ is the image of $S$ under the canonical map as in \eqref{eqn-2-21}.
By a sequence of blowing-ups $\sigma:\,\wt{S} \to S$ centered on the base points of $|M|$, we obtain a well-defined morphism $\phi:\, \wt{S} \to Y$ with the following diagram.
$$\xymatrix{\wt{S} \ar[d]_-{\phi} \ar[rrr]^-{\sigma}  &&& S \ar@{-->}[d]^-{\varphi=\varphi_{|K_{\sF}|}}\\
	Y \ar[rrr]^-{\rho}_-{\text{desingularization}} &&&\Sigma\, \ar@{^(->}[r] & \bbp^{p_g(\sF)-1}}$$
	
\begin{lemma}\label{lem-2-21}
	Let $\tilde\varphi=\varphi\circ\sigma=\rho\circ\phi:\,\wt{S} \to \Sigma \subseteq \bbp^{p_g(\sF)-1}$
	be as above. Then $\tilde\varphi=\varphi_{|K_{\wt\sF}|}$ is just the map defined by the complete linear system $|K_{\wt\sF}|$, where $\wt{\sF}=\sigma^*\sF$ is the induced reduced foliation on $\wt{S}$.
\end{lemma}
\begin{proof}
	Since $\sigma$ is composed of certain blowing-ups,
	it follows that the map $\tilde\varphi$ is the same as the rational map $\varphi_{|\sigma^*K_{\sF}|}$ defined by the complete linear system $|\sigma^*K_{\sF}|$. On the other hand, since $\sF$ is reduced,
	$$K_{\wt{\sF}}=\sigma^*K_{\sF} +\mathcal{E},$$
	where $\mathcal{E}$ is some effective divisor supported on the exceptional curves of the birational morphism $\sigma$. In particular,
	$$|K_{\wt{\sF}}|=\sigma^*|K_{\sF}|+\mathcal{E}.$$
	It follows that $\tilde\varphi=\varphi_{|\sigma^*K_{\sF}|}=\varphi_{|K_{\wt{\sF}}|}$ as required.
\end{proof}

\section{The canonical map is generically finite}\label{sec-gen-finite}
In this section, we consider the case when the canonical map $$\varphi:\,S \dashrightarrow \Sigma \subseteq \bbp^{p_g(\sF)-1},$$
is generically finite, i.e., the image $\Sigma$ is of dimension two.
In particular, $p_g(\sF)\geq 3$.
In \autoref{sec-non-degenerate}, we will prove some lower bounds
on the degree of a non-degenerate surface in a projective space.
Then we prove \autoref{thm-1-1} and \autoref{thm-1-2} in \autoref{sec-gen-finte-proof}.
Finally in \autoref{sec-exam-3}, we will construct examples reaching the equalities in both \autoref{thm-1-1} and \autoref{thm-1-2}.

\subsection{The degree of surfaces in projective spaces}\label{sec-non-degenerate}
In this subsection, we study the degree of a non-degenerate surface in projective spaces.
Here "non-degenerate" means that such a surface is not contained in any hyperplane.
This is a classical topic, which goes back to the 19-th century, cf. \cite{noether}.
The next proposition can be found in \cite[Lemme\,1.4 and Remarque\,1.5]{bea-79}, see also \cite[\S\,4.3]{gh-78}.
We include a proof here for readers' convenience.
\begin{proposition}\label{prop-3-1}
	Let $\Sigma\subseteq \bbp^N$ be a non-degenerate surface.
	\begin{enumerate}[(i)]
		\item The degree of $\Sigma$ satisfies that $\deg\Sigma \geq N-1$.
		Moreover, if $\deg\Sigma=N-1$, then $\Sigma$ is birational to a rational surface.
		\item Suppose that $\Sigma$ is not ruled. Then $\deg\Sigma\geq 2(N-1)$.
		Moreover, if $\deg\Sigma = 2(N-1)$, then $\Sigma$ is birational to a K3 surface.
	\end{enumerate}
\end{proposition}
\begin{proof}
	Let $\rho:\,Y \to \Sigma$ be the minimal desingulariztion of $\Sigma$ and $M=\rho^*(H)$,
	where $H$ is a hyperplane section.
	Then $|M|$ is base-point-free and $h^0(Y,M)\geq N+1$.
	Let $C\in |M|$ be a general element, and consider the exact sequence:
	\begin{equation}\label{eqn-31-7}
	0 \lra \mathcal{O}_{Y} \lra \mathcal{O}_{Y}(C) \lra \mathcal{O}_{C}(C) \lra 0.
	\end{equation}
	It follows that
	\begin{equation}\label{eqn-31-2}
		h^0\big(C, \mathcal{O}_{C}(C)\big)\geq h^0\big(Y, \mathcal{O}_{Y}(C)\big)-h^0(Y, \mathcal{O}_{Y}) \geq (N+1)-1=N.
	\end{equation}
	It follows that
	$$\deg \Sigma = C^2 =\deg\big(\mathcal{O}_{C}(C)\big) \geq h^0\big(C, \mathcal{O}_{C}(C)\big)-1\geq N-1.$$
	Moreover, if the equality holds, then $C$ is a smooth rational curve.
	But $C\in |M|$ is the pulling-back of a hyperplane section.
	Hence $Y$ is a rational surface.
	This proves (i). In fact, one can find a more detailed description in \cite[Chapter\,4.3]{gh-78}).
	
	It remains to prove (ii), and we assume that $Y$ is not ruled in the following.
	In particular, $K_Y$ is pseudo-effective.
	Hence
	$$\deg \mathcal{O}_{C}(K_{C}) -\deg \mathcal{O}_{C}(C)=\deg \mathcal{O}_{C}(K_{Y})=K_{Y} \cdot C \geq 0.$$
	According to the Clifford theorem, 
	\begin{equation}\label{eqn-31-3}
		\deg \mathcal{O}_{C}(C) \geq 2h^0\big(C, \mathcal{O}_{C}(C)\big)-2.
	\end{equation}
	Combining it with \eqref{eqn-31-2}, one obtains that
	$$\deg \Sigma = C^2 =\deg\big(\mathcal{O}_{C}(C)\big) \geq 2(N-1).$$
	Moreover, if the equality holds in the above inequality,
	then both \eqref{eqn-31-2} and \eqref{eqn-31-3} are equalities.
	Note that the curve $C$ is not hyperelliptic, since the map $\varphi_{|M|}=\varphi_{|C|}$
	defined by the complete linear system $|M|=|C|$ is birational.
	In particular, from the Clifford theorem it follows that $\mathcal{O}_{C}(C) \cong \mathcal{O}_{C}(K_{C})$,
	equivalently
	$\mathcal{O}_{C}(K_{Y}) \cong \mathcal{O}_{C}.$
	In particular, $K_{Y} \cdot C=0$ and the Kodaira dimension $\kappa(Y)=0$.
	Let $$\tau:\,Y \lra Y_0,$$
	be the contraction of $Y$ to its minimal model, and $C_0$ be the image of $C$.
	Suppose
	$$K_{Y}=\tau^*K_{Y_0}+\sum_{i=1}^{n}\mathcal{E}_i,\qquad C=\tau^*C_0+\sum_{i=1}^{n}a_i\mathcal{E}_i.$$
	Since $K_{Y} \cdot C=0$, it follows that $a_i=0$ for all $1\leq i \leq n$,
	namely $C = \tau^*C_0$.
	In particular, any possible $(-1)$-curve $E$ on $Y$ satisfying $C\cdot E=0$.		
	Since $\rho:\, Y \to \Sigma$ is the minimal desingularization,
	and $C$ is the pulling-back of some hyperplane section,
	it follows that there is no $(-1)$-curve on $Y$, i.e., $Y=Y_0$ is minimal.
	Consider now the exact sequence
	$$0 \lra \mathcal{O}_{Y}(K_{Y}-C) \lra \mathcal{O}_{Y}(K_{Y}) \lra \mathcal{O}_{C}(K_{Y}) \lra 0.$$
	Since the Kodaira dimension  $\kappa(Y)=0$,
	it follows that $K_{Y}$ is numerically trivial, and hence
	$$H^0(\mathcal{O}_{Y}(K_{Y}-C))=H^1(\mathcal{O}_{Y}(K_{Y}-C))=0.$$
	Therefore, 
	\begin{equation}\label{eqn-31-5}
		p_g(Y)=h^0(\mathcal{O}_{Y}(K_{Y}))=h^0(\mathcal{O}_{C}(K_{Y}))=h^0(\mathcal{O}_{C})=1.
	\end{equation}
    Taking the induced long exact sequence of the short exact sequence \eqref{eqn-31-7}, one obtains
	$$0 \lra H^0(\mathcal{O}_{Y}) \lra H^0(\mathcal{O}_{Y}(C)) \lra H^0(\mathcal{O}_{C}(C))
	\lra H^1(\mathcal{O}_{Y}) \lra H^1(\mathcal{O}_{Y}(C))=0.$$
	The equality $H^1(\mathcal{O}_{Y}(C))=0$ follows the Mumford vanishing theorem,
	since $C=K_{Y}+(C-K_{Y})$ and $C-K_{Y}$ is nef and big due to the numerical triviality of $K_{Y}$.
	By the equality in \eqref{eqn-31-2},
	$$q(Y)=h^1(\mathcal{O}_{Y})=h^0(\mathcal{O}_{C}(C))+h^0(\mathcal{O}_{Y})-h^0(\mathcal{O}_{Y}(C))=0.$$
	Combining this with \eqref{eqn-31-5}, one proves that $Y$ is a K3 surface.
\end{proof}

When the surface $\Sigma \subseteq \bbp^N$ is of general type, we improve the lower bound on the degree as follows.
In the following, by an $(a,b)$-surface $Y_0$ we mean a minimal surface of general type
whose volume $\vol(Y_0)=K_{Y_0}^2=a$ and geometric genus $p_g(Y_0)=b$.
\begin{proposition}\label{prop-3-2}
	Let $\Sigma\subseteq \bbp^N$ be a non-degenerate surface of general type.
\begin{enumerate}[(i)]
	\item Suppose that $\Sigma$ is not birational to a $(1,0)$-surface.
	Then
	\begin{equation}\label{eqn-31-6}
	\deg\Sigma \geq 2(N-1)+\frac{\sqrt{8N-15}-3}{2}.
	\end{equation}
	Moreover, if the equality holds, then $\Sigma$ is birational to a $(1,2)$-surface.
	
	\item Suppose that $\Sigma$ is birational to a $(1,0)$-surface.
	Then
	\begin{equation}\label{eqn-31-1}
		\deg\Sigma \geq 2(N-1)+\frac{\sqrt{8N-23}-7}{2}.
	\end{equation}
\end{enumerate}
\end{proposition}
Let's do some preparations before going to the proof of the above proposition.
\begin{lemma}\label{lem-3-1}
	Let $Y$ be a smooth projective surface, and $L,M$ be two divisors on $Y$.
	Let $\varphi_{|L|}$ and $\varphi_{|M|}$ be the rational map defined by the complete linear systems $|L|$ and $|M|$ respectively.
	Suppose that $\varphi_{|L|}$ is generically finite and $\varphi_{|M|}$
	is birational.
	Assume that $|M|$ has no fixed part, and let $C\in |M|$ be a general curve.
	Then $\varphi_{|L_C|}$ defines a birational map, where $L_C=L|_C$ is the restriction of $L$ to $C$.
\end{lemma}
\begin{proof}
	Let $L_0$ be the moving part of $L$. Then $\varphi_{|L|}=\varphi_{|L_0|}$ and $L_0|_C\subseteq L|_C=L_C$.
	By replacing $L$ by its moving part $L_0$, we may assume that $L$ has no fixed part.	
	Consider the following exact sequence
	$$0 \lra \calo_Y(-M) \lra \calo_Y \lra \calo_C \lra 0.$$
	Tensoring with $\calo_Y(L)$, one obtains
	$$0 \lra \calo_Y(L-M) \lra \calo_Y(L) \lra \calo_C(L_C) \lra 0.$$
	Let $\Lambda\subseteq |L_C|$ be the linear subsystem corresponding to the image of
	the induced map by taking the long exact sequence of the above short exact sequence.
	$$H^0(Y,L) \lra H^0(C,L_C).$$
	To show $\varphi_{|L_C|}$ is birational, it suffices to prove $\varphi_{\Lambda}$ is birational,
	where $\varphi_{\Lambda}$ is defined by the linear subsystem $\Lambda\subseteq |L_C|$.
	Note that $\varphi_\Lambda=\big(\varphi_{|L|}\big)\big|_C$ is the restriction of the map $\varphi_{|L|}$ on $C$.
	Therefore,
	that $\varphi_\Lambda$ is birational is equivalent to that $\varphi_{|L|}$ maps $C$ birationally to its image
	$\varphi_{|L|}(C)$.
		
	Since the map $\varphi_{|L|}$ is generically finite and $\varphi_{|M|}$ is birational,
	a general element $C\in |M|$ is not contracted by $\varphi_{|L|}$ and does not contained in the ramified divisor of $\varphi_{|L|}$.
	Let $p\in C$ be a general point, and $$O_p=\big\{q\in Y~|~\varphi_{|L|}(q)=\varphi_{|L|}(p)\big\} \subseteq Y.$$
	Since $\varphi_{|M|}$ is birational, for a general element $C\in |M|$, we may assume that
	$$C\cap O_p=\{p\}, \qquad \text{equivalently,}\qquad \varphi_{|L|}^{-1}\big(\varphi_{|L|}(p)\big) \cap C=\{p\}.$$
	It follows that $\varphi_{|L|}$ maps $C$ birationally to its image
	$\varphi_{|L|}(C)$ as required.
\end{proof}

\begin{lemma}\label{lem-3-2}
	Let $Y$ be a smooth projective surface of general type. Then
	\begin{enumerate}[(i)]
		\item The bi-canonical map $\varphi_{|2K_Y|}$ is generically finite unless
		$\vol(Y)=1$ and $\chi(\calo_Y)=1$.
		\item The triple-canonical map $\varphi_{|3K_Y|}$ is always generically finite.
	\end{enumerate}
\end{lemma}
\begin{proof}
	The statements are birationally invariant, and hence we may assume that $Y$ is minimal.
	The first statement is proved in \cite{xiao-85}, see also \cite[Theorem\,VII.7.6]{bhpv}.
	The proof of the second statement was kindly provided to the author by Professor Yifan Chen.
	The argument is given by contradiction.
	Let
	$$|3K_Y|=|M|+Z,$$
	where $M$ and $Z$ are respectively the moving and fixed parts of $|3K_Y|$.
	Suppose that the triple-canonical map $\varphi_{|3K_Y|}$ is composed with a pencil $\Lambda$.
	Let $A\in \Lambda$ be a general element, and suppose $M \sim_{num}\, aA$.
	Then	
	$$a+1\geq h^0(Y,M)=h^0(Y,3K_Y)=3K_Y^2+\chi(\mathcal{O}_Y).$$
	Hence
	$$3K_Y^2=K_Y(M+Z) \geq K_YM=aK_YA\geq \big(3K_Y^2+\chi(\mathcal{O}_Y)-1\big)K_YA\geq (3K_Y^2)K_YA.$$
	It follows that $K_YA=1$, and the inequalities above are all equalities.
	In particular, $\chi(\mathcal{O}_Y)=1$, $K_Y^2=1$ and $a=3$.
	By the Hodge index theorem,
	$A^2\leq 1$.
	Since $Y$ is of general type,
	$$2\leq 2g(A)-2 \leq (K_Y+A)A\leq 2,$$
	where $g(A)$ is the geometric genus of $A$.
	Hence $K_Y^2=A^2=1$ and $K_Y \sim_{num} A$.
	In particular, $|A|$ is a rational pencil with $h^0(Y,A)=2$.
	Thus $2A=K_Y+(2A-K_Y)$ and $2A-K_Y \sim_{num} K_Y$ is nef and big.
	By the Riemann-Roch theorem and the vanishing theorem, it follows that
	$$h^0(Y,2A)=\chi(\mathcal{O}_Y)+\frac{2A(2A-K_Y)}{2}=2.$$
	On the other hand, since $h^0(Y,A)=2$,
	it follows that
	$$h^0(Y,2A) \geq 2h^0(Y,A)-1=3.$$
	This gives a contradiction.
	Hence $\varphi_{|3K_Y|}$ is generically finite as required.
\end{proof}

\begin{proof}[Proof of \autoref{prop-3-2}]
	(i).
	We use the notations introduced in the proof of \autoref{prop-3-1}.
	Using the Riemann-Roch theorem on the curve $C$, one gets
	$$h^0(C,\mathcal{O}_C(C))=h^1(C,\mathcal{O}_C(C))+C^2+1-g(C)=h^0(C,\mathcal{O}_C(K_Y))+\frac{C^2-K_YC}{2},$$
	where $g(C)$ is the geometric genus of $C$ satisfying $2g(C)-2=C^2+K_YC$.
	Combining this with \eqref{eqn-31-2},
	$$C^2=\deg\Sigma\geq 2N+K_YC-2h^0(C,\mathcal{O}_C(K_Y)).$$
	Note that
	$$2(N-1)+\frac{\sqrt{8N-15}-3}{2}=\left(\frac{\sqrt{8N-15}+1}{2}\right)^2.$$
	This implies that the inequality \eqref{eqn-31-6} is equivalent to
	$$C^2=\deg\Sigma \geq 2N-4+\sqrt{\deg\Sigma}=2N-4+\sqrt{C^2}.$$
	Therefore, it is enough to prove
	\begin{equation}\label{eqn-31-8-0}
		K_YC\geq 2h^0(C,\mathcal{O}_C(K_Y))-4+\sqrt{C^2}.
	\end{equation}
	Let $D=(K_Y)|_C$, $d=\deg(D)$ and $h^0=h^0(C,D)$. The inequality \eqref{eqn-31-8-0} is also equivalent to
	\begin{equation}\label{eqn-31-8}
		\big(d-2h^0+4\big)^2+d\geq 2g(C)-2.
	\end{equation}
	We divide into two cases to complete the proof depending on whether $h^0\leq 2$ or not,
	which will be given in \autoref{lem-step-1} and \autoref{lem-step-2} respectively.

	\begin{lemma}\label{lem-step-1}
		Suppose $h^0\leq 2$. Then \eqref{eqn-31-8} holds.
		Moreover, if the equality holds, then $Y$ is a $(1,2)$-surface,
		i.e., $Y$ is minimal of general type with $K_{Y}^2=1$ and geometric genus $p_g(Y)=2$.
	\end{lemma}
	\begin{proof}[Proof of \autoref{lem-step-1}]
		Let $\tau:\,Y \to Y_0$
		be the contraction of $Y$ to its minimal model, and $C_0$ be the image of $C$.
		Suppose
		\begin{equation}\label{eqn-31-10}
			K_{Y}=\tau^*K_{Y_0}+\sum_{i=1}^{n}\mathcal{E}_i,\qquad C=\tau^*C_0+\sum_{i=1}^{n}a_i\mathcal{E}_i.
		\end{equation}
		By the Hodge index theorem,
		$$d^2=\big(K_YC\big)^2\geq \big(\tau^*(K_{Y_0})\cdot C\big)^2 \geq \tau^*(K_{Y_0})^2\cdot C^2 \geq C^2=2g(C)-2-d.$$
		Since $h^0\leq 2$, it follows that
		$$\big(d-2h^0+4\big)^2+d \geq d^2+d \geq 2g(C)-2.$$
		This proves the inequality \eqref{eqn-31-8}.
		Moreover, if the equality holds, then
		$$h^0=h^0(C,D)=2,\qquad K_{Y_0}^2=\tau^*(K_{Y_0})^2=1,\qquad K_YC=\tau^*(K_{Y_0})\cdot C,$$
		and that $C \sim_{num} r\,\tau^*(K_{Y_0})$ with $r=\sqrt{C^2}$.
		Combining this with \eqref{eqn-31-10}, it follows that $a_i=0$ for any $1\leq i\leq n$.
		In particular, any possible $(-1)$-curve $E$ on $Y$ satisfying $C\cdot E=0$.		
		Since $\rho:\, Y \to \Sigma$ is the minimal desingularization, and $C$ is the pulling-back of some hyperplane section,
		it follows that there is no $(-1)$-curve on $Y$, i.e., $Y=Y_0$ is minimal.
		Moreover, since $K_Y^2=K_{Y_0}^2=1$,
		it follows that the irregularity 
		$q(Y)=0$ by the Noether type inequality for irregular surfaces of general type.
		Consider the following exact sequence
		\begin{equation}\label{eqn-31-11}
			0 \lra \calo_Y(K_Y-C) \lra \calo_Y(K_Y) \lra \calo_C\big((K_Y)|_C\big)=\calo_C(D) \lra 0.
		\end{equation}
		Note that $r>1$;
		otherwise, $K_YC=C^2=1$ and hence
		$$N\leq h^0(C,\mathcal{O}_C(C))=h^0(C,\mathcal{O}_C(K_Y))=h^0=2,$$
		which is a contradiction, since $Y$ is of general type.
		It follows that $C-K_Y \sim_{num} (r-1)K_Y$ is nef and big, and hence
		$$H^0(Y,K_Y-C)=0,\qquad H^1(Y,K_Y-C)=H^1(Y,C)^{\vee}=H^1\big(Y,K_Y+(C-K_Y)\big)^{\vee}=0.$$
		By taking the induced long exact sequence of \eqref{eqn-31-11},
		one proves immediately that
		$$p_g(Y)=h^0(Y,K_Y)=h^0(C,D)=h^0=2.$$
		This shows that $Y$ is a $(1,2)$-surface.
	\end{proof}
	
	\begin{lemma}\label{lem-step-2}
		Suppose that $h^0\geq 3$.
		Then
		\begin{equation}\label{eqn-31-9}
			\big(d-2h^0+4\big)^2+d> 2g(C)-2.
		\end{equation}
	\end{lemma}
	\begin{proof}[Proof of \autoref{lem-step-2}]
		Let $\varphi_{|2D|}$ be the map defined by the complete linear system $|2D|$ on the curve $C$.
		Let $Y_0$ be the minimal model of $Y$.
		Suppose that $K_{Y_0}^2=1$ and $\chi(\mathcal{O}_{Y_0})=1$.
		Then the irregularity $q(Y)=0$; otherwise by the Noether inequality for irregular surface $K_{Y_0}^2 \geq 2\chi(\mathcal{O}_{Y_0})$, which is a contradiction.
		Hence $p_g(Y_0)=q(Y_0)=0$, i.e., $Y_0$ is a $(1,0)$-surface.
		Therefore, if $\Sigma$ is not birational to a $(1,0)$-surface,
		then $\varphi_{|2K_Y|}$ is generically finite by \autoref{lem-3-2}\,(i).
		Hence $\varphi_{|2D|}$ is birational by \autoref{lem-3-1}.
		
		Suppose first that $\varphi_{|D|}$ is birational,
		where $\varphi_{|D|}$ is the map defined by the complete linear system $|D|$ on the curve $C$.
		According to Castelnuovo's bound \cite[\S\,III.2]{acgh},
		$$g(C) \leq \frac{m(m-1)}{2}(h^0-2)+m\epsilon,$$
		where
		$$m=\left[\frac{d-1}{h^0-2}\right],\qquad d-1=m(h^0-2)+\epsilon.$$
		Note that $D$ is a special divisor on $C$. It follows that $m\geq 2$, and hence
		$$\begin{aligned}
		&\quad\,\big(d-2h^0+4\big)^2+d=\big((m-2)(h^0-2)+\epsilon+1\big)^2+m(h^0-2)+\epsilon+1\\
		&=\big((m-2)^2(h^0-2)+3m-4\big)(h^0-2)+\epsilon^2
		+\big(2(m-2)(h^0-2)+3\big)\epsilon+2\\
		&\geq \big((m-2)^2+3m-4\big)(h^0-2)+\epsilon^2
		+\big(2(m-2)+3\big)\epsilon+2,\qquad\text{as~}h^0\geq3,\\
		&=m(m-1)(h^0-2)+\epsilon^2+(2m-1)\epsilon+2,\\
		&\geq m(m-1)(h^0-2)+2m\epsilon+2\geq 2g(C)+2>2g(C)-2.
		\end{aligned}$$
		
		Suppose next that $\varphi_{|D|}$ is not birational.
		Since $\varphi_{|2D|}$ is birational as proved above, it follows that
		\begin{equation}\label{eqn-31-12}
			h^0_2:=h^0(C,2D)>2h^0(C,D)-1=2h^0-1,\quad\text{i.e.,}\quad h^0_2\geq 2h^0.
		\end{equation}
		Similar as above, by Castelnuovo's bound \cite[\S\,III.2]{acgh},
		$$g(C) \leq \frac{m_2(m_2-1)}{2}(h_2^0-2)+m_2\epsilon_2,$$
		where
		$$m_2=\left[\frac{d_2-1}{h^0_2-2}\right],~\qquad d_2-1=m_2(h^0_2-2)+\epsilon_2~\text{~with~}~d_2=\deg(2D)=2d.$$
		Note that $h^0_2\geq 6$ by \eqref{eqn-31-12}, since $h^0\geq 3$ by assumption.
		If $m_2\geq 2$, then
		$$d-2h^0+4 \geq d-h^0_2+4=\Big(\frac{m_2}{2}-1\Big)(h^0_2-2)+\frac{\epsilon_2+5}{2} > 0,$$
		and hence
		$$\begin{aligned}
		&\quad\,\big(d-2h^0+4\big)^2+d\geq \big(d-h^0_2+4\big)^2+d\\
		&=\left(\Big(\frac{m_2}{2}-1\Big)(h^0_2-2)+\frac{\epsilon_2+5}{2}\right)^2+\frac{m_2(h^0_2-2)+\epsilon_2+1}{2}\\
		&=\left(\Big(\frac{m_2}{2}-1\Big)^2(h^0_2-2)+3m_2-5\right)(h^0_2-2)\\
		&\quad\, +\frac{\epsilon_2^2+\big(2(m_2-2)(h^0_2-2)+12\big)\epsilon_2+27}{4}\\
		&\geq \left\{\begin{aligned}
		&4\left(m_2^2-m_2-1\right)+\frac{\epsilon_2^2+(8m_2-4)\epsilon_2+27}{4},&\quad&\text{if~}h^0_2=6,\\
		&\frac{5m_2^2-8m_2}{4}(h^0_2-2)
		+\frac{\epsilon_2^2+\big(10m_2-8\big)\epsilon_2+27}{4}, &&\text{if~}h^0_2\geq 7,\\
		\end{aligned}\right.\\
		&> m_2(m_2-1)(h^0_2-2)+2m_2\epsilon_2\geq 2g(C)>2g(C)-2.
		\end{aligned}$$
		If $m_2=1$, then $2D$ must be non-special, i.e., $h^1(C,2D)=0$;
		otherwise $m_2\geq 2$ by Clifford's theorem.
		Hence
		$$h^0_2=d_2+1-g(C)=2d+1-g(C),\quad\Longrightarrow\quad 2d-1=d_2-1=h^0_2-2+g(C).$$
		In particular, $h^0_2-2>\epsilon_2=g(C)$ according to the definition of $\epsilon_2$. Hence
		$$K_YC=d=\frac{h^0_2+g(C)-1}{2}\geq g(C)+1=\frac{K_YC+C^2+4}{2},
		\text{\quad i.e., \quad} K_YC\geq C^2+4.$$
		On the other hand, since the complete linear system $|\mathcal{O}_C(C)|$ defines a birational map
		and $\varphi_{|D|}$ is not birational, it follows that
		$$h^0(C,\mathcal{O}_C(K_Y-C))=0, ~\text{~i.e.,~}~
		h^1(C,\mathcal{O}_C(2C))=0~\text{~by Serre's duality.}$$
		Hence
		$$h^0(C,\mathcal{O}_C(2C))=2C^2+1-g(C)=\frac{3C^2-K_YC}{2}.$$
		Since $|\mathcal{O}_C(C)|$ defines a birational map, by Clifford's plus theorem (cf. \cite[\S\,III.2]{acgh}),
		$$h^0(C,\mathcal{O}_C(2C))\geq 3h^0(C,\mathcal{O}_C(C))-3.$$
		Hence
		$$C^2\geq 2h^0(C,\mathcal{O}_C(C))-2+\frac{K_YC}{3}\geq 2h^0(C,\mathcal{O}_C(C))-4+\frac{C^2+10}{3}>2h^0(C,\mathcal{O}_C(C))-4+\sqrt{C^2}.$$
		Using the Riemann-Roch theorem, one sees that the above inequality is equivalent to \eqref{eqn-31-9}.
	\end{proof}

    (ii) First, similar to the proof of (i), one shows that
    it suffices to prove
    \begin{equation}\label{eqn-31-19}
    	\big(d-2h^0+6\big)^2+d\geq 2g(C)-2.
    \end{equation}
    By a similar argument in \autoref{lem-step-1}, one shows that
    \eqref{eqn-31-19} holds if $h^0\leq 3$.
    Hence we may assume that $h^0\geq 4$.
    By \autoref{lem-3-1} together with \autoref{lem-3-2}\,(ii),
    the map $\varphi_{|3D|}$ is birational,
	where $\varphi_{|3D|}$ is the map defined by the complete linear system $|3D|$ on the curve $C$.
	Moreover, If $\varphi_{|D|}$ or $\varphi_{|2D|}$ is birational, then
	the arguments in \autoref{lem-step-2} show that the strong inequality \eqref{eqn-31-9} holds.
	Hence we may assume that neither $\varphi_{|D|}$ nor $\varphi_{|2D|}$ are birational.
		In particular,
		\begin{equation}\label{eqn-31-13}
		h^0_3:=h^0(C,3D)>3h^0(C,D)-2=3h^0-2,\quad\text{i.e.,}\quad h^0_3\geq 3h^0-1\geq 11.
		\end{equation}
		Based on Castelnuovo's bound \cite[\S\,III.2]{acgh},
		\begin{equation}\label{eqn-31-20}
			g(C) \leq \frac{m_3(m_3-1)}{2}(h_3^0-2)+m_3\epsilon_3=m_3(d_3-1)-\frac{m_3(m_3+1)}{2}(h_3^0-2),
		\end{equation}
		where
		$$m_3=\left[\frac{d_3-1}{h^0_3-2}\right],\qquad d_3-1=m_3(h^0_3-2)+\epsilon_3~\text{~with~}~d_3=\deg(3D)=3d.$$
		By \eqref{eqn-31-13}, one has
		$$d-2h^0+6\geq d-\frac{2h^0_3-16}{3}=\frac{m_3-2}{3}\,(h^0_3-2)+\frac{\epsilon_3+13}{3}>0,\qquad\text{~if~}m_3\geq 2.$$
	    
	    Consider the first the case when $m_3\geq 4$. Then
		$$\begin{aligned}
		&\quad\,\big(d-2h^0+6\big)^2+d\geq \Big(d-\frac{2h^0_3-16}{3}\Big)^2+d,\\
		&=\left(\frac{m_3-2}{3}\,(h^0_3-2)+\frac{\epsilon_3+13}{3}\right)^2+\frac{m_3(h^0_3-2)+\epsilon_3+1}{3}\\
		&=\left(\Big(\frac{m_3-2}{3}\Big)^2(h^0_3-2)+\frac{29m_3-52}{9}\right)(h^0_3-2) +\frac{\epsilon_3^2+\big(2(m_3-2)(h^0_3-2)+29\big)\epsilon_3+172}{9}\\
		&\geq \left\{\begin{aligned}
		&9m_3^2-7m_3-16
		+\frac{\epsilon_3^2+\big(18m_3-7\big)\epsilon_3+172}{9}, &\quad&\text{if~}h_3^0=11,\\
		&\frac{100m_3^2-110m_3-120}{9}
		+\frac{\epsilon_3^2+\big(20m_3-11\big)\epsilon_3+172}{9}, &\quad&\text{if~}h_3^0=12,\\
		&\frac{11m_3^2-15m_3-8}{9}(h^0_3-2)
		+\frac{\epsilon_3^2+\big(22m_3-15\big)\epsilon_3+172}{9}, &\quad&\text{if~}h_3^0\geq 13,
		\end{aligned}\right.\\
		&\geq m_3(m_3-1)(h^0_3-2)+2m_3\epsilon_3\geq 2g(C)>2g(C)-2.
		\end{aligned}$$
		
		Consider next the case when $m_3=3$. Then by similar arguments as above,
		$$\begin{aligned}
		\big(d-2h^0+6\big)^2+d\,&\geq \frac{(h^0_3+33)(h^0_3-2)}{9} +\frac{\epsilon_3^2+\big(2(h^0_3-2)+29\big)\epsilon_3+172}{9}\\
		&\geq \frac{(h^0_3+33)(h^0_3-2)}{9}+\frac{\epsilon_3^2+47\epsilon_3+172}{9},
		\qquad\text{since $h^0_3\geq 11$ by \eqref{eqn-31-13},}\\
		&= \frac{(h^0_3+33)(h^0_3-2)+160}{9}+\frac{\epsilon_3^2+47\epsilon_3+12}{9}\\
		&\geq 6(h^0_3-2)+6\epsilon_3= 2g(C)>2g(C)-2.
		\end{aligned}$$
		
		It remains to prove \eqref{eqn-31-19} when $m_3=1$ or $2$.
		Similar to the proof in the case when $m_2=1$ above, one proves that
		$h^1(C,\mathcal{O}_C(2C))=0$, and hence
		$$\frac{3C^2-K_YC}{2}=h^0(C,\mathcal{O}_C(2C))\geq 3h^0(C,\mathcal{O}_C(C))-3.$$
		Hence
		$$C^2\geq 2h^0(C,\mathcal{O}_C(C))-6+\frac{K_YC+12}{3}=2h^0(C,\mathcal{O}_C(C))-6+\frac{d+12}{3}.$$
		Using the Riemann-Roch theorem, this is equivalent to
		\begin{equation}\label{eqn-31-23}
			K_YC=d\geq 3(h^0-1).
		\end{equation}
		If $m_3=2$, then by Castelnuovo's bound \eqref{eqn-31-20},
		$$\begin{aligned}
		&\,\quad\,\big(d-2h^0+6\big)^2+d-\big(2g(C)-2\big)\\
		&\geq\Big(\frac{d}{3}+4\Big)^2+d-\big(2t+4\epsilon_3-2\big),\qquad\text{where~}t:=h^0_3-2,\\
		&=\frac{4t^2+4(10+\epsilon_3)t+\epsilon_3^2-223\epsilon_3+1558}{81},\qquad \text{since~}d=\frac{m_3t+\epsilon_3+1}{3}=\frac{2t+\epsilon_3+1}{3},\\
		&\geq \frac{4(\epsilon_3+1)^2+4(10+\epsilon_3)(\epsilon_3+1)+\epsilon_3^2-223\epsilon_3+1558}{81},\qquad\text{since~}t=h^0_3-2\geq \epsilon_3+1,\\
		&=\frac{9\epsilon_3^2-171\epsilon_3+1602}{81}>0.
		\end{aligned}$$	
		If $m_3=1$, then similar as in the case when $m_2=1$, one shows that $h^1(C,3D)=0$ and that
		\begin{equation}\label{eqn-31-24}
			d\geq \frac{h^0_3+g(C)-1}{3}\geq \frac{2(g(C)+1)}{3}.
		\end{equation}
		Hence
		$$\begin{aligned}
		&\,\quad\,\big(d-2h^0+6\big)^2+d-\big(2g(C)-2\big)\\
		&\geq\Big(\frac{d+12}{3}\Big)^2+d-\big(2g(C)-2\big),\qquad\text{by \eqref{eqn-31-23}},\\
		&\geq \Big(\frac{2g(C)+38}{9}\Big)^2+\frac{2(g(C)+1)}{3}-\big(2g(C)-2\big),\qquad\text{by \eqref{eqn-31-24}},\\
		&=\frac{4\big(g(C)^2+11g(C)+415\big)}{81}>0.
		\end{aligned}$$
	This completes the proof of \eqref{eqn-31-19}, and hence of \eqref{eqn-31-1}.
\end{proof}

\subsection{The degree of the canonical maps and canonical foliated surfaces}\label{sec-gen-finte-proof}
In this subsection, we prove \autoref{thm-1-1} and \autoref{thm-1-2}.
\begin{proof}[Proof of \autoref{thm-1-1}]
	We may assume that $\sF$ is reduced, and hence the canonical map $\varphi=\varphi_{|K_{\sF}|}$ is defined by the complete system $|K_{\sF}|$ as in \autoref{sec-canonical-map}.
	Moreover, by a sequence of possible blowing-ups $\sigma:\,\wt{S} \to S$, one obtains a morphism 
	$$\tilde\varphi=\rho\circ\phi:\,\wt{S} \to \Sigma \hookrightarrow \bbp^N,
	\qquad \text{where $N=p_g(\sF)-1$}.$$
	By \autoref{lem-2-21}, the above map $\tilde\varphi$
	is defined by the complete linear system $|K_{\wt\sF}|$, where $\wt{\sF}=\sigma^*\sF$ is the induced reduced foliation on $\wt{S}$.
	Let
	\begin{equation}\label{eqn-3-1}
		|K_{\wt{\sF}}|=|\wt{M}|+\wt{Z},
	\end{equation}
	be the decomposition of $|K_{\wt{\sF}}|$ into moving part $|\wt{M}|$ and fixed part $\wt{Z}$.
	Then
	\begin{equation}\label{eqn-31-25}
		\vol(\sF)=\vol(\wt{\sF}) \geq \wt{M}^2 \geq \deg(\tilde\varphi)\cdot \deg(\Sigma).
	\end{equation}
	By construction, $\deg(\tilde\varphi)=\deg(\varphi)$ since $\sigma$ is birational.
	Combining with \autoref{prop-3-1}\,(i) on the lower bound on $\deg\Sigma$, one obtains that
	$$\deg(\varphi)=\deg(\tilde\varphi) \leq \frac{\vol(\sF)}{\deg(\Sigma)} \leq \frac{\vol(\sF)}{p_g(\sF)-2}.$$
	Moreover, if the equality holds, it follows in particular that $\deg(\Sigma)=p_g(\sF)-2$; namely, the image $\Sigma$ is a surface of minimal degree in $\bbp^{p_g(\sF)-1}$.
	This completes the proof.
\end{proof}

\begin{proof}[Proof of \autoref{thm-1-2}]
	Using the same notations as in the proof of \autoref{thm-1-1}.
	Note that $\deg(\tilde\varphi)=\deg(\varphi)=1$ by assumption.
	Hence our theorem follows from the inequality \eqref{eqn-31-25} together with
	\autoref{prop-3-1} and \autoref{prop-3-2}.
\end{proof}

According to the proofs of \autoref{thm-1-1} and \autoref{thm-1-2} above,
we have actually proved the following.
\begin{proposition}\label{prop-3-3}
	Let $\sF$ be a foliation of general type on a smooth projective surface $S$, such that its canonical map $\varphi$ induces a generically finite map as in \autoref{thm-1-1}. Denote by $\Sigma=\varphi(S)$ the image of $S$ under the canonical map.
	\begin{enumerate}[$(i)$]
		\item Suppose that the image surface $\Sigma$ is not ruled. Then
		\begin{equation*}
		\vol(\sF) \geq \deg(\varphi)\cdot \big(2p_g(\sF)-4\big).
		\end{equation*}
		
		\item Suppose that the image surface $\Sigma$ is also of general type. Then
		\begin{equation*}
		\vol(\sF) \geq  \deg(\varphi)\cdot\Big(2p_g(\sF)-4+\frac{\sqrt{8p_g(\sF)-31}-7}{2}\Big).
		\end{equation*}
		If moreover $\Sigma$ is not birational to a $(1,0)$-surface, then
		\begin{equation*}
		\vol(\sF) \geq \deg(\varphi)\cdot \Big(2p_g(\sF)-4+\frac{\sqrt{8p_g(\sF)-23}-3}{2}\Big).
		\end{equation*}
	\end{enumerate}
\end{proposition}

\subsection{Examples with maximal canonical degree}\label{sec-exam-3}
In this section, We construct examples showing that most of the bounds in \autoref{thm-1-1} and \autoref{thm-1-2} are sharp.
More precisely,
\autoref{exam-3-1}, \autoref{exam-3-2}, \autoref{exam-3-3} and \autoref{exam-3-4}
show respectively the sharpness of \eqref{eqn-1-3}, \eqref{eqn-1-3'}, \eqref{eqn-1-5} and \eqref{eqn-1-10}.

\begin{example}\label{exam-3-1}
	For any even integer $d=2k$, we construct a foliated surface $(S_k,\sF_k)$ of general type,
	such that its canonical map $\varphi$ is generically finite with
	\begin{equation}\label{eqn-6-4}
	\deg(\varphi)=d=\frac{\vol(\sF_k)}{p_g(\sF_k)-2}.
	\end{equation}
\end{example}

Let $f:\,S \to \bbp^1$ be a semi-stable fibration of curves of genus $g=2$,
whose slope
$$\lambda_f:=\frac{K_{S/\bbp^1}^2}{\deg f_*\mathcal{O}_S(K_{S/\bbp^1})}=2,\qquad\text{where~}K_{S/\bbp^1}=K_S-f^*K_{\bbp^1}.$$
Such a semi-stable fibration $f:\,S \to \bbp^1$ exists.
Indeed, one can construct such a fibration as follows.
Let $Y=\mathbb{P}_{\mathbb{P}^1}\big(\mathcal{O}_{\mathbb{P}^1} \oplus \mathcal{O}_{\mathbb{P}^1}(n)\big)$ be the Hirzebruch surface with the ruling $h:\,Y \to \mathbb{P}^1$,
and $L_m=6C_0+2m\Gamma$, where $C_0$ is the section with $C_0^2=-n$ and $\Gamma$ is a general fiber of $h$.
Then $L_m$ is very ample if $m$ is sufficiently large (in fact $m>3n$ is enough, cf. \cite[Cor\,V.2.18]{har-77}).
By the Bertini theorem \cite[Thm\,II.8.18]{har-77}, a general element $R\in |L_m|$ satisfies that
\begin{enumerate}[(i).]
	\item the divisor $R$ is smooth;
	\item the restricted map $h|_R:\,R \to \bbp^1$ has only simply ramified points, i.e., the ramification indices are all equal to $2$.
\end{enumerate}
Let $\phi:\,S \to Y$ be the double cover branched over such a general divisor $R$, and $f=h\circ\phi:\, S \to \bbp^1$ the induced fibration.
Then $f$ is a semi-stable fibration of curves of genus $2$.
Moreover, one checks easily that
$$\begin{aligned}
K_{S/\bbp^1}^2&\,=2(K_{Y/\bbp^1}+R/2)^2=4m-6n,\\
\deg f_*\mathcal{O}_S(K_{S/\bbp^1})&\,=\frac14R\cdot \big(K_{Y/\bbp^1}+R/2\big)=2m-3n.
\end{aligned}$$
Let $\sF$ be the foliation on $S$ defined by taking the saturation of $\ker(\df:\,T_{S} \to f^*T_{\bbp^1})$ in $T_{S}$.
Then $\sF$ is reduced, relatively minimal, and $K_{\sF}=K_{S/\bbp^1}$.
Hence
$$\vol(\sF)=K_{\sF}^2=4m-6n.$$
Moreover, since the Hodge bundle $\deg f_*\mathcal{O}_S(K_{S/\bbp^1})$ is of rank two and semi-positive,
it follows that
$$p_g(\sF)=h^0(S,K_{\sF})=h^0\big(\bbp^1, f_*\mathcal{O}_S(K_{S/\bbp^1})\big)=2+\deg f_*\mathcal{O}_S(K_{S/\bbp^1})=2m-3n+2.$$
Moreover, one checks easily that the canonical map $\varphi$ factorizes as
$$\xymatrix{S \ar[rr]^-{\varphi} \ar[dr]_-{\phi} && \Sigma \ar@{^(->}[r] & \bbp^{p_g(\sF)-1}\\
	&Y \ar[ur]_-{\varphi'}}$$
where $\varphi':\,Y \to \Sigma$ is defined by the complete linear system
$$|K_{Y/\bbp^1}+R/2|=|(-2C_0-n\Gamma)+3C_0+m\Gamma|=|C_0+(m-n)\Gamma|.$$
In particular,
$$\deg(\varphi)=2\deg(\varphi')=2=\frac{\vol(\sF)}{p_g(\sF)-2}.$$
Let $(S_1,\sF_1)=(S,\sF)$.
Then $(S_1,\sF_1)$ satisfies \eqref{eqn-6-4} for $d=2$.

For any even number $d=2k>2$,
let $\mathcal{L}=\mathcal{O}_{\bbp^1}(\delta_k)$ be a line bundle with $\delta_k$ sufficiently large (one can check $\delta_k>2m-3n$ would be enough),
such that $$h^0(S,K_{\sF}\otimes f^*\mathcal{L}^{-i})=h^0(\bbp^1,f_*K_{\sF}\otimes \mathcal{L}^{-i})=0,\qquad\forall\,i\geq 1.$$
Let $D\in |\mathcal{L}^{\otimes k}|$ be a reduced divisor such that $D$ does not contain the discriminant locus of $f$.
Then one obtains a cyclic cover 
$\pi_k:\,C_k \to \bbp^1$
branched exactly over $D$.
Let $S_k=S \times_{\bbp^1} C_k$ be the fiber-product and $\Pi:\, S_k \to S$ the induced cyclic cover with diagram:
$$\xymatrix{S_k \ar[rr]^-{\Pi} \ar[d]_-{f_k}  && S \ar[d]^-{f}\\
	C_k \ar[rr]^-{\pi} && \bbp^1}$$
Denote by $\sF_k=\Pi^*\sF$ the induced foliation on $S_k$.
By construction, $\sF_k$ is nothing but the folition by taking the saturation of $\ker(\df_k:\,T_{S_k} \to f^*T_{C_k})$ in $T_{S_k}$.
Note that the ramified divisor of $\Pi$ is contained in fibers and hence $\sF_k$-invariant.
By \autoref{thm-a-1}, $K_{\sF_k}=\Pi^*K_{\sF}$.
$$\begin{aligned}
\vol(\sF_k)&\,=k\vol(\sF), \\
p_g(\sF_k)&\,=\sum_{i=0}^{k-1}h^0\big(S, K_{\sF} \otimes f^* \mathcal{L}^{-i}\big)=h^0\big(S, K_{\sF})=p_g(\sF).
\end{aligned}$$
It is clear that $\Pi^*|K_{\sF}| \subseteq |K_{\sF_k}|$.
Hence $|K_{\sF_k}|=\Pi^*|K_{\sF}|$, which implies in particular that the canonical map
$\varphi_{|K_{\sF_k}|}$ factorizes through $\varphi$:
$$\xymatrix{S_k \ar[rr]^-{\varphi_{|K_{\sF_k}|}} \ar[dr]_-{\pi} && \Sigma\\
	&S \ar[ur]_-{\varphi}}$$
In particular,
$$\deg(\varphi_{|K_{\sF_k}|})=k\deg(\varphi)=2k=\frac{\vol(\sF_k)}{p_g(\sF_k)-2}.$$

\begin{example}\label{exam-3-2}
	For any integer $d\geq 1$, we construct a foliated surface $(S,\sF)$ of general type with
	$\vol(\sF)=d$ and $p_g(\sF)=3$, such that its canonical map $\varphi$ induces a generically finite map to $\bbp^2$ with $$\deg(\varphi)=d=\vol(\sF).$$
\end{example}

Let $\sF_1$ be a reduced folition of degree two on $\bbp^2$, which admits $4$ invariant lines $\{L_1,L_2,L_3,L_4\}$.
Such a foliation exists: for instance, the foliation $\sF_1$ defined by
$$v=x(\lambda x+(1+\lambda)y+\lambda c) \frac{\partial}{\partial x}+y(y+c)\frac{\partial}{\partial y},$$
where $(x,y)$ is an affine coordinate of $\bbp^2$, $\lambda$ is an irrational number, and $c\neq 0$.
One checks directly that such a foliation $\sF_1$ is reduced with $\deg(\sF_1)=2$,
and that there are four $\sF_1$-invariant lines:
$$\big\{L_1=\{x=0\},\quad L_2=\{y=0\},\quad L_3=\{x+y+c=0\},\quad L_4=L_{\infty}\big\}.$$
The canonical divisor of $\sF_1$ is
$$K_{\sF_1}=\mathcal{O}_{\bbp^2}(1).$$
Hence the foliation $\sF_1$ gives an example with $\deg(\varphi_{|K_{\sF_1}|})=1=\vol(\sF_1)$.

Let now $d\geq 2$, and $\pi_0:\,S_0 \to \bbp^2$ be the (normalized) cyclic cover of degree $d$ defined by the relation:
$$\mathcal{O}_{\bbp^2}\big(L_1+(d-1)L_2+L_3+(d-1)L_4\big) \sim \mathcal{O}_{\bbp^2}(2)^{\otimes d}.$$
Let $\rho_1:\,S_1 \to S_0$ be the desingularization.
The induced foliation $(\pi_0\circ \rho_1)^*(\sF_1)$ might not be reduced.
Let $\rho_2:\,S \to S_1$ be possible further blowing-ups such that
the induced foliation $\sF_d:=\rho_2^*\big((\pi_0\circ \rho_1)^*(\sF_1)\big)=\pi^*(\sF_1)$ is reduced.
$$\xymatrix{S \ar[rr]^-{\rho_2} \ar@/_4mm/"1,7"_-{\pi:=\pi_0\circ \rho_1\circ \rho_2} && S_1 \ar[rr]^-{\rho_1} &&
	S_0 \ar[rr]^-{\pi_0} && \bbp^2}$$

Since the branch divisor is $\sF_1$-invariant, it follows from \autoref{thm-a-1} that
$$K_{\sF_d}=\pi^*K_{\sF_1}+\mathcal{E}=\pi^*\mathcal{O}_{\bbp^2}(1)+\mathcal{E},$$
where $\mathcal{E}\geq 0$ is supported over the curves contracted by $\pi$.
It follows that the above decomposition is just the Zariski decomposition of $K_{\sF_d}$,
namely $\pi^*\mathcal{O}_{\bbp^2}(1)$ and $\mathcal{E}$ are respectively the nef and negative parts of $K_{\sF}$.
In particular,
$$\vol(\sF_d)=\vol\big(\pi^*\mathcal{O}_{\bbp^2}(1)\big)=d.$$
Moreover, $\mathcal{E}$ is contained in the fixed part of $|K_{\sF_d}|$, i.e., $|K_{\sF_d}|=|\pi^*\mathcal{O}_{\bbp^2}(1)|+\mathcal{E}$.
Since $\pi_0$ is a cyclic cover whose branch divisor is normal crossing,
it follows that the singularities on $S_0$ are all Hirzebruch-Jung type, and hence rational, cf. \cite[\S III.3]{bhpv}. Combining with \cite[Corollary\,3.11]{ev-92},
$$\pi_*\mathcal{O}_{S}= (\pi_0)_*\mathcal{O}_{S_0}=\bigoplus_{i=0}^{d-1} {\mathcal{L}^{(i)}}^{-1},$$
where $D=L_1+(d-1)L_2+L_3+(d-1)L_4$, and
$${\mathcal{L}^{(i)}}^{-1}=\mathcal{O}_{\bbp^2}(-2i)\otimes \mathcal{O}_{\bbp^2}\Big(\Big[\frac{iD}{d}\Big]\Big)=\left\{\begin{aligned}
& \mathcal{O}_{\bbp^2}, &~& \text{if~}i=0;\\
& \mathcal{O}_{\bbp^2}(-2), &~& \text{if~}1\leq i \leq d-1.
\end{aligned}\right.$$
Therefore,
$$\pi_*\big(\pi^*\mathcal{O}_{\bbp^2}(1)\big)=\mathcal{O}_{\bbp^2}(1) \otimes \pi_*\mathcal{O}_S
=\mathcal{O}_{\bbp^2}(1) \oplus \mathcal{O}_{\bbp^2}(-1)^{\oplus (d-1)}.$$
This implies in particular that
$$h^0(S,K_\sF)=h^0\big(S,\pi^*\mathcal{O}_{\bbp^2}(1)\big)=h^0\big(\bbp^2,\pi_*\pi^*\mathcal{O}_{\bbp^2}(1)\big)=3.$$
Hence
$$|K_{\sF_d}|=|\pi^*\mathcal{O}_{\bbp^2}(1)|+\mathcal{E}=\pi^*|\mathcal{O}_{\bbp^2}(1)|+\mathcal{E}.$$
It follows that the canonical map $\varphi_{|K_{\sF_d}|}$ factorizes through $\pi$ as $\varphi_{|K_{\sF_d}|}=\varphi_{|K_{\sF_1}|}\circ \pi$.
As $\deg(\varphi_{|K_{\sF_1}|})=1$, it follows that
$$\deg(\varphi_{|K_{\sF_d}|})=d=\vol(\sF).$$

\begin{example}\label{exam-3-3}
	For any integer $d\geq 1$, we construct a canonical foliated surface $(S, \sF)$ of general type, such that $S$ is a K3 surface, and
	$$\vol(\sF)=2(d+2)^2,\qquad p_g(\sF)=d^2+4d+6,
	\qquad \Longrightarrow\qquad \vol(\sF)=2p_g(\sF)-4.$$
\end{example}

Let $\sF_0$ be a reduced foliation on $\bbp^2$ of degree $d\geq 7$.
Such a foliation exists; in fact, a general foliation on $\bbp^2$ is always reduced, cf. \cite[Proposition\,3.2]{ls-20}.
\begin{claim}\label{claim-6-1}
	There exists a smooth curve $D$ with $\deg(D)=6$, such that
	\begin{enumerate}[$(i)$]
		\item $D$ is not $\sF_0$-invariant;
		\item $D$ does not contain any singular point of $\sF_0$;
		\item $\tang(\sF_0,D,p)\leq 1$ for any $p\in D$.
	\end{enumerate}
\end{claim}
\begin{proof}
	This follows from some variant of Bertini's theorem.
	In fact, let $V$ be set of all homogeneous polynoimials of degree $6$ on $\bbp^2$, and let
	$$\left\{\begin{aligned}
	&V_1=\{F\in V~|~ \div(F)~\text{is not smooth}\};\\
	&V_2=\{F\in V~|~ \div(F)~\text{is $\sF_0$-invariant}\};\\
	&V_3=\{F\in V~|~\div(F)~\text{contains some singular point of $\sF_0$}\};\\
	&V_4=\{F\in V~|~\tang(\sF_0,D,p)\geq 2~\text{for some point $p\in D\setminus \text{Sing}(\sF_0)$, where $D=\div(F)$}\}.
	\end{aligned}\right.$$
	If $V\setminus\{V_1\cup V_2 \cup V_3 \cup V_4\}\neq \emptyset$, then $D=\div(F)$ satisfies our requirements for any $F\in V\setminus\{V_1\cup V_2 \cup V_3 \cup V_4\}$.
	It is clear that $V_1$, $V_2$ and $V_3$ are of codimension at least $1$ in $V$.
	It suffices to prove that $V_4$ is also of codimension at least $1$ in $V$.
	Let $U=\bbp^2\setminus \text{Sing}(\sF_0)$, and let
	$$W=\big\{(F,p) \in V \times U~|~F(p)=0 \text{~and~} \tang(\sF_0,D,p)\geq 2, ~\text{~where $D=\div(F)$}\big\}.$$
	There are two natural projections $pr_1:\,W \to V$ and $pr_2:\,W \to U$.
	Then $pr_1(W)=V_4\subseteq V$.
	Clearly $pr_2$ is surjective, and for any fixed $p\in U$, the fiber
	$$pr_2^{-1}(p)=\left\{F\in V~\big|~F(p)=0 \text{~and~} \tang(\sF_0,D,p)\geq 2, ~\text{~where $D=\div(F)$}\right\}.$$
	Since $p\in U$ is a smooth point of $\sF_0$, the above condition defining the fiber $pr_2^{-1}(p)$ just means that $p\in D=\div(F)$ and that the tangent space of $D$ at $p$ is the same as $T_{\sF_0}$ in $T_{\bbp^2}$.
	In particular, $pr_2^{-1}(p)$ is of codimension $3$ in $V$.
	As $\dim U=\dim \bbp^2=2$, it follows that $V_4=pr_1(W)$ is of codimension at least $1$ in $V$ as required.
\end{proof}

Let $D$ be a smooth curve as in \autoref{claim-6-1},
and let $\pi:\, S \to \bbp^2$ be a double cover branched over $D$.
Then $S$ is a smooth K3 surface.
Let $\sF=\pi^*\sF_0$ be the foliation by pulling back $\sF_0$.
Then we claim that $\sF$ is reduced.
In fact, this follows from the following general fact.
\begin{lemma}\label{claim-6-2}
	Let $\sF_0$ be a reduced foliation on a smooth projective surface $S_0$,
	and let $D\subseteq S_0$ be a curve such that
		\begin{enumerate}[$(i)$]
		\item $D$ is not $\sF_0$-invariant;
		\item $D$ does not contain any singular point of $\sF_0$;
		\item $\tang(\sF_0,D,p)\leq 1$ for any $p\in D$.
	\end{enumerate}
Suppose that there is a double cover $\pi:\,S \to S_0$,
branched exactly over the curve $D$.
Then the foliation $\sF=\pi^*\sF_0$ is reduced.
\end{lemma}
\begin{proof}
	Let $q$ be a singular point of $\sF$.
	Then either $p=\pi(q)$ is a singular point of $\sF_0$,
	or $p=\pi(q)\in D$ with $\tang(\sF_0,D,p)>0$.
	If $p=\pi(q)$ is a singular point of $\sF_0$, then $p\not\in D$ by the assumption on $D$,
	and hence $\sF$ is reduced at $q$ as $\sF_0$ is reduced at $p$.
	If $p=\pi(q)\in D$ with $\tang(\sF_0,D,p)>0$,
	then $\tang(\sF_0,D,p)=1$ and $\sF_0$ is smooth at $p$ by our assumptions.
	Locally around $p$, we may choose a local coordinate $(x,y)$ such that
	$D=\{x=0\}$, and that $\pi$ is defined by $z^2=x$.
	Let $v_0=\alpha(x,y)\frac{\partial}{\partial x}+\beta(x,y) \frac{\partial}{\partial y}$
	be a local vector field defining $\sF_0$ around $p$.
	Then
	$$1=\tang(\sF_0,D,p)=I_p\big(x,\alpha(x,y)\big).$$
	Hence $\alpha(x,y)=ay+\alpha_2(x,y)$
	such that $a\neq 0$ and that $\alpha_2(x,y)$ consists of monomials of degree at least two.
	Since $\sF_0$ is non-singular at $p$, it follows that $\beta(0,0)\neq 0$.
	Note that $\omega_0=\beta(x,y)dx-\alpha(x,y)dy$ can be viewed as a local one-form defining $\sF_0$ around $p$.
	Hence the foliation $\sF$ is defined around $q$ by the following one-form
	$$\omega=\pi^*\omega_0=2z\beta(z^2,y)dz-\alpha(z^2,y)dy,$$
	or equivalently by the following local field
	$$v=\alpha(z^2,y)\frac{\partial}{\partial z}+2z\beta(z^2,y) \frac{\partial}{\partial y}.$$
	The two eigenvalues of the differential matrix $(Dv)$ at the singularity $q$
	are equal to $\pm\sqrt{2a\beta(0,0)}$.
	Hence the eigenvalue of $\sF$ at $q$ is equal to $-1$, which implies that $\sF$ is reduced at $q$.
\end{proof}

Let's now compute the invariants of $\sF$.
By \autoref{thm-a-1},
$$K_{\sF}=\pi^*(K_{\sF_0}+D/2)=\pi^*\mathcal{O}_{\bbp^2}(d+2).$$
In particular, $K_{\sF}$ is ample with
$$\vol(\sF)=K_{\sF}^2=2(d+2)^2,\qquad p_g(\sF)=h^0(K_{\sF})=h^0(\pi_*K_{\sF})=d^2+4d+6.$$
Thus, $\vol(\sF)=2p_g(\sF)-4$ as required.
Moreover, by Reider's method \cite{rei-88}, one checks easily that the canonical map $\varphi$ is an embedding once $d\geq 1$; namely, $(S,\sF)$ is a canonical foliated surface if $d\geq 1$.

\begin{example}\label{exam-3-4}
	For any integer $d\geq 2$, we construct a canonical foliated surface $(S,\sF)$ of general type with reduced singularities, such that
	\begin{enumerate}[(i)]
		\item the surface $S$ is a $(1,2)$-surface;
		\item the volume $\vol(\sF)=(2d+5)^2$, and the geometric genus $p_g(\sF)=2d^2+9d+13$, and hence
		\begin{equation}\label{eqn-6-9}
		\left\{\begin{aligned}
		\vol(\sF) &= 2p_g(\sF)-4+\frac{\sqrt{8p_g(\sF)-23}-3}{2},\quad\text{or equivalently,}\quad\\
		p_g(\sF)&= \frac{\vol(\sF)-\sqrt{\vol(\sF)}+6}{2}.
		\end{aligned}\right.	
		\end{equation}
	\end{enumerate}
\end{example}

The construction is similar to that of \autoref{exam-3-3}.
At this time, we will first choose a reduced foliation $\sF_0$ on $\bbp^2$ of degree $d\geq 2$ and with two $\sF_0$-invariant lines.
Such a foliation can be explicitly constructed as follows.
Let
$$v=x\big(\alpha+x^d+y^d\big)\frac{\partial}{\partial x}
+y\big(\beta+y^{d-1}+x^d+y^d\big)\frac{\partial}{\partial y},$$
where $\alpha,\beta \in \mathbb{C}$, and $(x,y)$ is an affine coordinate of $\bbp^2$.
Then $v$ defines a foliation $\sF_0$ of degree $d$ on the projective plane.
\begin{claim}\label{claim-3-1}
	Suppose that $\alpha,\beta\in \mathbb{C}\setminus \mathbb{Q}$ are algebraically independent.
	Then the foliation $\sF_0$ is reduced.
\end{claim}
\begin{proof}[Proof of \autoref{claim-3-1}]
	Since the degree of $\sF_0$ is $d$, the number of singularities (counted by multiplicity) is $d^2+d+1$,
	cf. \cite[Example\,2.3.2]{bru-04}.
	In fact, the singularities of $\sF_0$ are all on the affine plane satisfying
	\begin{equation}\label{eqn-3-2}
		\left\{\begin{aligned}
			&x\big(\alpha+x^d+y^d\big)=0;\\
			&y\big(\beta+y^{d-1}+x^d+y^d\big)=0.
		\end{aligned}\right.
	\end{equation}
	The differential matrix of the vector field $v$ is
	\begin{equation}\label{eqn-3-3}
		Dv=\left(\begin{aligned}
			(d+1)x^d+y^d+\alpha \quad &\quad\qquad\qquad dxy^{d-1} \\
			dx^{d-1}y \qquad\qquad &\quad (d+1)y^d+dy^{d-1}+x^d+\beta
		\end{aligned}\right).
	\end{equation}
	Let $p_0=(x_0,y_0)$ be any singularity of $\sF_0$,
	and $\lambda_1$ and $\lambda_2$ be the two eigenvalues of the differential matrix $(Dv)(p_0)$ at $p_0$.
	To show that $\sF_0$ is reduced at $p_0$, it suffices to show that
	the eigenvalue $\lambda_{p_0}=\frac{\lambda_1}{\lambda_2}$
	\big(up to the exchange $\frac{\lambda_1}{\lambda_2} \leftrightarrow \frac{\lambda_2}{\lambda_1}$\big)
	of $\sF_0$ at $p_0$ is not a positive rational number
	if $\alpha,\beta\in \mathbb{C}\setminus \mathbb{Q}$ are algebraically independent.
	Let compute the eigenvalue $\lambda_{p_0}$ in the following.
	\begin{enumerate}[(i)]
	\item If $x_0=0$ and $y_0=0$, i.e., $p_0=(0,0)$, then $\lambda_{p_0}=\frac{\alpha}{\beta}$ by \eqref{eqn-3-3}.
	
	\item If $y_0=0$ and $x_0\neq 0$, then $x_0^d=-\alpha$ by \eqref{eqn-3-2}.
	According to \eqref{eqn-3-3}, $\lambda_1=-d\alpha$ and $\lambda_2=\beta-\alpha$,
	and hence $\lambda_{p_0}=\frac{d\alpha}{\alpha-\beta}$.
	
	\item If $x_0=0$ and $y_0\neq 0$, then $\beta+y_0^{d-1}+y_0^d=0$ by \eqref{eqn-3-2}.
	According to \eqref{eqn-3-3},
	$$\lambda_1=y_0^d+\alpha,\qquad\lambda_2=(d+1)y_0^d+dy_0^d+\beta=-y_0^d-d\beta.$$
	Hence
	$$\lambda_{p_0}=-\frac{(d+1)y_0^d+\alpha}{y_0^d+d\beta}.$$
	
	\item Finally, if $x_0\neq 0$ and $y_0\neq 0$, then by \eqref{eqn-3-2},
		\begin{equation*} 
			y_0^{d-1}=\alpha-\beta,\quad\text{and}\quad x_0^d+y_0^d+\alpha=0.
	\end{equation*}
	According to \eqref{eqn-3-3},
	$$(Dv)(p_0)=\left(\begin{aligned}
		dx_0^d \qquad &\quad\qquad dx_0y_0^{d-1} \\
		dx_0^{d-1}y_0 \quad &\quad dy_0^d+(d-1)y_0^{d-1}
	\end{aligned}\right).$$
	Hence
	$$\lambda_i=\frac{-\big(d\alpha-(d-1)y_0^{d-1}\big)\pm\sqrt{\big((2d-1)\alpha-(d-1)\beta\big)^2+4d(d-1)(\alpha-\beta)^2y_0}}{2}.$$
	Thus the eigenvalue of $\sF_0$ at $p_0$ is
	$$\lambda_{p_0}=-1+\frac{\big(\alpha+(d-1)\beta\big)\,\sqrt{\big((2d-1)\alpha-(d-1)\beta\big)^2+4d(d-1)(\alpha-\beta)^2y_0}}{2d(d-1)(\alpha-\beta)\big(\alpha+(\alpha-\beta)y_0\big)}.$$
    \end{enumerate}

	Since $\alpha,\beta\in \mathbb{C}\setminus \mathbb{Q}$ are algebraically independent,
	one sees that the eigenvalue of $\sF_0$ at each singularity $p_0$ is  $\lambda_{p_0}\notin \mathbb{Q}$.
	In particular, $p_0$ is non-denegerate reduced for any singularity $p_0$ of $\sF_0$,
	and hence $\sF_0$ is reduced as required.
\end{proof}

\begin{claim}\label{claim-3-2}
	Let $\alpha,\beta\in \mathbb{C}\setminus \mathbb{Q}$ such that $\sF_0$ is reduced as in \autoref{claim-3-1}.
	Then there exists a smooth curve $D$ with $\deg(D)=5$, such that
	\begin{enumerate}[$(i)$]
		\item $D$ is not $\sF_0$-invariant;
		\item $D$ does not contain any singular point of $\sF_0$;
		\item $\tang(\sF_0,D,p)\leq 1$ for any $p\in D$;
		\item $D$ intersects both $L_{0}:=\{x=0\}$ and $L_{\infty}:=\{y=0\}$ transversely.
	\end{enumerate}
\end{claim}
\begin{proof}[Proof of \autoref{claim-3-2}]
	By \autoref{claim-6-1} and its proof, a general curve $D$ with $\deg(D)=5$ satisfies the first three conditions above.
	For the last one, we can prove similarly;
	namely, the subsets of all homogeneous polynomials of degree $5$ on $\bbp^2$,
	such that the corresponding curves do not satisfy (iv) will be a proper closed subsets.
	Hence a general curve $D$ with $\deg(D)=5$ will satisfy all the four conditions above.
	We leave the details to the interested readers.
\end{proof}

Come back to the construction of \autoref{exam-3-4}.
Let $\alpha,\beta\in \mathbb{C}\setminus \mathbb{Q}$ such that $\sF_0$ is reduced as in \autoref{claim-3-1},
and $D$ be a smooth curve as in \autoref{claim-3-2}.
Let $\sigma:\,Y_1 \to \bbp^1$ be the blowing-up centered at $p=(0,0)$ with exceptional curve $C_1$,
and $h_1:\,Y_1 \to \bbp^1$ be the induced ruling.
Denote by $F_0$ and $F_{\infty}$ the strict transforms of $L_0=\{x=0\}$ and $L_{\infty}=\{y=0\}$ in $Y_1$ respectively.
Let $\pi:\,\bbp^1 \to \bbp^1$ be the double cover branched over the two points $h(F_0)$ and $h(F_{\infty})$,
and denote by $\Pi:\,Y_2=Y_1\times_{\bbp^1}{\bbp^1} \to Y_1$ the induced double cover branched over $F_0$ and $F_{\infty}$.
Then $Y_2\cong \mathbb{P}_{\mathbb{P}^1}\big(\mathcal{O}_{\mathbb{P}^1} \oplus \mathcal{O}_{\mathbb{P}^1}(-2)\big)$ is a Hirzebruch surface, and $C_2^2=-2$, where $C_2=\Pi^{-1}(C_1)$.
Let $D_2=\big(\sigma\circ\Pi\big)^{-1}(D)$, where $D$ is the smooth curve as in \autoref{claim-3-2}.
Then we have the following linear equivalence relation.
$$C_2+D_2\sim 2\big(3C_2+5F\big),$$
where $F$ is a general fiber of the ruling $h_2:\,Y_2 \to \bbp^1$.
Hence one can construct a double cover $\Psi:\, \wt S \to Y_2$ branched exactly over $C_2+D_2$.
Since $C_2+D_2$ is a smooth divisor,
$\wt S$ is smooth and the strict transform $E=\Psi_*^{-1}(C_2)$ is an exceptional curve.
Let $\rho:\,\wt S \to S$ be the blowing-down of this exceptional curve.
$$\xymatrix{\wt S \ar[d]^-{\Psi} \ar@/_6mm/"3,1"_-{f} \ar[rr]^-{\rho} && S\\
Y_2 \ar[d]^-{h_2}\ar[rr]^-{\Pi} && Y_1 \ar[d]^-{h_1}\ar[rr]^-{\sigma} && \bbp^2 \\
\bbp^1 \ar[rr]^-{\pi} && \bbp^1}$$

\noindent
By \cite[\S\,2]{hor-76-2}, one checks that $S$ is minimal with $\vol(S)=1$ and $p_g(\sF)=2$,
i.e., $S$ is a $(1,2)$-surface.
Let $$\sF_1=\sigma^*\sF_0,\qquad \sF_2=\Pi^*\sF_1,\qquad \wt \sF=\Psi^*\sF_2,\qquad \sF=\rho_*\wt \sF.$$
It is clear that $\sF_1$ is reduced, since $\sF_0$ is reduced by \autoref{claim-3-1}.
As $F_{0}$ and $F_{\infty}$ is $\sF_1$-invariant, and the eigenvalues of $\sF_0$
at all its singularities belong to $\mathbb{C}\setminus \mathbb{Q}$,
it follows that $\sF_2$ is reduced as well.
Finally, by our choice of $D$ in \autoref{claim-3-2},
its inverse image $D_2$ in $Y_2$ together with the foliation $\sF_2$ satisfies the assumptions of \autoref{claim-6-2}.
Hence the foliation $\wt \sF$ is also reduced.
By the Riemann-Roch theorem, one computes that
$$\left\{\begin{aligned}
	K_{\sF_2}&\,=\Pi^*K_{\sF_1}=\Pi^*\big(\sigma^*K_{\sF_0}\big)\sim dC_2+2dF;\\
	K_{\wt\sF}&\,=\Psi^*K_{\sF_2}+\wt D_2 \sim (2d+5)\big(E+\wt F\big),
\end{aligned}\right.$$
where $\wt F\subseteq \wt S$ (resp. $\wt D_2\subseteq \wt S$, resp. $E\subseteq \wt S$)
the strict transform of a general fiber $F$ of $h_2$ (resp. $D_2$, resp. $C_2$).
It is clear that  $K_{\wt \sF}$ is nef, and hence
$$\vol(\wt \sF)=(2d+5)^2(E+\wt F)^2=(2d+5)^2.$$
Moreover,
$$K_{\wt S}=\Psi^*K_{Y_2}+E+\wt D_2\sim 2E+\wt F.$$
Hence $K_{\wt\sF}\sim K_{\wt S}+(2d+3)E+(2d+4)\wt F$.
By the Mumford vanishing theorem,
$$p_g(\wt \sF)=h^0\big(\wt S, K_{\wt\sF}\big)=\chi\big(\wt S, K_{\wt\sF}\big)=2d^2+9d+13.$$
Moreover, by Reider's method \cite{rei-88}, one checks easily that the canonical map $\varphi_{|K_{\wt \sF}|}$ is brational.
Finally, one checks easily that $K_{\wt \sF}\cdot E=0$ and $Z(\wt \sF, E)=2$.
Moreover, there are two singularities of $\wt\sF$ on $E$, which are the intersection points $E\cap \wt F_0$ and $\wt F_{\infty}$,
where $\wt F_0$ (resp. $\wt F_{\infty}$) is the inverse image of $F_0$ (resp. $F_{\infty}$).
Hence $E$ is an $\wt \sF$-exceptional curve, cf. \cite[\S\,5]{bru-04}.
In other words, $\rho(E)$ is a reduced singularity of $\sF$, and hence $\sF$ is reduced.
Moreover,
$$\vol(\sF)=\vol(\wt\sF),\qquad p_g(\sF)=p_g(\wt\sF).$$
This compete the construction.

\section{The canonical map induces a fibration}\label{sec-fibration}
In this section, we will always assume that
the image $\Sigma=\varphi(S)$ under the canonical map $\varphi$ is of dimension one,
and aim to prove \autoref{thm-1-3} and \autoref{thm-1-4}.
By replacing $(S,\sF)$ by a suitable reduced model, we may assume that
$\sF$ is reduced, and the canonical map $\varphi=\varphi_{|K_{\sF}|}$
is defined by the complete linear system $|K_{\sF}|$.
In this case, by the Stein factorization, we obtain a diagram as follows.
$$\xymatrix{&& \wt{S} \ar[dll]_-{f} \ar[d]^-{\phi} \ar[rrr]^-{\sigma}  &&& S \ar@{-->}[d]^-{\varphi=\varphi_{|K_{\sF}|}}\\
	B \ar[rr]^-{\pi} && Y \ar[rrr]^-{\rho}_-{\text{desingularization}} &&&\Sigma\, \ar@{^(->}[r] & \bbp^{p_g(\sF)-1}}$$
	
\noindent
Here $\pi:\, B \to Y$ is finite, and $f:\,\wt{S} \to B$ is a family of curves with connected fibers.
By construction, the map 
$$\rho\circ\phi:\,\wt{S} \lra \Sigma \hookrightarrow \bbp^N,$$
is defined by the complete linear system $|\wt{M}|$,
where $|\wt{M}|$ is obtained by blowing-up the base points of the moving part $|M|$ in \eqref{eqn-2-22}.
Since $|\wt{M}|$ is base-point-free and induces a fibration $f:\,\wt{S} \to B$,
it follows that
$$p_g(\sF)=h^0(S,M)=h^0(\wt{S}, \wt{M})=h^0\big(B,f_*\mathcal{O}_{\wt{S}}(\wt M)\big).$$
According to the Riemann-Roch theorem,
$$\deg(L)\geq p_g(\sF)-1,$$
where $L=f_*\mathcal{O}_{\wt{S}}(\wt M)$ is a line bundle on $B$.
Note that $\wt{M}=f^*(L)$.
Hence, we have numerically,
\begin{equation}\label{eqn-4-0}
\wt{M} \sim_{num} \deg(L)\,F, \qquad \text{with~}\deg(L) \geq p_g(\sF)-1,
\end{equation}
where $F$ is a general fiber of $f$.
There are two foliations on the surface $\wt{S}$:
the first is $\wt\sF=\sigma^*\sF$ by pulling-back the foliation $\sF$ by the birational morphism $\sigma$;
the other one is $\mathcal{G}$ defined by taking the saturation of $\ker(\df:\,T_{\wt S} \to f^*T_{B})$ in $T_{\wt S}$.
They can be the same or different from each other.

The organization of this section is as follows.
In \autoref{sec-foliation-same}, we consider the case when the foliation $\wt\sF=\sigma^*\sF$ is the same as the foliation $\mathcal{G}$ defined by taking the saturation of $\ker(\df:\,T_{\wt S} \to f^*T_{B})$ in $T_{\wt S}$, and prove \autoref{thm-1-3}.
In \autoref{sec-foliation-different}, we consider the case when the foliation $\wt\sF=\sigma^*\sF$ is different from the foliation $\mathcal{G}$, and prove \autoref{thm-1-4}.
In \autoref{sec-exam-4}, we will construct two examples of foliations
whose canonical maps induce fibrations.

\subsection{Foliation defined by the canonical fibration}\label{sec-foliation-same}
In this subsection, we consider the case when the foliation $\wt\sF=\sigma^*\sF$ is the same as the foliation defined by taking the saturation of $\ker(\df:\,T_{\wt S} \to f^*T_{B})$ in $T_{\wt S}$.
We will prove \autoref{thm-1-3}.
\begin{proof}[{Proof of \autoref{thm-1-3}}]
	Without loss of generality, we may assume that the original foliation $\sF$ is relatively minimal.
	Since the pulling-back foliation $\wt\sF=\sigma^*\sF$ is the same as the foliation defined by taking the saturation of $\ker(\df:\,T_{\wt S} \to f^*T_{B})$ in $T_{\wt S}$, it follows that $\sigma$ is the identity, i.e., $|M|$ has no base point; otherwise, the last exceptional curve $E$ of $\sigma$ maps surjectively onto the base $B$, and hence is not $\wt\sF$-invariant.
	This implies that the image $\sigma(E)$ would be a dicritical singularity of $\sF$ (cf. \cite[Proposition\,1.1]{bru-04}), which contradicts the assumption that $\sF$ is reduced.
	Hence $\sigma:\,\wt S \overset{=}{\lra} S$,
	$\wt \sF=\sF$ and
	$$K_{\wt \sF}=K_{\sF}=K_{S/B}-\sum (a_i-1)C_i,$$
	where the sum is taken over all irreducible components $C_i$'s in fibers of $f$,
	and $a_i$ is the multiplicity of $C_i$ in its fiber.
	By \eqref{eqn-2-22} and \eqref{eqn-4-0}, we may also write
	\begin{equation}\label{eqn-41-2}
		K_{\sF}=M+Z,\qquad \text{with $M=f^*(L)$ satisfying $\deg(L)\geq p_g(\sF)-1$}.
	\end{equation}
	In particular,
	$$d=K_{\sF}\cdot F=K_{S/B}\cdot F=Z\cdot F=K_{S/B}F=2g(F)-2\geq 2.$$
	It remains to prove \eqref{eqn-1-6}.
	
	Let $\epsilon:\,S \to X$ be the contraction of all components of the negative part $N$ in the Zariski decomposition of $K_{\sF}$, and let $\ol K=\epsilon_*K_{\sF}$.
	Since $\sF$ is induced by the fibration $f$, all the components in $N$ are contained in fibers of $f$.
	Hence there is an induced fibration $\bar f:\,X \to B$.
	By \eqref{eqn-2-22},
	$$\ol K=\epsilon_*K_{\sF}=\epsilon_*(M)+\epsilon_*(Z)=\epsilon_*(M)+\sum_{j=1}^{m}a_j\ol D_j+ \ol D_v,$$
	where $\ol D_v$ consists of curves in fibers of the induced fibration $\bar f:\,X \to B$, and $\ol D_j$ maps surjectively onto $B$ for $1\leq j \leq m$.
	By \eqref{eqn-41-2},
	$$\epsilon_*(M)=\bar f^*(L), \quad\text{with}\quad \deg(L)\geq p_g(\sF)-1.$$
	Since $\ol D_{j_0}$ maps surjectively onto $B$, for any $1\leq j_0 \leq m$,
	it follows that the inverse image $\epsilon_*^{-1}\big(\ol D_{j_0}\big)$ is not $\sF$-invariant.
	Thus by \cite[proposition\,2.13]{cf-18} (see also \cite[Lemma III.1.1]{mcq-08}),
	$$0\leq (\ol K+\ol D_{j_0}) \cdot \ol D_{j_0}=(a_{j_0}+1)\ol D_{j_0}^2+\Big(\bar f^*(L)+\sum_{j\neq j_0}a_j \ol D_j+ \ol D_v\Big)\cdot \ol D_{j_0}.$$
	Hence
	$$\begin{aligned}
		\ol K \cdot \ol D_{j_0}\,&=a_{j_0} \ol D_{j_0}^2+\Big(\bar f^*(L)+\sum_{j\neq j_0}a_j \ol D_j+ \ol D_v\Big) \cdot \ol D_{j_0}\\
		&\geq -\frac{a_{j_0}}{a_{j_0}+1}\Big(\bar f^*(L)+\sum_{j\neq j_0}a_j \ol D_j+ \ol D_v\Big) \cdot \ol D_{j_0}+\Big(\bar f^*(L)+\sum_{j\neq j_0}a_j \ol D_j+ \ol D_v\Big) \cdot \ol D_{j_0}\\
		&= \frac{1}{a_{j_0}+1}\Big(\bar f^*(L)+\sum_{j\neq j_0}a_j \ol D_j+ \ol D_v\Big)\cdot \ol D_{j_0} \\
		&\geq \frac{\bar f^*(L)\cdot \ol D_{j_0}}{a_{j_0}+1}\geq \frac{\ol D_{j_0} \cdot F}{a_{j_0}+1}\,\big(p_g(\sF)-1\big).
	\end{aligned}$$
	By the Negativity Lemma, $P=\epsilon^*(\ol K)$, where $P$ is the positive part in the Zariski decomposition of $K_{\sF}$ as in \eqref{eqn-zariski}.
	Therefore,
	$$\begin{aligned}
		\vol(\sF)\,&=P^2=\ol K^2=\ol K\cdot \Big(\bar f^*(L)+\sum_{j=1}^{m}a_j \ol D_j+ \ol D_v\Big)\\
		&\geq \ol K \cdot \bar f^*(L)+ \sum_{j=1}^{m} a_j\ol K\cdot \ol D_j\\
		&\geq d\big(p_g(\sF)-1\big)+\sum_{j=1}^{m}\frac{a_j\ol D_{j}\cdot F}{a_{j}+1}\,\big(p_g(\sF)-1\big)\\
		&\geq d\big(p_g(\sF)-1\big)+\frac{d}{d+1}\big(p_g(\sF)-1\big),\qquad\text{since~}\sum_{j=1}^{m}a_j\ol D_{j}\cdot F=d,\\
		&=\frac{d(d+1)}{d+1}\big(p_g(\sF)-1\big).
	\end{aligned}$$
	This proves \eqref{eqn-1-6}.
\end{proof}

\subsection{Foliation different from the canonical fibration}\label{sec-foliation-different}
In this subsection, we consider the case when the foliation $\wt\sF=\sigma^*\sF$ is different from the foliation defined by taking the saturation of $\ker(\df:\,T_{\wt S} \to f^*T_{B})$ in $T_{\wt S}$.
We will first do some preparations.
The proof of \autoref{thm-1-4} will be given after \autoref{prop-4-4}.

As mentioned at the beginning of this section, we may assume that $\sF$ is reduced.
Let $K_{\sF}$ be its canonical divisor, and suppose that
the image $\varphi(S)$ is of one-dimension, where
$$\varphi=\varphi_{|K_{\sF}|}:\,S \dashrightarrow \Sigma \subseteq \mathbb{P}^{p_g(\sF)-1},$$ is the rational map defined by $|K_{\sF}|$.
Recall the following diagram introduced at the beginning of this section.
$$\xymatrix{&& \wt{S} \ar[dll]_-{f} \ar[d]^-{\phi} \ar[rrr]^-{\sigma}  &&& S \ar@{-->}[d]^-{\varphi}\\
	B \ar[rr]^-{\pi} && Y \ar[rrr]^-{\rho}_-{\text{desingularization}} &&&\Sigma \, \ar@{^(->}[r] & \bbp^{p_g(\sF)-1}}$$
	
\noindent
The foliation $\sF$ lifts to a foliation $\wt{\sF}$ on $\wt{S}$, which is also reduced.
Since $\sF$ is reduced,
$$K_{\wt{\sF}}=\sigma^*K_{\sF} +\mathcal{E},$$
where $\mathcal{E}$ is some effective divisor supported on the exceptional curves of the birational morphism $\sigma$.
Moreover,
$$|K_{\wt{\sF}}|=\sigma^*|K_{\sF}|+\mathcal{E}.$$
It follows that the map $\varphi_{|K_{\wt{\sF}}|}$ factors through $\varphi$,
and hence it induces the same fibration $f:\,\wt{S} \to B$.
By replacing $(S,\sF)$ by $(\wt{S},\wt{\sF})$, we may assume that
the moving part of the linear system $|K_{\sF}|$ is base-point-free,
so that the canonical map $\varphi$ induces a fibration $f:\, S \to B$.
$$\xymatrix{ && S \ar[d]^-{\phi} \ar[lld]_-{f} \ar[rrrd]^-{\varphi=\varphi_{|K_{\sF}|}} &&& \\
	B \ar[rr]^-{\pi} &&Y \ar[rrr]^-{\rho}_-{\text{desingularization}} &&&\Sigma\, \ar@{^(->}[r] & \bbp^{p_g(\sF)-1}}$$
	
\noindent
After such a replacement,
the foliation $\sF$ is still reduced of general type with the same volume and geometric genus,
but not necessarily relatively minimal any more.	
Let $F$ be a general fiber of $f$.
As the foliation $\sF$ is assumed to be of general type,
it follows that
\begin{equation}\label{eqn-41-1}
	d=K_{\sF} \cdot F \geq 1.
\end{equation}

\begin{proposition}\label{prop-4-2}
	Let $\sF$ be a foliation of general type on $S$ whose canonical map induces a fibration $f:\, S \to B$ as above.
	Suppose that $\sF$ is different from the foliation $\mathcal{G}$ defined by taking the saturation of $\ker(\df:\,T_{S} \to f^*T_{B})$ in $T_{S}$.
	Then the canonical divisor $K_{\sF}$ has the following form
	\begin{equation}\label{eqn-4-5}
		K_{\sF} = f^*K_B+\Delta,
	\end{equation}
	where $\Delta$ is effective.
	Moreover, for any irreducible curve $C\subseteq \Delta$ mapping surjectively onto the base $B$,
	it is not $\sF$-invariant.
\end{proposition}
\begin{proof}
	Let $\omega_i$ be any local holomorphic one-form on the base curve $B$.
	Since $\sF$ is different from the foliation $\mathcal{G}$ defined by taking the saturation of $\ker(\df:\,T_{S} \to f^*T_{B})$ in $T_{S}$,
	the contraction of $f^*(\omega_i)$ with the local field $v_i$ defining $\sF$ gives a non-zero local section of $K_{\sF}$.
	Globally, it is just the following contraction map
	\begin{equation*}
		\begin{aligned}
			H^0\big(S, T_{S}(T_{\sF}^*)\big) &\quad\otimes&& f^*H^0(B,\Omega^1_B) &\lra\quad &  H^0(S,K_{\sF})\\
			\big(v, &&& f^*\omega\big) &\mapsto\,\quad & (v,f^*\omega),
		\end{aligned}\qquad
	\end{equation*}
	where $v=\{(U_i,v_i)\}$ is a family of local vector fields defining $\sF$, which can be viewed as a section of $H^0\big(S, T_{S}(T_{\sF}^*)\big)$ by \eqref{eqn-2-8},
	and $\omega=\{(V_i,\omega_i)\} \in H^0(B,\Omega^1_B)$ with $f(U_i) \subseteq V_i$.
	Remark that the above contraction map also makes sense if $\omega$ is a rational one-form on $B$
	(in this case, the image $(v,f^*\omega)$ would be a rational section of $K_{\sF}$ in general).
	To be concrete,
	let $t_i$ be a local coordinate of $B$ on $V_i$.
	Then $\{(V_i,dt_i)\}$ defines a holomorphic section of $\tilde \omega\in H^0(B,\Omega^1_B(-K_B))$, i.e., $\tilde \omega$ is a twist one-form on $B$.
	Hence
	\begin{equation}\label{eqn-4-7}
		\alpha:=(v,f^*\tilde\omega)=\big\{\big(U_i,\, (v_i,f^*dt_i)\big)\big\} \in H^0(S,K_{\sF}\otimes f^*(-K_B)).
	\end{equation}
	In other words,
	\begin{equation*}
		K_{\sF} - f^*K_B = \div(\alpha).
	\end{equation*}
	By construction, $\alpha$ is locally holomorphic, and hence $\Delta=\div(\alpha)$ is effective.
	
	Let $C\subseteq \Delta$ be any irreducible curve mapping surjectively onto the base $B$.
	Our aim is to prove that $C$ is not $\sF$-invariant.
	Suppose on the contrary that $C$ is $\sF$-invariant.
	Let $p\in C$ be a general point.
	Then we may assume that $C$ is smooth at $p$.
	Moreover, there exists a suitable local coordinate $(x,t)$ around $p$,
	such that $C=\{x=0\}$, and that the map $f$ is given by $(x,t) \mapsto t$.
	Suppose that $v=h(x,t)\frac{\partial}{\partial x}+g(x,t)\frac{\partial}{\partial t}$ around $p$.
	Since $C$ is $\sF$-invariant, it follows that $x$ divides $h(x,t)$, i.e.,
	$h(x,t)=x\cdot \tilde h(x,t)$ for some local holomorphic function $\tilde h(x,t)$.
	By the above argument, locally around $p$,
	$$\alpha=(v,dt)=g(x,t).$$
	As $C\subseteq \Delta=\div(\alpha)$,
	it follows that $x$ divides $g(x,t)$, i.e.,
	$g(x,t)=x\cdot \tilde g(x,t)$ for some local holomorphic function $\tilde g(x,t)$.
	This implies that $x$ divides the local field $v$ defining the foliation $\sF$,
	which is a contradiction, as $v$ has at most isolated zeros by definition.
	This completes the proof.
\end{proof}

If $K_{\sF}\cdot F< 2g(F)-2$, we can give a more detailed description on the effective divisor $\Delta$
appearing in \eqref{eqn-4-5}.
\begin{lemma}\label{lem-4-1}
	Let $\sF$ be a foliation of general type on $S$ whose canonical map induces a fibration $f:\, S \to B$ as before.
	Suppose that $K_{\sF}\cdot F< 2g(F)-2$, where $F$ is a general fiber of $f$.
	Let $F_0$ be any fiber of $f:\, S \to B$.
	Then there is at least one component $C\subseteq F_0$ which is not $\sF$-invariant.
	In particular, any smooth fiber of $f$ is not $\sF$-invariant.
\end{lemma}
\begin{proof}
	We prove by contradiction. Suppose that there exists a fiber $F_0=\sum\limits_{j=1}^{r} m_{j}C_{j}$, such that every component $C_j$ is $\sF$-invariant.
	Since $\sF$ is reduced, the fiber $F_0$ is normal crossing.
	According to \autoref{prop-2-2}, for any component $C_j \subseteq F_0$,
	\begin{equation}\label{eqn-4-10}
		Z(\sF,C_j)=K_{\sF}C_j-K_{S}C_j-C_j^2.
	\end{equation}
	By definition, if $p\in C_j$ is a singularity of $\sF$, then $Z(\sF,C_j,p)\geq 1$.
	Hence
	$$Z(\sF,C_j)\geq \sum_{i\neq j} C_iC_j.$$
	Combining this with \eqref{eqn-4-10},
	$$\sum_{i\neq j} C_iC_j \leq K_{\sF}C_j-K_{S}C_j-C_j^2,\qquad \text{for any fixed~}1\leq j \leq r.$$
	It follows that
	$$\begin{aligned}
		\sum_{i=1}^{r}\sum_{j\neq i} m_jC_iC_j=\sum_{j=1}^{r}\sum_{i\neq j} m_jC_iC_j &\,\leq \sum_{j=1}^{r}\left(m_jK_{\sF}C_j-m_jK_{S}C_j-m_jC_j^2\right)\\
		&\,= (K_{\sF}-K_S)\cdot \sum_{j=1}^{r}m_jC_j - \sum_{j=1}^{r}m_jC_j^2\\
		&\,= (K_{\sF}-K_S)\cdot F_0 - \sum_{i=1}^{r}m_iC_i^2 \\
		&\,=K_{\sF}\cdot F-(2g(F)-2)- \sum_{i=1}^{r}m_iC_i^2.
	\end{aligned}$$
	Since $K_{\sF}\cdot F< 2g(F)-2$ by assumption, it follows that
	$$0>K_{\sF}\cdot F-(2g(F)-2)\geq \sum_{i=1}^{r}\left(\sum_{j\neq i} m_jC_iC_j+m_iC_i^2\right)=\sum_{i=1}^{r} F_0\cdot C_i=0.$$
	This gives a contradiction.
	The proof is complete.
\end{proof}

\begin{proposition}\label{prop-4-3}
	Let $\sF$ be a foliation of general type on $S$ whose canonical map induces a fibration $f:\, S \to B$ as before.
	Suppose that $K_{\sF}\cdot F< 2g(F)-2$, where $F$ is a general fiber of $f$. Let $\Delta$ be the effective divisor as in \eqref{eqn-4-5}.
	Then
	\begin{equation}\label{eqn-4-6}
		\Delta=\Delta_h+\sum_{i=1}^{\ell}\sum_{j=1}^{r_i} n_{ij}C_{ij},
	\end{equation}
	where $\Delta_h\subseteq \Delta$ consists of the horizontal components of $\Delta$,
	and $C_{ij}$'s (pairwise different) are all contained in fibers of $f$.
	Moreover, the coefficients $n_{ij}$'s satisfy the following property:
	if $F_i=\sum\limits_{j=1}^{r_i} m_{ij}C_{ij}$ is a fiber of $f$,
	then there exists at least one $1\leq j \leq r_i$ such that $n_{ij}<m_{ij}$.
\end{proposition}
\begin{proof}
	Since $K_{\sF}\cdot F< 2g(F)-2$, by \autoref{thm-1-3}
	the foliation $\sF$ is different from the foliation $\mathcal{G}$ defined by taking the saturation of $\ker(\df:\,T_{S} \to f^*T_{B})$ in $T_{S}$.
	Let $F_i=\sum\limits_{j=1}^{r_i} m_{ij}C_{ij}$ be any fiber of $f$, and let
	\begin{equation*}
		\Delta=\div(\alpha)=\sum_{j=1}^{r_i} n_{ij}C_{ij}+\Delta',
	\end{equation*}
	where $\alpha$ is given in \eqref{eqn-4-7},
	and the support of $\Delta'$ does not contain any component of $F_i$.
	Then we have to show that
	\begin{equation}\label{eqn-4-11}
		\text{there exists at least one $1\leq j \leq r_i$ such that $n_{ij}<m_{ij}$.}
	\end{equation}
	
	Let $p_{ij} \in C_{ij}$ be a general point of $C_{ij}$.
	As $p_{ij} \in C_{ij}$ is general, there exist a local coordinate $(x,y)$ around $p_{ij}$ such that $C_{ij}=\{x=0\}$, and a local coordinate $t$ around $q=f(F_i)$ such that the map $f$ is defined by $t=x^{m_{ij}}$.
	Suppose that $v=h(x,y)\frac{\partial}{\partial x}+g(x,y)\frac{\partial}{\partial y}$ around $p_{ij}$.
	Then around the point $p_{ij}$, $f^*dt=m_{ij}x^{m_{ij}-1}dx$, and hence
	$$\alpha=\Big(h(x,y)\frac{\partial}{\partial x}+g(x,y)\frac{\partial}{\partial y},f^*dt\Big)=m_{ij}h(x,y)x^{m_{ij}-1}.$$
	Therefore,
	$$n_{ij}=m_{ij}-1+\text{order}_{x}\big(h(x,y)\big),$$
	where $$\text{order}_{x}\big(h(x,y)\big)=\max\left\{k~\big|~x^k\text{~divides~}h(x,y)\right\}.$$
	It follows that
	$$n_{ij} \geq m_{ij}, 
	\quad\Longleftrightarrow\quad x\text{~divides~}h(x,y),
	\quad\Longleftrightarrow\quad\text{$C_{ij}$ is $\sF$-invariant}.$$
	Hence \eqref{eqn-4-11} follows immediately from \autoref{lem-4-1}.
	This completes the proof.
\end{proof}

\begin{proposition}\label{prop-4-4}
	Let $\sF$ be a foliation of general type on $S$ whose canonical map induces a fibration $f:\, S \to B$ as before.
	Suppose that $K_{\sF}\cdot F< 2g(F)-2$, where $F$ is a general fiber of $f$.
	Then either $g(B)=0$, or $g(B)\geq 2$.
	Moreover, if $g(B) \geq 2$, then
	\begin{equation}\label{eqn-4-3}
		p_g(\sF)=g(B).
	\end{equation}   
	More precisely, if $g(B)\geq 2$, then
	\begin{equation}\label{eqn-4-4}
		|K_{\sF}| = |f^*K_B|+\Delta,
	\end{equation}
	such that $|f^*K_B|$ (resp. $\Delta$) is the moving part (resp. fixed part) of $|K_{\sF}|$,
	where $\Delta$ is effective as in \eqref{eqn-4-5}.
\end{proposition}
\begin{proof}
	Since $K_{\sF}\cdot F< 2g(F)-2$, by \autoref{thm-1-3}
	the foliation $\sF$ is different from the foliation $\mathcal{G}$ defined by taking the saturation of $\ker(\df:\,T_{S} \to f^*T_{B})$ in $T_{S}$.
	In particular, $g(F)\geq 2$, since $K_{\sF}\cdot F\geq 1$ by \eqref{eqn-41-1}.
	According to \autoref{prop-4-2} and \autoref{prop-4-3},
	\begin{equation}\label{eqn-4-8}
		K_{\sF}=f^*K_B+\Delta,\qquad \text{with~}\Delta=\Delta_h+\Delta_v=\Delta_h+\sum_{i=1}^{\ell}\sum_{j=1}^{r_i} n_{ij}C_{ij},
	\end{equation}
	where $\Delta_h\subseteq \Delta$ consists of the horizontal components, and $C_{ij}$'s are all contained in fibers of $f$.
	In particular,
	$$p_g(\sF)=h^0(S, K_{\sF}) \geq h^0(B,K_B)=g(B).$$
	On the other hand, since the canonical map $\varphi$ induces the fibration $f:\,S \to B$,
	it follows that
	$$|K_{\sF}|=|f^*(L)|+Z,$$
	where $Z$ is the fixed part, and $L$ is some divisor on $B$ with $h^0(B,L)=h^0(S,K_{\sF})=p_g(\sF)$.
	Therefore, for any effective divisor $D\in |K_{\sF}|$ and any horizontal component $C\subseteq D$ (i.e., $C\cdot F>0$),
	it holds  that $C\subseteq Z$, i.e., the horizontal part of $D$ is contained in the fixed part $Z$ of $K_{\sF}$.
	In particular, if $g(B)\geq 1$, then $\Delta_h\subseteq Z$,
	since $f^*(D_0)+\Delta_h+\Delta_v$ is an effective divisor in $|K_{\sF}|$ for any effective divisor $D_0 \in |K_B|$.
	Therefore, if $g(B)\geq 1$, then
	$$p_g(\sF)=h^0(S,K_{\sF})=h^0(S,f^*K_B+\Delta_v)=h^0\big(B,\mathcal{O}_B(K_B)\otimes f_*\mathcal{O}_S(\Delta_v)\big).$$
	According to \autoref{prop-4-3}, the coefficients $n_{ij}$'s satisfy the following property:
	if $F_i=\sum\limits_{j=1}^{r_i} m_{ij}C_{ij}$ is a fiber of $f$,
	then there exists at least one $1\leq j \leq r_i$ such that $n_{ij}<m_{ij}$.
	This means that $h^0(S,\,\Delta_v)=1$, and that	
	$$f_*\mathcal{O}_{S}(\Delta_v) = \mathcal{O}_B.$$
	Therefore,
	$$p_g(\sF)=h^0(S,K_{\sF})=h^0\big(B,\mathcal{O}_B(K_B)\otimes f_*\mathcal{O}_S(\Delta_v)\big)=h^0\big(B,\mathcal{O}_B(K_B)\big)=g(B).$$
	Since $\varphi$ induces the fibration $f:S \to B$,
	it follows that $g(B)=p_g(\sF)\geq 2$, and the decomposition \eqref{eqn-4-4} holds.		 
\end{proof}

We can now prove \autoref{thm-1-4}.
\begin{proof}[Proof of \autoref{thm-1-4}]
	As explained at the beginning of this subsection, we may replace $(S,\sF)$ by
	$(\wt S, \wt \sF)$.
	Hence we will write $\wt S$ (resp. $\wt \sF$) as $S$ (resp. $\wt \sF$) in the following.
	
	(i) The formula \eqref{eqn-1-7} is proved in \autoref{prop-4-2}.
	Note also that $d=K_{\sF}\cdot F\geq 1$ by \eqref{eqn-41-1}.
	It remains to prove \eqref{eqn-1-8}.
	Let $(S_0,\sF_0)$ be the relatively minimal model of $(S,\sF)$.
	Denote by $\tau:\,S \to S_0$ the contraction of $\sF$-exceptional curves,
	and by $K_{\sF_0}$ the canonical divisor of $\sF_0$.
	Based on the decomposition in \eqref{eqn-2-22},
	one has
	$$K_{\sF_0}=M_0+Z_0,$$
	where $M_0=\tau_*(M)$ and $Z_0=\tau_*(Z)$.
	By \eqref{eqn-4-0},
	\begin{equation}\label{eqn-41-3}
		M_0 \sim \deg(L) A,\qquad \text{with $\deg(L)\geq p_g(\sF)-1$},
	\end{equation}
	where $A=\tau_*(F)$ and $F$ is a general fiber of $f$.
	Let
	$$Z_0=\sum_{j=1}^{m}a_jD_j+Z_{0v},$$
	where $D_j$'s are all the components of $Z_0$ intersecting positively with $A$;
	namely, $d_j:=D_j\cdot A>0$ for any $1\leq j \leq m$,
	and $D\cdot A=0$ for any component $D\subseteq Z_{0v}$.
	Let
	\begin{equation}\label{eqn-41-7}
		K_{\sF_0}=P_0+N_0
	\end{equation} be the Zariski decomposition of $K_{\sF_0}$.
	It is clear that $M_0\subseteq P_0$, and hence $N_0\subseteq Z_0$.
	Let
	$$N_0=\sum_{j=1}^{m}b_jD_j+N_{0v},$$
	where the support of $N_{0v}$ is contained in that of $Z_{0v}$.
	\begin{lemma}\label{lem-41-1}
		For any $1\leq j \leq m$, it holds $b_j\leq \frac{a_j}{p_g(\sF)}$.
	\end{lemma}
\begin{proof}[Proof of \autoref{lem-41-1}]
	To prove the lemma, we may assume that $b_j\neq 0$;
	namely $D_j$ is contained in some $\sF_0$-chain $C=C_1+\cdots C_r$.
	Suppose that $D_j=C_{i_0}$, and that
	$$N_0=\sum_{i=1}^{r}y_iC_i+N',\qquad Z_0=\sum_{i=1}^{r}x_iC_i+Z',$$
	where both $N'$ and $Z'$ consist of components different from $C_i$'s.
	As $D_j=C_{i_0}$ by our assumption, it follows that $x_{i_0}=a_j$ and $y_{i_0}=b_j$.
	Let $e_i=-C_i^2 \geq 2$ for $1\leq i \leq r$.
	By \eqref{eqn-41-3},
	$$M_0\cdot C_{i_0}=M_0\cdot D_j\geq d_j\big(p_g(\sF)-1\big).$$
	It follows that
	\begin{equation}\label{eqn-41-4}
		\hspace{-5mm}\left\{\begin{aligned}
			&-1=K_{\sF_0}\cdot C_1=(M_0+Z_0)\cdot C_1\geq -e_1x_1+x_2+\delta_{1,i_0}d_j\big(p_g(\sF)-1\big),\\
			&0=K_{\sF_0}\cdot C_i=(M_0+Z_0)\cdot C_i\geq -e_ix_i+x_{i-1}+x_{i+1}+\delta_{i,i_0}d_j\big(p_g(\sF)-1\big),~\, \forall\,2\leq i \leq r,
		\end{aligned}\right.
	\end{equation}
	where we set $x_{r+1}=0$, and
	$$\delta_{i,i_0}=0\text{\,~if~\,}i\neq i_0,\qquad\text{and}\qquad \delta_{i,i_0}=0\text{\,~if~\,}i= i_0.$$
	Let
	$$\tilde x_i=\left\{\begin{aligned}
		&x_i, &\quad&\text{if~}i\geq i_0,\\
		&x_i+d_j\big(p_g(\sF)-1\big), &&\text{if~}i=i_0-1,\\
		&e_{i+1}\tilde x_{i+1}-\tilde x_{i+2}, &&\text{if~}1\leq i \leq i_0-2.
	\end{aligned}\right.$$
	Then
	\begin{equation}\label{eqn-41-6}
		-e_i\tilde x_i+ \tilde x_{i-1}+ \tilde x_{i+1}=0,\qquad \forall\,2\leq i \leq r.
	\end{equation}
	If $i_0=1$, then $\tilde x_i=x_i$ by definition, and hence by \eqref{eqn-41-4} the following inequality holds
	\begin{equation}\label{eqn-41-5}
		-p_g(\sF)\geq -1-d_j\big(p_g(\sF)-1\big) \geq -e_1\tilde x_1+\tilde x_2.
	\end{equation}
	We want to show that the above inequality also holds if $i_0\geq 2$.
	To this aim, one can first prove inductively that
	$$\tilde x_{i-1}-x_{i-1} \geq \tilde x_{i}-x_{i}, \qquad \forall\,i\geq 2.$$
	Indeed, $\tilde x_{r+1}-x_{r+1}=\tilde x_{r}-x_{r}=0$ by definition.
	Hence by the second inequality in \eqref{eqn-41-4}, for any $2\leq i \leq r$, 
	$$\begin{aligned}
		\tilde x_{i-1}-x_{i-1} &\,\geq e_i(\tilde x_i-x_i)-(\tilde x_{i+1}-x_{i+1})\\
		&\,\geq e_i(\tilde x_i-x_i)-(\tilde x_{i}-x_{i})\\
		&\,=(e_i-1)(\tilde x_i-x_i)\geq \tilde x_i-x_i,\quad\text{since~}e_i\geq 2.
	\end{aligned}$$
	Note that $\tilde x_{i_0-1}-x_{i_0-1}=d_j\big(p_g(\sF)-1\big)$ by definition.
	Hence
	$$\tilde x_1-x_1\geq \tilde x_{i_0-1}-x_{i_0-1}=d_j\big(p_g(\sF)-1\big).$$
	Combining with the first inequality in \eqref{eqn-41-4},
	$$\begin{aligned}
		e_1\tilde x_1-\tilde x_2&\,=e_1(\tilde x_1-x_1)-(\tilde x_2-x_2)+e_1x_1-x_2\\
		&\,\geq e_1(\tilde x_1-x_1)-(\tilde x_1-x_1)+1\\
		&\,=(e_1-1)(\tilde x_1-x_1)+1\geq (\tilde x_1-x_1)+1\\
		&\,\geq d_j\big(p_g(\sF)-1\big)+1 \geq p_g(\sF).
	\end{aligned}$$
	Hence in any case, the equality \eqref{eqn-41-5} holds.
	On the other hand, by the property of the Zariski decomposition,
	$$\left\{\begin{aligned}
		-1&\,=N_0\cdot C_1= -e_1y_1+y_2,\\
		0&\,=N_0\cdot C_i= -e_iy_i+y_{i-1}+y_{i+1},\quad \forall\,2\leq i \leq r,
	\end{aligned}\right.$$
	Here we set $y_{r+1}=0$.
	Combining this with \eqref{eqn-41-6} and \eqref{eqn-41-5},
	$$\left\{\begin{aligned}
		0&\,\geq  -e_1z_1+z_2,\\
		0&\,\geq  -e_iz_i+z_{i-1}+z_{i+1},\quad \forall\,2\leq i \leq r,
	\end{aligned}\right.$$
	where $z_i=\tilde x_i-y_ip_g(\sF)$ for any $i\geq 1$.
	Note that the intersection matrix of the curve $C=C_1+\cdots+C_r$ is negative definite.
	Hence by \cite[Lemma\,4.1]{km-98}, $z_i\geq 0$ for any $i\geq 1$.
	In particular,
	$$b_j=y_{i_0}\leq \frac{\tilde x_{i_0}}{p_g(\sF)} = \frac{x_{i_0}}{p_g(\sF)} =\frac{a_j}{p_g(\sF)}.$$
	This completes the proof.
\end{proof}

Come back to the proof of \autoref{thm-1-4}\,(i).
By \autoref{lem-41-1},
$$(Z_0-N_0)\cdot A =\sum_{j=1}^{m}(a_j-b_j)D_j\cdot A \geq \sum_{j=1}^{m}\frac{a_j(p_g(\sF)-1)}{p_g(\sF)}D_j\cdot A=\frac{Z_0\cdot A\,(p_g(\sF)-1)}{p_g(\sF)}.$$
Combining this with \eqref{eqn-41-3},
\begin{equation}\label{eqn-41-8}
	(Z_0-N_0)\cdot M_0 \geq \frac{(p_g(\sF)-1)\,Z_0\cdot M_0}{p_g(\sF)}.
\end{equation}
Let
$$F=\tau^*(A)-\sum s_i\mathcal{E}_i,\qquad K_{\sF}=\tau^*(K_{\sF_0})+\sum t_i\mathcal{E}_i,$$
where  $s_i\geq 0$, and $t_i=0$ or $1$, since $\sF_0$ is reduced.
In particular,
$$d= K_{\sF}\cdot F=K_{\sF_0}\cdot A+\sum t_is_i \leq K_{\sF_0}\cdot A+\sum s_i^2=K_{\sF_0}\cdot A+A^2.$$
Hence
$$\begin{aligned}
\vol(\sF)&\,=\vol(\sF_0)=P_0^2\geq P_0\cdot M_0=\big(M_0+Z_0-N_0\big)\cdot M_0\\
&\,\geq M_0^2 +\frac{(p_g(\sF)-1)\,Z_0\cdot M_0}{p_g(\sF)}\\
&\,\geq (p_g(\sF)-1)\,M_0\cdot A+\frac{(p_g(\sF)-1)^2\,Z_0\cdot A}{p_g(\sF)}\\
&\,\geq \frac{(p_g(\sF)-1)^2\,(M_0+A)\cdot A}{p_g(\sF)}+\frac{(p_g(\sF)-1)^2\,Z_0\cdot A}{p_g(\sF)}\\
&\,=\frac{(p_g(\sF)-1)^2\,(K_{\sF_0}+A)\cdot A}{p_g(\sF)}\geq \frac{d\,(p_g(\sF)-1)^2}{p_g(\sF)}
\end{aligned}$$
This proves \eqref{eqn-1-8}.

\vspace{2mm}
(ii) Let $\epsilon:\,S \to X$ be the contraction of all components of the negative part $N$ in the Zariski decomposition of $K_{\sF}$, and let $\ol K=\epsilon_*K_{\sF}$.
Since $g(B)\geq 1$, the support of $N$ is contained in fibers of $f$.
Hence there is an induced fibration $\bar f:\,X \to B$.
Since the canonical map $\varphi$ induces the fibration $f:\, S \to B$,
one write $K_{\sF}=M+Z$, where
$$M=f^*(L),\qquad\text{with~}~h^0(B,L)=p_g(\sF).$$
Let $\Delta$ be the effective divisor in \eqref{eqn-4-5},
and $\Delta_h\subseteq \Delta$ (resp. $\Delta_v\subseteq \Delta$) consists of the horizontal (resp. vertical) components.
Since $f^*(D_0)+\Delta_h+\Delta_v\in |K_{\sF}|$ for any effective divisor $D_0 \in |K_B|$, it follows that $\Delta_h\subseteq Z$ if $g(B)\geq 1$.
Hence
$$\ol K=\epsilon_*K_{\sF}=\bar f^*(L)+\epsilon_*(\Delta)=\bar f^*(L)+\sum_{j=1}^{m}a_j\ol D_j+ \ol \Delta_v,$$
where $\ol \Delta_v=\epsilon_*(\Delta_v)$ consists of curves in fibers of the induced fibration $\bar f:\,X \to B$, and $\ol D_j$ maps surjectively onto $B$ for $1\leq j \leq m$.
As $p_g(\sF)=h^0(B,L)\geq g(B)$ by (i) and $p_g(\sF)\geq 2$ by assumption,
by the Riemann-Roch theorem for curves one has
\begin{equation}\label{eqn-4-40}
	\deg(L)\geq p_g(\sF)-1+g(B).
\end{equation}
Since $g(B)\geq 1$, 
it follows that $\ol D_{j}$ is not $\sF$-invariant for any $1\leq j \leq m$.
Thus by \cite[proposition\,2.13]{cf-18} (see also \cite[Lemma III.1.1]{mcq-08}),
$$0\leq (\ol K+\ol D_{j_0}) \cdot \ol D_{j_0}=(a_{j_0}+1)\ol D_{j_0}^2+\Big(\bar f^*(L)+\sum_{j\neq j_0}a_j \ol D_j+ \ol \Delta_v\Big)\cdot \ol D_{j_0},\quad\forall\,1\leq j_0 \leq m.$$
Hence
$$\begin{aligned}
\ol K \cdot \ol D_{j_0}\,&=a_{j_0} \ol D_{j_0}^2+\Big(\bar f^*(L)+\sum_{j\neq j_0}a_j \ol D_j+ \ol \Delta_v\Big) \cdot \ol D_{j_0}\\
&\geq -\frac{a_{j_0}}{a_{j_0}+1}\Big(\bar f^*(L)+\sum_{j\neq j_0}a_j \ol D_j+ \ol \Delta_v\Big) \cdot \ol D_{j_0}+\Big(\bar f^*(L)+\sum_{j\neq j_0}a_j \ol D_j+ \ol \Delta_v\Big) \cdot \ol D_{j_0}\\
&= \frac{1}{a_{j_0}+1}\Big(\bar f^*(L)+\sum_{j\neq j_0}a_j \ol D_j+ \ol \Delta_v\Big)\cdot \ol D_{j_0} \\
&\geq \frac{\bar f^*(L)\cdot \ol D_{j_0}}{a_{j_0}+1}= \frac{\ol D_{j_0} \cdot F}{a_{j_0}+1}\,\deg(L).
\end{aligned}$$
By the Negativity Lemma, $P=\epsilon^*(\ol K)$, where $P$ is the positive part in the Zariski decomposition of $K_{\sF}$ as in \eqref{eqn-zariski}.
Therefore,
$$\begin{aligned}
\vol(\sF)\,&=P^2=\ol K^2=\ol K\cdot \Big(\bar f^*(L)+\sum_{j=1}^{m}a_j \ol D_j+ \ol \Delta_v\Big)\\
&\geq \ol K \cdot \bar f^*(L)+ \sum_{j=1}^{m} a_j\ol K\cdot \ol D_j\\
&\geq d\,\deg(L)+\sum_{j=1}^{m}\frac{a_j\ol D_{j}\cdot F}{a_{j}+1}\,\deg(L)\\
&=\frac{d(d+1)}{d+1}\,\deg(L)\geq \frac{d(d+1)}{d+1}\big(p_g(\sF)-1+g(B)\big).
\end{aligned}$$
We used the equality $\sum\limits_{j=1}^{m}a_j\ol D_{j}\cdot F=K_{\sF}\cdot F=d$ above,
and the last inequality follows from \eqref{eqn-4-40}.
This proves \eqref{eqn-1-9}.

\vspace{2mm}
(iii) Since $d=K_{\sF}\cdot F<2g(F)-2$,
by \autoref{prop-4-4}, either $g(B)=0$ or $p_g(\sF)=g(B)\geq 2$.
It remains to prove \eqref{eqn-1-10} if $p_g(\sF)=g(B)\geq 2$.
The proof is complete the same as that of \eqref{eqn-1-9} above.
It suffices to note that in this case $L=K_B$ by \eqref{eqn-4-4},
and hence
$$\deg(L)=2\big(g(B)-1\big)=2\big(p_g(\sF)-1\big).$$
This proves \eqref{eqn-1-10}.
\end{proof}

To end this subsection, we prove some related inequalities when
the foliation is different from the fibration induced by the canonical map.
\begin{lemma}\label{lem-41-2}
	Let $\sF$ be a foliation of general type on $S$ whose canonical map induces a fibration $f:\, S \to B$ as before.
	Suppose that $h^0(F,\mathcal{O}_F(K_{\sF}))=1$ for a general fiber $F$ of $f$.
	Then $\Delta_h$ is contained in the fixed part $Z$ of the complete linear system $|K_{\sF}|$,
	where $\Delta_h\subseteq \Delta$ consists of horizontal components (i.e., intersect positively with $F$),
	and $\Delta$ is the effective divisor in \eqref{eqn-4-5}.
\end{lemma}
\begin{proof}
	By \eqref{eqn-4-5}, we have decomposition
	$$K_{\sF}=f^*K_B+\Delta=f^*K_B+\Delta_h+\Delta_v,$$
	where $\Delta_v\subseteq \Delta$ consists of the vertical components,
	i.e., $C\cdot F=0$ for any component $C\subseteq \Delta_v$.
	Note that $D|_F \in \big|K_{\sF}|_F\big|$ for any effective divisor $D\in |K_{\sF}|$.
	Since $h^0(F,\mathcal{O}_F(K_{\sF}))=1$ for a general fiber $F$ of $f$,
	it follows that $\Delta_h\subseteq Z$ as required.
\end{proof}

\begin{remark}\label{rem-41-1}
	Suppose that $h^0(F,\mathcal{O}_F(K_{\sF}))=1$. Then $g(F)\geq 1$, since $d=K_{\sF}\cdot F \geq 1$ by \eqref{eqn-41-1}.
	If $g(F)\geq 3$, then by the Riemann-Roch theorem, one checks easily that
	$d<2g(F)-2$. Hence \autoref{prop-4-3} and \autoref{prop-4-4} apply in this case. 
	If $g(F)=1$, then $d=1$ again by the Riemann-Roch theorem.
\end{remark}

\begin{corollary}\label{cor-41-1}
	Let $\sF$ be a foliation of general type on $S$ whose canonical map induces a fibration $f:\, S \to B$ as before.
	Suppose that $h^0(F,\mathcal{O}_F(K_{\sF}))=1$ for a general fiber $F$ of $f$.
	\begin{enumerate}[(i)]
		\item If $g(F)\geq 3$, then $p_g(\sF)=g(B)\geq 2$ and
		\begin{equation}\label{eqn-41-9}
			\vol(\sF) \geq \frac{2d(d+2)}{d+1}\big(p_g(\sF)-1\big).
		\end{equation}
		\item If $g(F)=2$, then $d=K_{\sF}\cdot F=1$ or $2$, and
		\begin{equation}\label{eqn-41-9-1}
			\vol(\sF) \geq \left\{\begin{aligned}
			&3\big(p_g(\sF)-1\big), &\quad&\text{if~}d=1;\\[2pt]
			&\frac83\big(p_g(\sF)-1\big), &\quad&\text{if~}d=2.
			\end{aligned}\right.
		\end{equation}
		\item If $g(F)=1$, then $d=K_{\sF}\cdot F=1$ and
		\begin{equation}\label{eqn-41-10}
			\vol(\sF) \geq \frac{3}{2}\big(p_g(\sF)-1\big).
		\end{equation}
	\end{enumerate}
\end{corollary}
\begin{proof}
	(i) By \autoref{lem-41-2}, one has
	$$p_g(\sF)=h^0(S,K_{\sF})=h^0(S,f^*K_B+\Delta_v)=h^0\big(B,\mathcal{O}_B(K_B)\otimes f_*\mathcal{O}_S(\Delta_v)\big).$$
	If $g(F)\geq 3$, then $d=K_{\sF}\cdot F < 2g(F)-2$ by \autoref{rem-41-1}.
	Hence according to \autoref{prop-4-3}, \autoref{prop-4-4} and their proofs,
	$f_*\mathcal{O}_S(\Delta_v) =\mathcal{O}_B$.
	Therefore,
	$$p_g(\sF)=h^0(S,K_{\sF})=h^0\big(B,\mathcal{O}_B(K_B)\otimes f_*\mathcal{O}_S(\Delta_v)\big)=h^0\big(B,\mathcal{O}_B(K_B)\big)=g(B).$$
	In particular, $p_g(\sF)=g(B)\geq 2$,
	and \eqref{eqn-41-9} holds by \autoref{thm-1-4}\,(ii).
	
	(ii) Since $g(F)=2$, according to the Riemann-Roch theorem, $d=K_{\sF}\cdot F=1$ or $2$.
	If $d=1$, then $d<2g(F)-2=2$. Hence by a similar argument as in (i), one shows that
	$$\vol(\sF) \geq \frac{2d(d+2)}{d+1}\big(p_g(\sF)-1\big)=3\big(p_g(\sF)-1\big).$$
	Next we assume that $d=2$.
	Note that any component in the horizontal part $\Delta_h$ is not $\sF$-invariant by \autoref{prop-4-2}.
	Hence by a similar argument as the proof of \eqref{eqn-1-9} in \autoref{thm-1-4},
	one proves that
	$$\vol(\sF) \geq \frac{d(d+2)}{d+1}\big(p_g(\sF)-1\big)=\frac{8}{3}\big(p_g(\sF)-1\big).$$
	We leave the details to the readers.
	
	(iii) If $g(F)=1$, then $d=K_{\sF}\cdot F=1$ by \autoref{rem-41-1},
	i.e., $\Delta_h=C$ is a section of $f$.
	Hence
	$$K_{\sF}=f^*K_B+C+\Delta_v.$$
	Moreover, $\Delta_h=C$ is not $\sF$-invariant by \autoref{prop-4-2}.
	Hence again by a similar argument as the proof of \eqref{eqn-1-9} in \autoref{thm-1-4},
	one proves that
	\[\vol(\sF) \geq \frac{d(d+2)}{d+1}\big(p_g(\sF)-1\big)=\frac{3}{2}\big(p_g(\sF)-1\big).\]
	This completes the proof.
\end{proof}

A special case which satisfies the assumption that $h^0(F,\mathcal{O}_F(K_{\sF}))=1$ is when $g(F)\geq 1$ and $d=K_{\sF}\cdot F=1$.
\begin{corollary}\label{cor-41-2}
	Let $\sF$ be a foliation of general type on $S$ whose canonical map induces a fibration $f:\, S \to B$ as before.
	Suppose that $g(F)\geq 1$ and $d=K_{\sF}\cdot F=1$ for a general fiber $F$ of $f$.
	\begin{enumerate}[(i)]
		\item If $g(F)\geq 2$, then
		\begin{equation*}
		\vol(\sF) \geq 3\big(p_g(\sF)-1\big).
		\end{equation*}
		\item If $g(F)=1$, then
		\begin{equation*}
		\vol(\sF) \geq \frac{3}{2}\big(p_g(\sF)-1\big).
		\end{equation*}
	\end{enumerate}
\end{corollary}

\subsection{Examples of foliations with canonical fibrations}\label{sec-exam-4}
In this subsection, we construct examples of foliations whose canonical maps induce fibrations.
\autoref{exam-6-4} provides a reduced foliated surface $(S,\sF)$ of general type
whose canonical map induces a fibration $f:\, S \to B$ such that $\sF$ is the same as the foliation $\mathcal{G}$ defined by taking the saturation of $\ker(\df:\,T_{S} \to f^*T_{B})$ in $T_{S}$;
while \autoref{exam-6-5} provides a reduced foliated surface $(S,\sF)$ of general type
whose canonical map induces a fibration $f:\, S \to B$ with $\sF\neq \mathcal{G}$.
\begin{example}\label{exam-6-4}
	For any genus $g\geq 2$, we construct a foliated surface $(S,\sF)$ of general type such that
	\begin{enumerate}[(i)]
		\item $\vol(\sF)=2(g-1)gn$ and $p_g(\sF)=n$ for some even integer $n\geq 2$;
		\item the canonical map $\varphi$ induces a fibration
		$f:\,S \to B$;
		\item the foliation $\sF$ is the same as the foliation defined by taking the saturation of $\ker(\df:\,T_{S} \to f^*T_{B})$ in $T_{S}$.
	\end{enumerate} 
\end{example}

Let $n\geq 2$ be an even number, and $Y=\mathbb{P}_{\mathbb{P}^1}\big(\mathcal{O}_{\mathbb{P}^1} \oplus \mathcal{O}_{\mathbb{P}^1}(n)\big)$ be the Hirzebruch surface with the ruling $h:\,Y \to \mathbb{P}^1$.
Denote by $\Gamma$ (resp. $C_0$) a general fiber of $h$ (resp. the unique section of $h$ with $C_0^2=-n$).
Then the linear systemn $|C_0+n\Gamma|$ is base-point-free, cf. \cite[\S\,V.2]{har-77}.
Let
$$C_1,\cdots,C_{2g+1} \in |C_0+n\Gamma|,$$
be $(2g+1)$ general elements.
Note that none of $C_i$'s intersects $C_0$ for $1\leq i \leq 2g+1$, and we may assume that the divisor $R_0=C_0+C_1+\cdots+C_{2g+1}$ is normal crossing.
Let $\Gamma_1,\cdots,\Gamma_m$ be $m$ fibers of $g$, none of which passing through the singularities of $R_0$,
where $m=(2g-3)n+2$.
Then the divisor
$$R:=R_0+\Gamma_1+\cdots+\Gamma_m ~\sim~ 2\Big((g+1)C_0+\big((2g-1)n+1\big)\Gamma\Big).$$
Hence it defines a double cover $\pi_0:\, S_0 \to Y$ branched exactly over $R=R_0+\Gamma_1+\cdots+\Gamma_m$.
Let $\rho:\,S \to S_0$ be the minimal desingularization, and $f:\,S \to \bbp^1$ be the induced fibration:
$$\xymatrix{S \ar[rr]_-{\rho} \ar@/^4mm/"1,5"^-{\pi} \ar[drr]_-{f} && S_0 \ar[d] \ar[rr]_-{\pi_0} && Y \ar[dll]^-{h}\\
	&&\bbp^1 &&}$$
The fibration $f$ induces a foliation $\sF$ on $S$ by taking the saturation of $\ker(\df:\,T_{S} \to f^*T_{\bbp^1})$ in $T_{S}$.
We want to show that the foliation $\sF$ satisfies our requirements.

By the construction of the branch divisor $R$ above,
it has only nodes as its singularities.
It follows that
$$K_{S/\bbp^1}=\pi^*(K_{Y/\bbp^1}+R/2)\sim \pi^*\Big((g-1)C_0+\big((g-1/2)n+m/2\big)\Gamma\Big).$$
Moreover, by the construction,
all the fibers of $f$ are semi-stable except these fibers over $p_i=h(\Gamma_i)$'s,
and the fiber $F_i=f^{-1}(p_i)$ has the following dual graph for any $1\leq i \leq m$:
{\upshape
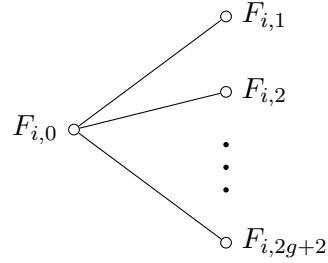
	\begin{center}
		\begin{tikzpicture}
			[  place/.style={circle,draw,inner sep=0.5mm},
			place2/.style={circle,draw,fill,inner sep=0.2mm},]
			
			\node[place] (v1) at (0,1.5)  [label=left:{$F_{i,0}$}]  {};
			\node[place] (v2)  at (2,0)   {};
			\node[place] (v3) at (2,2)    {};
			\node[place] (v4) at (2,3)    {};
			
			\node[place2]   at (2,1) {};
			\node[place2]   at (2,0.7) {};
			\node[place2] at (2,1.3)    {};
			
			\node[right=2pt] at (v2)  { ${F_{i,2g+2}}$};
			\node[right=2pt] at (v3)   { ${F_{i,2}}$};
			\node[right=2pt] at (v4)   {${F_{i,1}}$};
			
			\draw (v1)--(v2); \draw (v1)--(v3); \draw (v1)--(v4);
		\end{tikzpicture}
\end{center}}
\vspace{-2mm}
\noindent
In the above, $F_{i,0}$ is of multiplicity two, and the other $F_{i,j}$'s are all of multiplicity one for $1\leq j \leq 2g+2$.
Moreover, $F_{i,0}^2=-(g+1)$ and $F_{i,j}^2=-2$ for $1\leq j \leq 2g+2$.
It follows that
$$K_{\sF}=K_{S/\bbp^1}-\sum_{i=1}^{m}F_{i,0}.$$
One checks directly that the Zariski decomposition of $K_{\sF}$ is as follows,
$$K_{\sF}=\bigg(K_{S/\bbp^1}-\sum_{i=1}^{m}\Big(F_{i,0}+\sum_{j=1}^{2g+2}\frac{F_{i,j}}{2}\Big)\bigg)+\frac12\sum_{i=1}^{m}\sum_{j=1}^{2g+2}F_{i,j}.$$
Note that
$$P:=K_{S/\bbp^1}-\sum_{i=1}^{m}\Big(F_{i,0}+\sum_{j=1}^{2g+2}\frac{F_{i,j}}{2}\Big)=K_{S/\bbp^1}-\sum_{i=1}^{m}\frac{F_i}{2}
\sim \pi^*\Big((g-1)C_0+(g-1/2)n\Gamma\Big),$$
where $F_i=f^{-1}(p_i)$ is the fiber over $p_i$.
Hence
$$\vol(\sF)=2\Big((g-1)C_0+(g-1/2)n\Gamma\Big)^2=2(g-1)gn.$$
We next compute the space $H^0(S,K_{\sF})$.
Since each $F_{i,j}$ is contained in the support of the negative part of $K_{\sF}$,
it is contained in the fixed part of $|K_{\sF}|$,
for any $1\leq i \leq m$ and $1\leq j \leq 2g+2$.
Hence
$$|K_{\sF}|=\Big|K_{\sF}-\sum_{i=1}^{m}\sum_{j=1}^{2g+2}F_{i,j}\Big|+\sum_{i=1}^{m}\sum_{j=1}^{2g+2}F_{i,j}
=\Big|K_{S/\bbp^1}-\sum_{i=1}^{m}\sum_{j=0}^{2g+2}F_{i,j}\Big|+\sum_{i=1}^{m}\sum_{j=1}^{2g+2}F_{i,j}.$$
For any $1\leq i \leq m$,
$$F_{i,0}\cdot \Big(K_{S/\bbp^1}-\sum_{i=1}^{m}\sum_{j=0}^{2g+2}F_{i,j}\Big)=
F_{i,0}\cdot K_{S/\bbp^1} -F_{i,0}^2-\sum_{j=1}^{2g+2}F_{i,0}\cdot F_{i,j}=-2<0.$$
It follows that each $F_{i,0}$ is contained in the fixed part of the complete system
$\Big|K_{S/\bbp^1}-\sum\limits_{i=1}^{m}\sum\limits_{j=0}^{2g+2}F_{i,j}\Big|$
for $1\leq i \leq m$.
Therefore,
$$\begin{aligned}
	|K_{\sF}|\,&=\Big|K_{S/\bbp^1}-\sum_{i=1}^{m}\sum_{j=0}^{2g+2}F_{i,j}\Big|+\sum_{i=1}^{m}\sum_{j=1}^{2g+2}F_{i,j}\\
	&=\Big|K_{S/\bbp^1}-\sum_{i=1}^{m}\sum_{j=0}^{2g+2}F_{i,j}-\sum_{i=1}^{m}F_{i,0}\Big|+\sum_{i=1}^{m}\sum_{j=1}^{2g+2}F_{i,j}+\sum_{i=1}^{m}F_{i,0}\\
	&=\Big|K_{S/\bbp^1}-\sum_{i=1}^{m}F_i\Big|+\sum_{i=1}^{m}\sum_{j=0}^{2g+2}F_{i,j}\\
	&=\Big|\pi^*\Big((g-1)C_0+\big((g-1/2)n-m/2\big)\Gamma\Big)\Big|+\sum_{i=1}^{m}\sum_{j=0}^{2g+2}F_{i,j}\\
	&=\Big|\pi^*\Big((g-1)C_0+(n-1)\Gamma\Big)\Big|+\sum_{i=1}^{m}\sum_{j=0}^{2g+2}F_{i,j}.
\end{aligned}$$
Note that
$$\pi_*\mathcal{O}_S=(\pi_0)_*\mathcal{O}_{S_0}=\mathcal{O}_Y \oplus \mathcal{O}_Y\Big(-(g+1)C_0-\big((2g-1)n+1\big)\Gamma\Big).$$
Hence
$$\begin{aligned}
	p_g(\sF)=h^0(S,K_{\sF})&\,=h^0\Big(S,\mathcal{O}_S\Big(\pi^*\big((g-1)C_0+(n-1)\Gamma\big)\Big)\\
	&\,=h^0\Big(Y,\mathcal{O}_Y\big((g-1)C_0+(n-1)\Gamma\big)
	\otimes\pi_*\mathcal{O}_S\Big)\\
	&\,=h^0\Big(Y,\mathcal{O}_Y\big((g-1)C_0+(n-1)\Gamma\big)\Big)\\
	&\,=h^0\Big(Y,\mathcal{O}_Y\big((n-1)\Gamma\big)\Big)=n.
\end{aligned}$$
In particular,
$$\begin{aligned}
	|K_{\sF}|&\,=\Big|\pi^*\Big((g-1)C_0+(n-1)\Gamma\Big)\Big|+\sum_{i=1}^{m}\sum_{j=0}^{2g+2}F_{i,j}\\
	&\,=\pi^*\big|(g-1)C_0+(n-1)\Gamma\big|+\sum_{i=1}^{m}\sum_{j=0}^{2g+2}F_{i,j}\\
	&\,=\pi^*\big|(n-1)\Gamma\big|+\pi^*\big((g-1)C_0\big)+\sum_{i=1}^{m}\sum_{j=0}^{2g+2}F_{i,j}.
\end{aligned}$$
This means that the canonical map factorizes through $\pi$ as $\varphi=\psi\circ\pi$,
where $\psi:=\varphi_{|(n-1)\Gamma|}$ is the map defined by the complete linear system $|(n-1)\Gamma|$, which is the same as the fibration $h:\,Y \to \bbp^1$.
In other words, the canonical map $\varphi$ induces the fibration $f:\,S \to \bbp^1$ as required.

\begin{example}\label{exam-6-5}
	In this example, we construct a foliated surface $(S,\sF)$ of general type,
	whose canonical map $\varphi$ induces a fibration $f:\, S \to B$
	such that
	\begin{enumerate}[(i)]
		\item $\sF$ is algebraically integrable, and $S$ is also of general type;
		\item $g(B)\geq 2$, $g(F)=g\geq 2$, and (where $F$ is a general fiber of $f$);
		\item $\sF$ is different from the foliation defined by taking the saturation of $\ker(\df:\,T_{S} \to f^*T_{B})$ in $T_{S}$;
		\item $K_{\sF}\cdot F=1$, and
		\begin{equation}\label{eqn-6-5}
				\vol(\sF) = 4p_g(\sF)-4=4g(B)-4.
		\end{equation}
	\end{enumerate}

	\vspace{2mm}
	Let $B$ be a curve of genus $g(B)\geq 2$, and $\psi:\,B \to \bbp^1$ is a finite cover of degree $2m$ with only simply ramified points, i.e., the ramification indices are all equal to $2$.
	Let $\hat{Y}=\bbp^1 \times B$, with two projections $\hat{\gamma}_1,\hat{\gamma}_2$.
	The curve $B$ can be embedded into $\hat{Y}$ by mapping $x\in B$ to $(\psi(x),x)\in \hat{Y}$ as a section of $\hat{\gamma}_2$:
	$$\xymatrix{B\, \ar@{^(->}[rr]^{i} \ar[drr]_-{\psi} &&\hat{Y} \ar[rr]^-{\hat{\gamma}_2} \ar[d]^-{\hat{\gamma}_1} && B\\
		&&\bbp^1}$$
	
	Denote by $\hat{D}=i(B)\subseteq \hat{Y}$, and let $\hat\Gamma_i=\hat{\gamma}_1^{-1}(p_i)$'s ($1\leq i \leq 2g+1$) be $2g+1$ general fibers of $\hat{\gamma}_1$, where $g\geq 0$.
	As $\hat{Y}=\bbp^1 \times B$ is a ruled surface, one checks easily that
	$\hat{R}=\hat{D}+\sum\limits_{i=1}^{2g+1}\hat\Gamma_i$ is an even divisor, and hence one can construct a double cover
	$\hat{\pi}:\, \hat{S} \to \hat{Y}$ branched exactly over $\hat{R}$.
	The surface $\hat{S}$ is singular.
	In fact, each $\hat\Gamma_i$ intersects $\hat{D}$ transversely at $2m$ points, since $\hat\Gamma_i$ is general.
	One can perform a canonical resolution to resolve the singularities on $\hat{S}$, cf. \cite[\S\,V.22]{bhpv}.
	Let $\sigma:\, Y \to \hat{Y}$ be the birational map by blowing up these intersection points $\hat{D}\bigcap \sum\limits_{i=1}^{2g+1}\hat\Gamma_i$, and let $\pi:\,S \to Y$ be the induced double cover.
	Then $S$ would be a smooth surface with following diagram.
	$$
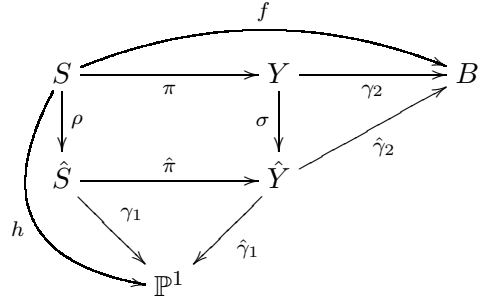
\xymatrix{S \ar@/^6mm/"1,5"^-{f} \ar[d]^-{\rho} \ar@/_12mm/"3,2"_-{h} \ar[rr]_-{\pi} && Y \ar[d]_-{\sigma} \ar[rr]_-{\gamma_2}&&B\\
		\hat{S} \ar[rr]^-{\hat{\pi}} \ar[dr]^-{\gamma_1} && \hat{Y}  \ar[rru]_-{\hat{\gamma}_2} \ar[dl]^-{\hat{\gamma}_1} \\
		&\bbp^1}$$		
	The fibration $h:\,S \to \bbp^1$ defines a foliation $\sF$ on $S$ by
	taking the saturation of $\ker(\mathrm{d}h:\,T_{S} \to h^*T_{\bbp^1})$ in $T_{S}$.
	We want to check that the foliation $\sF$ satisfies our requirements.
	
	\vspace{2mm}	
	Let $D=\sigma_*^{-1}(\hat D)$ and $C=(\sigma\circ\pi)_*^{-1}(\hat D)$ be the strict transforms of $\hat{D}$ in $Y$ and $S$ respectively.
	Denote by $\Gamma_i\subseteq Y$ (resp. $\Delta_i\subseteq S$) the strict transform of $\hat\Gamma_i$ in $Y$ (resp. $S$).
	Let $\mathcal{E}_i$ be the union of exceptional curves intersecting $\Gamma_i$
	($\mathcal{E}_i$ consists of $2m$ components).
	Then
	$$K_Y=\sigma^*(K_{\hat{Y}})+\sum_{i=1}^{2g+1}\mathcal{E}_i,\qquad
	\Gamma_i=\sigma^*(\hat\Gamma_i)-\mathcal{E}_i,\qquad
	D=\sigma^*(\hat{D})-\sum_{i=1}^{2g+1}\mathcal{E}_i.$$
	Hence
	\begin{equation}\label{eqn-6-7}
		K_S=\pi^*(K_Y)+C+\sum_{i=1}^{2g+1}\Delta_i=(\sigma\circ\pi)^*\Big(K_{\hat{Y}}+\frac{1}{2}\big(\hat{D}+\sum\limits_{i=1}^{2g+1}\hat\Gamma_i\big)\Big).
	\end{equation}
	By construction, any singular fiber of $h$ is normal crossing.
	Hence the foliation $\sF$ defined by
	taking the saturation of $\ker(\mathrm{d}h:\,T_{S} \to h^*T_{\bbp^1})$ in $T_{S}$ is reduced.
	Moreover, the curves $\Delta_i$'s are of multiplicity equal to two in fibers of $h$,
	and all the other components in fibers of $h$ are of multiplicity one.
	It follows that
	$$\begin{aligned}
		K_{\sF}=K_{S/\bbp^1}-\sum_{i=1}^{2g+1}\Delta_i=\pi^*(K_{Y/\bbp^1})+C
		&\,=(\sigma\circ\pi)^*(K_{\hat{Y}/\bbp^1})+\sum_{i=1}^{2g+1}\ol{\mathcal{E}}_i+C\\
		&\,=(\hat \gamma_2\circ\sigma\circ\pi)^*(K_B)+\sum_{i=1}^{2g+1}\ol{\mathcal{E}}_i+C\\
		&\,=f^*K_B+\sum_{i=1}^{2g+1}\ol{\mathcal{E}}_i+C,
	\end{aligned}$$
	where $\ol{\mathcal{E}}_i\subseteq S$ is the strict transform of $\mathcal{E}_i$.
	Let $F$ be a general fiber $f=f_2:\,S \to B$.
	Then $\pi(F)\cong \bbp^1$ and $F$ is double cover of $\pi(F)$ branched over $2g+2$ points
	(the branched divisor is $\big(D+\sum\limits_{i=1}^{2g+1}\Gamma_i\big) \cap \pi(F)$).
	By the Hurwitz formula, one obtains
	$$g(F)=g.$$
	Moreover,
	$$K_{\sF} \cdot F=\Big(f^*K_B+\sum_{i=1}^{2g+1}\ol{\mathcal{E}}_i+C\Big)\cdot F = C \cdot F=1.$$
	Since $C\cdot F=1$ and $g(F)=g\geq 2$, the curve $C$ must be contained in the fixed part of the linear system $|K_{\sF}|$.
	It follows that
	$$p_g(\sF)=h^0(S,K_{\sF})=h^0\Big(S,f^*K_B+\sum_{i=1}^{2g+1}\ol{\mathcal{E}}_i\Big)=h^0(B,K_B)=g(B).$$
	In fact, one proves moreover that
	$$|K_{\sF}| = |f^*K_B|+\Delta,$$
	where $\Delta=\sum\limits_{i=1}^{2g+1}\ol{\mathcal{E}}_i+C$ is the fixed part of $|K_{\sF}|$.
	In particular, the canonical map $\varphi=\varphi_{|K_{\sF}|}$ is nothing but the same as the map
	$\varphi_{|f^*K_B|}$ defined by the linear system $|f^*K_B|$.
	Hence the canonical map $\varphi$ induces the fibration $f:\,S \to B$ as required.
	It remains to check the equality
	\begin{equation}\label{eqn-6-6}
		\vol(\sF) = 4g(B)-4.
	\end{equation}
	To this purpose, we first compute the Zariski decomposition of $K_{\sF}$.
	By construction, every irreducible component in $\sum\limits_{i=1}^{2g+1}\ol{\mathcal{E}}_i$ is a smooth rational curve with self-intersection $-2$.
	Moreover,
	$$(\sigma\circ\pi)^*(\hat D)=\pi^*\Big(D+\sum_{i=1}^{2g+1}\mathcal{E}_i\Big)=2C+\sum_{i=1}^{2g+1}\ol{\mathcal{E}}_i.$$
	Hence
	$$K_{\sF}=f^*K_B+\sum_{i=1}^{2g+1}\ol{\mathcal{E}}_i+C=P+N,$$
	where
	$$P=(\sigma\circ\pi)^*\Big(\hat \gamma_2^*(K_B)+\frac{\hat D}{2}\Big),\qquad N=\frac{1}{2}\sum_{i=1}^{2g+1}\ol{\mathcal{E}}_i.$$
	One checks easily that this is the Zariski decomposition of $K_{\sF}$.
	Hence
	$$\vol(\sF)=P^2=2\Big(\hat \gamma_2^*(K_B)+\frac{\hat D}{2}\Big)^2=2(2g(B)-2).$$
	This proves \eqref{eqn-6-6}.
	Finally, By \eqref{eqn-6-7}, one computes that
	$$K_S^2=4(g-1)\big(2g(B)-2+m\big)>0.$$
	In particular, $S$ is of general type.
\end{example}

\begin{remark}\label{rem-4-1}
	(i) By a suitable base change of the above example, one can
	one obtains a foliated surface $(S_d,\sF_d)$ whose canonical map induces
	a fibration $f_d:\, S_d \to B$ such that
	$\sF_d$ is different from the foliation defined by taking the saturation of $\ker\big(\df_d:\,T_{S_d} \to f_d^*(T_{B})\big)$ in $T_{S_d}$,
	and that $K_{\sF_d}\cdot F_d=d$ for any $d\geq 2$,
	where $F_d$ is a general fiber of $f_d$.
	
	Indeed, let $\call=\mathcal{O}_{\bbp^1}(\delta_d)$ be a line bundle with $\delta_d$ sufficiently large, and $R_d\in |\mathcal{L}^{\otimes d}|$ be a reduced divisor
	such that the fibers of $h$ over $R_d$ are all smooth.
	Then one can construct a cyclic cover $\psi_d:\,C_d \to \bbp^1$ branched exactly over $R_d$.
	Let $S_d=S \times_{\bbp^1} C_d$ be the fiber-product with the following diagram
	$$\xymatrix{
		S_d \ar@/^6mm/"1,5"^-{f_d} \ar[rr]_-{\Psi_d} \ar[d]_-{h_d} && S \ar[d]^-{h} \ar[rr]_-{f} && B\\
	C_d \ar[rr]^-{\psi_d} && \bbp^1}$$
Since the foliation $\sF$ is induced by the fibration $h:\, S \to \bbp^1$,
it follows that the pulling-back foliation $\sF_d=\Psi_d^*\sF$ is still reduced with
$K_{\sF_d}=\Psi_d^*(K_{\sF})$.
Hence
$$\vol(\sF_d)=d\vol(\sF),\qquad K_{\sF_d} \cdot F_d=d,\quad\text{where $F_d$ is a general fiber of $f_d$}.$$
Moreover, by the construction,
$$p_g(\sF_d)=h^0(S_d,K_{\sF_d})=h^0\big(S,K_{\sF}\otimes (\Psi_d)_*\mathcal{O}_{S_d}\big)=h^0(S,K_{\sF})=p_g(\sF).$$
In other words, the canonical map $\varphi_{|K_{\sF_d}|}=\varphi_{|K_{\sF}|}\circ \Psi_d$,
and hence it induces the fibration $f_d=f\circ \Psi_d:\, S_d \to B$ as required.

\vspace{2mm}
(ii) The foliated surfaces in \cite[Example\,9.2]{lt-24} provide examples reaching the equality in \eqref{eqn-1-8} for $d=K_{\wt\sF}\cdot F=1$.
\end{remark}

\section{The Noether type inequalities}\label{sec-noether}
The main purpose of this section is to prove the Noether type inequalities for a foliated surface of general type as stated in \autoref{thm-1-5}.
\begin{proof}[Proof of \autoref{thm-1-5}]
	(i) This has been proved in \cite[Theorem\,1.9(i)]{lt-24}.
	One can also deduce it from our study on the canonical map of a foliated
	surface of general type.
	Indeed, when the canonical map $\varphi$ induces a fibration, then
	by \eqref{eqn-1-6} and \eqref{eqn-1-8},
	$$\vol(\sF)\geq \frac{\big(p_g(\sF)-1\big)^2}{p_g(\sF)}=p_g(\sF)-2+\frac{1}{p_g(\sF)}>p_g(\sF)-2.$$
	When the canonical map $\varphi$ is generically finite,
	then by \eqref{eqn-1-3},
	$$\vol(\sF) \geq p_g(\sF)-2.$$
	Moreover, if the equality holds, then $\deg(\varphi)=1$, and
	the image $\Sigma$ is a surface of minimal degree in $\bbp^{p_g(\sF)-1}$.
	
	\vspace{2mm}
	(ii) Consider first the case when the canonical map $\varphi$ is generically finite.
	If $\deg(\varphi)\geq 2$, then by \eqref{eqn-1-3},
	$$\vol(\sF) \geq 2\big(p_g(\sF)-2\big).$$
	If $\deg(\varphi)=1$, then $(S,\sF)$ is a canonical foliated surface.
	As $S$ is not ruled, it follows from \eqref{eqn-1-5} that
	$$\vol(\sF) \geq 2p_g(\sF)-4.$$
	
	Consider next the case when the canonical map $\varphi$ induces a fibration $f:\, \wt S \to B$.
	Let $\sigma:\,\wt S \to S$ be the induced birational map as above \autoref{thm-1-3}.
	Since $S$ is not ruled, the genus $g(F)\geq 1$, where $F$ is a general fiber of $f$.
	If the pulling-back foliation $\wt \sF=\sigma^*\sF$ coincides with the foliation by taking the saturation of $\ker(\df:\,T_{\wt S} \to f^*T_{B})$ in $T_{\wt S}$,
	then $g(F)\geq 2$ since $\sF$ is of general type, and by \eqref{eqn-1-6},
	$$\vol(\sF)\geq \frac{4\big(g(F)-1\big)g(F)}{2g(F)-1}\big(p_g(\sF)-1\big)\geq \frac{8}{3}\big(p_g(\sF)-1\big).$$
	Suppose next that $\wt \sF=\sigma^*\sF$ is different from the foliation by taking the saturation of $\ker(\df:\,T_{\wt S} \to f^*T_{B})$ in $T_{\wt S}$.
	If $d=K_{\sF}\cdot F\geq 2$,  then by \eqref{eqn-1-8},
	$$\vol(\sF)\geq \frac{2\big(p_g(\sF)-1\big)^2}{p_g(\sF)}=2p_g(\sF)-4+\frac{2}{p_g(\sF)}>2p_g(\sF)-4.$$
	If $d=K_{\sF}\cdot F=1$,  then by \autoref{cor-41-2},
	$$\vol(\sF) \geq \frac{3}{2}\big(p_g(\sF)-1\big).$$
	This completes the proof of \eqref{eqn-1-13}.
	
	\vspace{2mm}
	(iii) Suppose first that the canonical map $\varphi$ induces a fibration $f:\, \wt S \to B$.
 	Let $\sigma:\,\wt S \to S$ be the induced birational map as above \autoref{thm-1-3}.
	Since the surface $S$ is of general type, the genus $g(F)\geq 2$, where $F$ is a general fiber of $f$.
	If the pulling-back foliation $\wt \sF=\sigma^*\sF$ coincides with the foliation by taking the saturation of $\ker(\df:\,T_{\wt S} \to f^*T_{B})$ in $T_{\wt S}$,
	then by \eqref{eqn-1-6},
	$$\vol(\sF)\geq \frac{4\big(g(F)-1\big)g(F)}{2g(F)-1}\big(p_g(\sF)-1\big)\geq \frac{8}{3}\big(p_g(\sF)-1\big)>2p_g(\sF)-4.$$
	Suppose next that $\wt \sF=\sigma^*\sF$ is different from the foliation by taking the saturation of $\ker(\df:\,T_{\wt S} \to f^*T_{B})$ in $T_{\wt S}$.
	If $d=K_{\sF}\cdot F\geq 2$,  then by \eqref{eqn-1-8},
	$$\vol(\sF)\geq \frac{2\big(p_g(\sF)-1\big)^2}{p_g(\sF)}=2p_g(\sF)-4+\frac{2}{p_g(\sF)}>2p_g(\sF)-4.$$
	If $d=K_{\sF}\cdot F=1$,  then by \autoref{cor-41-2}\,(i),
	$$\vol(\sF) \geq 3\big(p_g(\sF)-1\big)>2p_g(\sF)-4.$$
	
	Consider next the case when the canonical map $\varphi$ is generically finite.
	If $\deg(\varphi)=1$, then $(S,\sF)$ is a canonical foliated surface.
	As $S$ is of general type, by \autoref{thm-1-2}, one has
	$$\vol(\sF) > 2p_g(\sF)-4.$$
	If $\deg(\varphi)\geq 2$, then by \eqref{eqn-1-3},
	$$\vol(\sF)\geq \deg(\varphi)\big(p_g(\sF)-2\big)\geq 2\big(p_g(\sF)-2\big).$$
	Moreover, if the equality holds, then $\deg(\varphi)= 2$, and
	the image $\Sigma$ is a surface of minimal degree in $\bbp^{p_g(\sF)-1}$ by \autoref{thm-1-1}.
	This completes the proof.
\end{proof}

As mentioned in \autoref{rem-1-1}, the inequalities \eqref{eqn-1-12} and \eqref{eqn-1-14} are both sharp.
However, it is not clear whether \eqref{eqn-1-13} is sharp or not.
In fact, we can improve \eqref{eqn-1-13} if $\sF$ is algebraically integrable.
To this aim, we first show the following lemma, which is an analogy to \autoref{lem-4-1}.

\begin{lemma}\label{lem-5-1}
	Let $\sF$ be a reduced foliation of general type on $S$ whose canonical map induces a fibration $f:\, \wt S \to B$ after a sequence of possible blowing-ups $\sigma:\,\wt S \to S$.
	Suppose that $\sF$ is algebraically integrable and that the induced foliation $\wt\sF=\sigma^*\sF$ is different from the foliation defined by taking the saturation of $\ker(\df:\,T_{\wt S} \to f^*T_{B})$ in $T_{\wt S}$.
	Then for any fiber $F_0$ of $f$, there is at least one component $C\subseteq F_0$ which is not $\sF$-invariant.
\end{lemma}
\begin{proof}
	Since $\sF$ is reduced and algebraically integrable,
	so is $\wt \sF$.
	Hence there is a fibration $\tilde f:\, \wt S \to \wt B$, such that
	every fiber of $\tilde f$ is normal crossing, and that $\wt \sF$ coincides with
	the foliation defined by taking the saturation of $\ker(d\tilde f:\,T_{\wt S} \to \tilde f^*T_{\wt B})$ in $T_{\wt S}$.
	As $\wt \sF$ is algebraically integrable, the $\wt \sF$-invariant curves are just those curves contained in fibers of $\tilde f$.
	Suppose there is some fiber $F_0$ of $f$, such that every component of $F_0$ is $\wt \sF$-invariant.
	Then $F_0$ is contained in some fiber $\wt F_0$ of $\tilde f$.
	This implies that $\tilde f$ contracts $F_0$ to a point, and hence every fiber of $f$ would be contracted to a point by $\tilde f$ by the rigidity lemma, cf. \cite[Lemma\,1.6]{km-98}.
	This is impossible, since the two fibrations  $f$ and $\tilde f$ are different from each other by assumption.
\end{proof}

Along the ideas of \autoref{prop-4-3}, \autoref{prop-4-4} and \autoref{lem-41-2}, based on the above lemma, one obtains the following.
\begin{corollary}\label{cor-5-1}
	Let $\sF$ be a reduced foliation of general type on $S$ whose canonical map induces a fibration $f:\, \wt S \to B$ after a sequence of possible blowing-ups $\sigma:\,\wt S \to S$.
	Suppose that $\sF$ is algebraically integrable, $d=K_{\wt\sF}\cdot F=1$, and $g(F)=1$, where $F$ is a general fiber of $f$. 
	Then $p_g(\sF)=g(B)\geq 2$, and
	$$|K_{\wt \sF}|=|f^*K_B|+\Delta.$$
\end{corollary}
\begin{proof}
	Since $\sF$ is of general type and $g(F)=1$,
	it follows that the induced foliation $\wt\sF=\sigma^*\sF$ is different from the foliation defined by taking the saturation of $\ker(\df:\,T_{\wt S} \to f^*T_{B})$ in $T_{\wt S}$.
	Hence by \autoref{lem-5-1} above,
	for any fiber $F_0$ of $f$, there is at least one component $C\subseteq F_0$ which is not $\sF$-invariant.
	Therefore, similar to \autoref{prop-4-3} and \autoref{prop-4-4},
	one shows that the vertical part $\Delta_v$ of $\Delta$ is contained in the fixed part of the complete linear system $|K_{\wt \sF}|$;
	namely
	$$|K_{\wt \sF}|=|f^*K_B+\Delta_h|+\Delta_v.$$
	On the other hand, since $d=K_{\wt\sF}\cdot F=1$,
	it follows that $\Delta_h=C_0$ is a section of $f$,
	and hence $$h^0(F,\mathcal{O}_F(K_{\wt \sF}))=h^0(F,\mathcal{O}_F(C_0))=1.$$
	Hence by \autoref{lem-41-2}, $\Delta_h$ is also contained in the fixed part of $|K_{\wt \sF}|$.
	Hence
	$$|K_{\wt \sF}|=|f^*K_B|+\Delta.$$
	Thus $g(B)=p_g(\sF)\geq 2$ as required.
\end{proof}

\begin{proposition}\label{prop-5-1}
	Let $\sF$ be a reduced foliation of general type on a smooth projective surface $S$. Suppose that the surface $S$ is not ruled, and that $\sF$ is algebraically integrable.
	Then
	$$\vol(\sF)\geq 2\big(p_g(\sF)-2\big).$$
\end{proposition}
\begin{proof}
	Following the proof of \autoref{thm-1-5}\,(ii) above,
	it is enough to consider the case when the canonical map $\varphi$
	induces a fibration $f:\, \wt S \to B$ with $d=K_{\wt \sF}\cdot F=1$,
	where $F$ is a general fiber of $f$.
	But in this case, by the \autoref{cor-5-1} above, it holds $p_g(\sF)=g(B)\geq 2$, and
	$$|K_{\wt \sF}|=|f^*K_B|+\Delta.$$
	Let $K_{\wt \sF}=P+N$ be the Zariski decomposition.
	Then the horizontal part $\Delta_h=C_0$ is not contained in the negative part $N$,
	or equivalently $\Delta_h=C_0\subseteq P$, since $g(C_0)=g(B)\geq 2$.
	It follows that $P\cdot F=K_{\wt\sF}\cdot F=1$.
	Hence 
	\[\vol(\sF)=P^2\geq P\cdot f^*K_B=\big(2g(B)-2\big)\,P\cdot F=\big(2p_g(\sF)-1\big)>2\big(p_g(\sF)-2\big).\qedhere\]
\end{proof}

\begin{remark}\label{rem-5-1}
	Taking $d=2$ in \autoref{exam-3-1}, one obtains a reduced foliated surface $(S,\sF)$ of general type with
	$$\vol(\sF) = 2p_g(\sF)-4.$$
	Moreover, if one choose $m>\frac{3n+4}{2}$ in \autoref{exam-3-1}, then $$K_{S}^2=K_{S/\bbp^1}^2-8=4m-6n-8>0.$$
	Hence $S$ would be a surface of general type.
	This provides an example of reduced foliated surface reaching the equality in \eqref{eqn-1-14}.
\end{remark}

\vspace{3mm}
\appendix
\numberwithin{equation}{section}
\renewcommand{\theequation}{\Alph{section}.\arabic{equation}}
\section{The Riemann-Hurwitz formula for foliated surfaces}\label{sec-appendix}
In the appendix, we present the Riemann-Hurwitz formula for foliated surfaces.
\begin{theorem}\label{thm-a-1}
	Let $\sF$ be a foliation on a smooth projective surface $S$,
	and $\Pi:\,\wt{S} \to S$ be a generically finite map with $\wt{S}$ being smooth.
	Denote by $\wt\sF:=\Pi^*\sF$ the induced foliation on $\wt{S}$.
	\begin{enumerate}[$(i).$]
		\item The canonical divisor of $\wt\sF$ can be expressed as
		\begin{equation}\label{eqn-A-1}
		K_{\wt \sF} =\Pi^*(K_{\sF}) + \sum (t_i-1)R_i + \mathcal{E},
		\end{equation}
		where the sum is taken over all ramified curves $R_i$'s which are not $\wt{\sF}$-invariant,
		$t_i$ is the ramification index of $R_i$,
		and $\mathcal{E}$ is supported on the set of curves contracted by $\Pi$.\vspace{2mm}
		
		\item If moreover $\sF$ is reduced, then the divisor $\mathcal{E}$ in \eqref{eqn-A-1} is effective.
	\end{enumerate}
\end{theorem}
\begin{proof}
	(i).
	The first statement might be well-known to the experts; see for instance \cite[pp. 19\,-\,20]{bru-04} for the ramified covering case.
	We include a proof here to be self-contained.
	
	Let $\{(U_i,\omega_i)\}_{i\in I}$ be a collection of local one-forms defining the foliation $\sF$.
	Then the foliation $\wt{\sF}$ is defined by the data
	$$\{(V_{ij}, \wt\omega_{ij})\},~i\in I, j\in J,$$
	where $V_{ij}$ is an open covering of $\wt{S}$ with $\Pi(V_{ij}) \subseteq U_i$,
	and $\wt\omega_{ij}=\frac{\Pi^*(\omega_i)}{h_{ij}}$ with $h_{ij}$ being some holomorphic function over $V_{ij}$ satisfying $\div(h_{ij})=\div\big(\Pi^*(\omega_i)|_{V_{ij}}\big)$.
	From \eqref{eqn-2-9}, it follows that
	\begin{equation}\label{eqn-2-4}
	N_{\wt\sF}^*=\Pi^*(N_{\sF}^*)+\div(\Pi^*\omega),
	\end{equation}
	where $N_{\wt\sF}^*$ (resp. $N_{\sF}^*$) is the conormal bundle of $\wt\sF$ (resp. $\sF$),
	and $\div(\Pi^*\omega)$ is the divisor defined by $\big\{\big(\Pi^{-1}U_i, \div(\Pi^*\omega_i)\big)\big\}_{i\in I}$.
	On the other hand, by the Riemann-Hurwitz formula for surfaces \cite{bhpv}, we have
	\begin{equation}\label{eqn-2-5}
	K_{\wt{S}} = \Pi^*K_S+ \div\big(\hspace{-1mm}\det(\mathrm{Jac})\big),
	\end{equation}
	where $\mathrm{Jac}$ is the Jacobian matrix of $\Pi$.
	Locally, suppose $\Pi$ is defined by
	$$x=\alpha(u,v),\qquad y=\beta(u,v),$$
	where $(x,y)$ and $(u,v)$ are respectively local coordinates of $S$ and $\wt{S}$.
	Then $$\mathrm{Jac}=\left(\begin{aligned}
	\alpha_u~&~\alpha_v \\
	\beta_u~&~\beta_v
	\end{aligned}\right),$$
	where $\alpha_u=\frac{\partial \alpha}{\partial u}$, and similar for other notations.
	Combing \eqref{eqn-2-4} with \eqref{eqn-2-5}, it follows that
	\begin{equation}\label{eqn-2-6}
	K_{\wt\sF}=\Pi^*(K_{\sF}) + \div\big(\hspace{-1mm}\det(\mathrm{Jac})\big)-\div(\Pi^*\omega).
	\end{equation}
	For any irreducible curve $C$ contained in $\div\big(\hspace{-1mm}\det(\mathrm{Jac})\big)$ or $\div(\Pi^*\omega)$,
	it is clear that $C$ is either contained in the ramified divisor of $\Pi$ or contracted by $\Pi$.
	
	Let $C=R_i$ is contained in the ramified divisor of $\Pi$ with ramification index $t=t_i$.
	Take a general point $p\in C$, and let $q=\Pi(p)$.
	Then locally we may assume that $C$ is defined by $u=0$, and that $\Pi$ is given by
	$$x=u^t, \qquad y=v,$$
	where $(x,y)$ and $(u,v)$ are respectively local coordinates of $S$ and $\wt{S}$.
	Hence around a neighborhood of $p$,
	$$\div\big(\hspace{-1mm}\det(\mathrm{Jac})\big)=(t-1)C.$$
	As $p$ is general, $q=\Pi(p)$ is not a singularity of $\sF$.
	Locally around $q$, we may assume that $\sF$ is defined by $\omega=m_1dx+m_2dy$, where the two complex number $m_1(q),m_2(q)$ are not simultaneously zero.
	It follows that $$\Pi^*\omega=tu^{t-1}m_1du+m_2dv.$$	
	If the image $\Pi(C)$ is not $\sF$-invariant, then $m_2(q)\neq 0$ since $q$ is general, and hence
	$\div(\Pi^*\omega)=0$ around a neighborhood of $p$;
	if the image $\Pi(C)$ is $\sF$-invariant, then we may assume that $u$ divides $m_2$, and that $m_1(q)\neq 0$ since $q$ is general (it is not a singularity of $\sF$). Hence
	$\div(\Pi^*\omega)=(t-1)C$ around a neighborhood of $p$.
	This completes the proof in view of \eqref{eqn-2-6}.
	
	\vspace{2mm}
	(ii).
	Take any curve $E$ contracted by $\Pi$, and denote by $q=\Pi(E)\in S$.
	Let $t_1=v_E\big(\det(\mathrm{Jac})\big)$ (resp. $t_2=v_E(\Pi^*\omega)$) be the multiplicity of $\det(\mathrm{Jac})$ (resp. $\Pi^*\omega$) along $E$,
	where $\mathrm{Jac}$ is the Jacobian matrix of $\Pi$ and $\omega$ is a local one-form defining $\sF$.
	By \eqref{eqn-2-6}, it suffices to prove that $t_1\geq t_2$ if the foliation $\sF$ is reduced.
	
	Let $p\in E$ be a general point, $(u,v)$ be a local coordinate of $\wt{S}$ such that $E=\{u=0\}$, and $(x,y)$ be a local coordinate of $S$ around $q$.
	Since $E$ is contracted by $\Pi$, we may assume that the map $\Pi$ is defined by
	\begin{equation}\label{eqn-2-7}
	x=u^{k_1}\alpha(u,v),\qquad y=u^{k_2}\beta(u,v),
	\end{equation}
	where $k_1, k_2$ are positive integers, and $u$ divides neither $\alpha(u,v)$ nor $\beta(u,v)$.
	Then
	\begin{equation*} 
	\mathrm{Jac}=\left(\begin{aligned}
	u^{k_1-1}(k_1\alpha+u\alpha_u)~\,~&~\,~u^{k_1}\alpha_v \\
	u^{k_2-1}(k_2\beta+u\beta_u)~\,~&~\,~u^{k_2}\beta_v
	\end{aligned}\right).
	\end{equation*}
	It follows that
	\begin{equation*}
	\det(\mathrm{Jac})=u^{k_1+k_2-1}\big((k_1\alpha+u\alpha_u)\beta_v-(k_2\beta+u\beta_u)\alpha_v\big).
	\end{equation*}
	This implies that
	\begin{equation}\label{eqn-a-5}
	t_1=v_E\big(\det(\mathrm{Jac})\big) \geq k_1+k_2-1.
	\end{equation}
	Let $\omega=f(x,y)dx+g(x,y)dy$ around $q$.
	Then
	\begin{equation}\label{eqn-a-6}
	\Pi^*\omega =\big(f(x,y)~\,~g(x,y)\big)\cdot \mathrm{Jac}\cdot {du \choose dv}
	=F(u,v) du+G(u,v)dv,
	\end{equation}
	where
	$$\begin{aligned}
	F(u,v)&\,=f(u^{k_1}\alpha,u^{k_2}\beta)u^{k_1-1}(k_1\alpha+u\alpha_u)+g(u^{k_1}\alpha,u^{k_2}\beta)u^{k_2-1}(k_2\beta+u\beta_u),\\
	G(u,v)&=\,f(u^{k_1}\alpha,u^{k_2}\beta)u^{k_1}\alpha_v+g(u^{k_1}\alpha,u^{k_2}\beta)u^{k_2}\beta_v.
	\end{aligned}$$
	
	\vspace{2mm}
	We consider first the case when $q=\Pi(E)$ is not a singular point of $\sF$.
	Then either $f(q)\neq 0$ or $g(q)\neq 0$.
	Let $\mathrm{Jac}^*$ be the adjoint matrix of $\mathrm{Jac}$.
	Viewing $(F,\,G)$ as a row vector, one gets
	$$(F,~G) \cdot \mathrm{Jac}^*= (f,~g) \cdot \mathrm{Jac} \cdot \mathrm{Jac}^* 
	=\big(f\cdot \det(\mathrm{Jac}),~g\cdot \det(\mathrm{Jac})\big).$$
	Hence in this case,
	$$t_2= v_E(\Pi^*\omega)=\min\big\{v_E(F),\,v_E(G)\big\} \leq v_E\big(\det(\mathrm{Jac})\big)=t_1.$$
	
	\vspace{2mm}
	Consider next the case when $q=\Pi(E)$ is a singularity of $\sF$, which is reduced by assumption.
	In this case, we claim that
	\begin{claim}\label{claim-2-1}
		Let $\Pi:\,\wt{S} \to S$ be generically finite map as in \autoref{thm-a-1},
		and $E$ be a curve contracted by $\Pi$.
		Let $p\in E$ be a general point, and $(u,v)$ be a local coordinate around $p$ with $E=\{u=0\}$.
		Suppose that $q=\Pi(E)$ is a reduced singularity of $\sF$, and $\omega$ is a local one-form defining $\sF$ around $q$. 
		There exists a local coordinate $(x,y)$ around $q$
		satisfying the following.
		\begin{enumerate}[(i).]
			\item Write $\omega=f(x,y)dx+g(x,y)dy$. Then
			\begin{equation}\label{eqn-a-1}
			f(x,y)=-\lambda_2y+f_2(x,y),\qquad g(x,y)=\lambda_1x+g_2(x,y),
			\end{equation}
			where $\lambda_1,\lambda_2$ are the two eigenvalues of $\sF$ at $p$,
			and monomials in $f_2(x,y), g_2(x,y)$ are of degree at least two.
			\item There exist integers $k_2\geq k_1 \geq 1$ such that
			$\Pi$ is locally defined by
			\begin{equation}\label{eqn-a-2}
			x=u^{k_1}\alpha(u,v),\qquad y=u^{k_2}\beta(u,v).
			\end{equation}
			\item If $k=k_2/k_1$ is an integer, then $u$ does not divide $\beta(u,v)-c\alpha(u,v)^{k}$ for any constant $c\in \mathbb{C}$.
		\end{enumerate}
	\end{claim}
	%
	%
	
	We postpone the proof of the above claim.
	Let $(x,y)$ be a local coordinate satisfying the three conditions in \autoref{claim-2-1}.
	Let $\omega=f(x,y)dx+g(x,y)dy$.
	Then we can write $f(x,y)$ as $$f(x,y)=-\lambda_2y+f_2(x,y)=-\lambda_2y+h_1(x)+yh_2(x,y),$$
	where $h_2(0,0)=0$, and $h_1(x)=\sum\limits_{j\geq j_0} a_jx^j$ with $j_0\geq 2$ and $a_{j_0}\neq 0$.
	It follows that
	$$\begin{aligned}
	f(u^{k_1}\alpha,u^{k_2}\beta)=\,&-\lambda_2u^{k_2}\beta+\sum_{j\geq 2} a_ju^{jk_1}\alpha^j+u^{k_2}\beta\cdot h_2(u^{k_1}\alpha,u^{k_2}\beta)\\
	=\,&\left\{\begin{aligned}
	&u^{k_2}\big(-\lambda_2\beta+u\,\varphi_1(u,v)\big), &\quad&\text{if~}k_2<j_0k_1;\\[2mm]
	&u^{k_2}\big(-\lambda_2\beta+a_{j_0}\alpha^{j_0}+u\,\varphi_2(u,v)\big), &\quad&\text{if~}k_2=j_0k_1;\\[2mm]
	&u^{j_0k_1}\big(a_{j_0}\alpha^{j_0}+u\,\varphi_3(u,v)\big), &\quad&\text{if~}k_2>j_0k_1,
	\end{aligned}\right.
	\end{aligned}$$
	where $\varphi_i(u,v)$'s are some holomorphic functions for $i=1,2,3$.
	Note also that
	$$g(u^{k_1}\alpha,u^{k_2}\beta)=\lambda_1u^{k_1}\alpha+g_2(u^{k_1}\alpha,u^{k_2}\beta)=u^{k_1}\big(\lambda_1\alpha+\varphi_4(u,v)\big),$$
	where $\varphi_4(u,v)$ is holomorphic.
	
	Suppose that $k_2<j_0k_1$. Then (where $F(u,v)$ is given in \eqref{eqn-a-6})
	$$\begin{aligned}
	F(u,v)=\,&f(u^{k_1}\alpha,u^{k_2}\beta)u^{k_1-1}(k_1\alpha+u\alpha_u)+g(u^{k_1}\alpha,u^{k_2}\beta)u^{k_2-1}(k_2\beta+u\beta_u)\\
	=\,&u^{k_1+k_2-1}\big((-\lambda_2\beta+u\,\varphi_1)(k_1\alpha+u\alpha_u)+(\lambda_1\alpha+\varphi_4(u,v))(k_2\beta+u\beta_u)\big)\\
	=\,&u^{k_1+k_2-1}\big((-\lambda_2k_1+k_2\lambda_1)\alpha\beta+u\, \psi_1(u,v)\big),
	\end{aligned}$$
	where $\psi_1(u,v)$ is some holomorphic function.
	Since $q$ is a reduced singularity of $\sF$, the complex number $-\lambda_2k_1+k_2\lambda_1$ can never be zero.
	Hence
	$$v_E(F(u,v))=k_1+k_2-1.$$
	
	Suppose that $k_2=j_0k_1$. Then
	$$\begin{aligned}
	F(u,v)=\,&f(u^{k_1}\alpha,u^{k_2}\beta)u^{k_1-1}(k_1\alpha+u\alpha_u)+g(u^{k_1}\alpha,u^{k_2}\beta)u^{k_2-1}(k_2\beta+u\beta_u)\\
	=\,&u^{k_1+k_2-1}\big((-\lambda_2\beta+a_{j_0}\alpha^{j_0}+u\,\varphi_2)(k_1\alpha+u\alpha_u)+(\lambda_1\alpha+\varphi_4(u,v))(k_2\beta+u\beta_u)\big)\\
	=\,&u^{k_1+k_2-1}\Big(\alpha\big((-\lambda_2k_1+k_2\lambda_1)\beta+k_1a_{j_0}\alpha^{j_0}\big)+u\, \psi_2(u,v)\Big),
	\end{aligned}$$
	where $\psi_2(u,v)$ is some holomorphic function.
	As $q$ is a reduced singularity of $\sF$, $-\lambda_2k_1+k_2\lambda_1\neq 0$.
	By \autoref{claim-2-1},
	$u$ does not divide $(-\lambda_2k_1+k_2\lambda_1)\beta+k_1a_{j_0}\alpha^{j_0}$,
	and hence
	$$v_E(F(u,v))=k_1+k_2-1.$$
	
	Finally, we assume that $k_2>j_0k_1$.
	Then
	$$\begin{aligned}
	F(u,v)=\,&f(u^{k_1}\alpha,u^{k_2}\beta)u^{k_1-1}(k_1\alpha+u\alpha_u)+g(u^{k_1}\alpha,u^{k_2}\beta)u^{k_2-1}(k_2\beta+u\beta_u)\\
	=\,&u^{j_0k_1+k_1-1}\big(a_{j_0}\alpha^{j_0}+u\,\varphi_3(u,v)\big)(k_1\alpha+u\alpha_u)+u^{k_1+k_2-1}\big(\lambda_1\alpha+\varphi_4(u,v)\big)(k_2\beta+u\beta_u)\\
	=\,&u^{j_0k_1+k_1-1}\big(k_1a_{j_0}\alpha^{j_0+1}+u\, \psi_3(u,v)\big),
	\end{aligned}$$
	where $\psi_2(u,v)$ is some holomorphic function.
	Hence
	$$v_E(F(u,v))=j_0k_1+k_1-1<k_1+k_2-1.$$
	In any case, we have proved that $v_E(F(u,v)) \leq k_1+k_2-1$.
	Therefore,
	$$t_2=v_E(\Pi^*\omega) =\min\big\{v_E(F),\,v_E(G)\big\}\leq v_E(F(u,v)) \leq k_1+k_2-1\leq t_1.$$
	This completes the proof.
\end{proof}

\begin{proof}[Proof of \autoref{claim-2-1}]
	Since $q$ is a reduced singularity of $p$,
	the two eigenvalues $\{\lambda_1,\lambda_2\}$ (cf. \autoref{def-2-1}) are different.
	In particular, the corresponding matrix $(Dv)(q)$ at the singularity $q$ is diagonalizable.
	In other word, one can choose a local coordinate $(x,y)$ around $q$ such that
	the one form $\omega=f(x,y)dx+g(x,y)dy$ satisfies \eqref{eqn-a-1}.
	
	Since $E$ is contracted by $\Pi$, there exist positive integers $k_1,k_2$ such that $\Pi$ is locally defined by
	\begin{equation*}
	x=u^{k_1}\alpha(u,v),\qquad y=u^{k_2}\beta(u,v),
	\end{equation*}
	where $\alpha(u,v), \beta(u,v)$ are holomorphic such that $u$ divides neither $\alpha(u,v)$ nor $\beta(u,v)$.
	By exchanging $\{x,y\}$ if necessary we may assume that $k_1 \leq k_2$.
	Note that
	\begin{equation*} 
	\mathrm{Jac}=\left(\begin{aligned}
	u^{k_1-1}(k_1\alpha+u\alpha_u)~\,~&~\,~u^{k_1}\alpha_v \\
	u^{k_2-1}(k_2\beta+u\beta_u)~\,~&~\,~u^{k_2}\beta_v
	\end{aligned}\right).
	\end{equation*}
	It follows that
	$$\det(\mathrm{Jac})=u^{k_1+k_2-1}\big((k_1\alpha+u\alpha_u)\beta_v-(k_2\beta+u\beta_u)\alpha_v\big).$$
	This implies that 
	\begin{equation}\label{eqn-a-4}
	k_1+k_2-1\leq v_E\big(\det(\mathrm{Jac})\big).
	\end{equation}
	
	If either $k=k_2/k_1$ is not an integer, or $k$ is an integer but $u$ does not divide $\beta(u,v)-c\alpha(u,v)^{k}$ for any constant $c\in \mathbb{C}$,
	then we are done.
	Suppose now that
	$k=k_2/k_1$ is an integer, and $u$ divides $\beta(u,v)-c_0\alpha(u,v)^{k}$ for some constant $c_0\in \mathbb{C}$.
	Let $\beta(u,v)-c_0\alpha(u,v)^{k}=u^\ell \tilde \beta(u,v)$ such that $u$ does not divide $\tilde\beta(u,v)$.
	Let
	$$\tilde x=x,\qquad \tilde y=y-c_0x^k.$$
	It is clear that $(\tilde x, \tilde y)$ is again a local coordinate of $q$,
	and \eqref{eqn-a-1} still holds.
	Moreover,
	$$\begin{aligned}
	&\tilde x=x=u^{k_1}\alpha(u,v),\\
	&\tilde y=u^{k_2}\beta(u,v)-c_0u^{kk_1}\alpha(u,v)^k=u^{k_2}(\beta-c_0\alpha(u,v)^k)=u^{\tilde k_2} \tilde \beta(u,v),
	\end{aligned}$$
	where $\tilde k_2=k_2+\ell>k_2$.
	This shows that \eqref{eqn-a-2} also hold.
	By \eqref{eqn-a-4}, $\tilde k_2\leq v_E\big(\det(\mathrm{Jac})\big)+1-k_1$.
	Hence the above process must stop.
	Equivalently, we will finally find a local coordinate $(x,y)$ around $q$ satisfying the three conditions.
\end{proof}

\begin{corollary}\label{cor-a-1}
	Let $\Pi:\, \wt S \to S$ be as in \autoref{thm-a-1}, and $\sF$ be a reduced foliation on $S$ such that $\wt \sF=\Pi^*\sF$ is also reduced.
	Suppose that the ramified divisor of $\Pi$ consists of curves which are all $\wt\sF$-invariant. Then
	$$\vol(\wt\sF)=\deg(\Pi)\cdot \vol(\sF).$$ 
\end{corollary}
\begin{proof}
	By \autoref{thm-a-1},
	$$K_{\wt \sF}=\Pi^*(K_{\sF})+\mathcal{E},$$
	where $\mathcal{E}\geq 0$ is supported on the curves which are contracted by $\Pi$.
	Let $$K_{\sF}=P+N,$$ be the Zariski decomposition with $P$ and $N$ being respectively the nef and negative parts.
	Then one checks easily that the Zariski decomposition of $K_{\wt\sF}$ is as follows,
	$$K_{\wt \sF}=\wt P+\wt N,$$
	where $\wt P:=\Pi^*(P)$ and $\wt N:=\big(\Pi^*(N)+\mathcal{E}\big)$ are respectively the nef and negative parts.
	In particular,
	\[\vol(\wt\sF)=\wt P^2=\deg(\Pi)\cdot P^2=\deg(\Pi)\cdot \vol(\sF).\qedhere\]
\end{proof}

\begin{remark}\label{rem-a-1}
	(1). If $\Pi$ is generically one-to-one, i.e., $\Pi$ is a birational morphism,
	then $\Pi$ consists of a sequence of blowing-ups and contains no ramified divisor.
	In this case,
	$$K_{\wt\sF} = \Pi^*(K_{\sF}) + \mathcal{E},$$
	where $\mathcal{E}$ is supported over exceptional divisors of $\Pi$.
	If $\mathcal{E}$ is always effective for any such birational morphism $\Pi$,
	then $\sF$ is said to admit canonical singularities \cite{mcq-08}.
	The conclusion (ii) in the above proposition in this case just says that a reduced foliation contains only canonical singularities.
	
	(2). If the foliation $\sF$ is not reduced, then the divisor $\mathcal{E}$ in \eqref{eqn-A-1} may fail to be effective.
	For instance, let $p\in S$ be singularity of $\sF$, and suppose that $\sF$ is locally defined by
	$\omega=xdy-ydx$ around $p$.
	Let $\sigma:\,\wt S \to S$ be the blowing-up centered at $p$ with $\mathcal{E}$ being the exceptional curve, and $\wt\sF$ be the induced foliation on $\wt{S}$.
	Then
	$$K_{\wt\sF} =\sigma^*K_{\sF} -E.$$
\end{remark}


\begin{thebibliography}{ACGH85}
	
	\bibitem[ACGH85]{acgh}
	E.~Arbarello, M.~Cornalba, P.~A. Griffiths, and J.~Harris.
	\newblock {\em Geometry of algebraic curves. {V}ol. {I}}, volume 267 of {\em
		Grundlehren der mathematischen Wissenschaften [Fundamental Principles of
		Mathematical Sciences]}.
	\newblock Springer-Verlag, New York, 1985.
	
	\bibitem[Bea79]{bea-79}
	A.~Beauville.
	\newblock {\em L'application canonique pour les surfaces de type
		g\'{e}n\'{e}ral}.
	\newblock Invent. Math., 55(2):\,121--140, 1979.
	
	\bibitem[BHPV04]{bhpv}
	W.~P. Barth, K.~Hulek, C.~A.~M. Peters, and A.~{Van de Ven}.
	\newblock {\em Compact complex surfaces}, volume~4 of {\em Ergebnisse der
		Mathematik und ihrer Grenzgebiete. 3. Folge. A Series of Modern Surveys in
		Mathematics [Results in Mathematics and Related Areas. 3rd Series. A Series
		of Modern Surveys in Mathematics]}.
	\newblock Springer-Verlag, Berlin, second edition, 2004.
	
	\bibitem[BM16]{bm-16}
	F.~Bogomolov and M.~McQuillan.
	\newblock {\em Rational curves on foliated varieties}.
	\newblock In {\em Foliation theory in algebraic geometry}, Simons Symp., pages
	21--51. Springer, Cham, 2016.
	
	\bibitem[Bom73]{bom-73}
	E.~Bombieri.
	\newblock {\em Canonical models of surfaces of general type}.
	\newblock Inst. Hautes \'{E}tudes Sci. Publ. Math., (42):\,171--219, 1973.
	
	\bibitem[Bru97]{bru-97}
	M.~Brunella.
	\newblock {\em Some remarks on indices of holomorphic vector fields}.
	\newblock Publ. Mat., 41(2):\,527--544, 1997.
	
	\bibitem[Bru99]{bru-99}
	M.~Brunella.
	\newblock {\em Minimal models of foliated algebraic surfaces}.
	\newblock Bull. Soc. Math. France, 127(2):\,289--305, 1999.
	
	\bibitem[Bru04]{bru-04}
	M.~Brunella.
	\newblock {\em Birational geometry of foliations}.
	\newblock Publica\c{c}\~{o}es Matem\'{a}ticas do IMPA. [IMPA Mathematical
	Publications]. Instituto de Matem\'{a}tica Pura e Aplicada (IMPA), Rio de
	Janeiro, 2004.
	
	\bibitem[Cas21]{cas-21}
	P.~Cascini.
	\newblock {\em New directions in the minimal model program}.
	\newblock Boll. Unione Mat. Ital., 14(1):\,179--190, 2021.
	
	\bibitem[CCJ20a]{ccj-202}
	J.~A. Chen, M.~Chen, and C.~Jiang.
	\newblock {\em Addendum to ``{T}he {N}oether inequality for algebraic
		3-folds''}.
	\newblock Duke Math. J., 169(11):\,2199--2204, 2020.
	
	\bibitem[CCJ20b]{ccj-201}
	J.~A. Chen, M.~Chen, and C.~Jiang.
	\newblock {\em The {N}oether inequality for algebraic 3-folds}.
	\newblock Duke Math. J., 169(9):\,1603--1645, 2020.
	\newblock With an appendix by J\'{a}nos Koll\'{a}r.
	
	\bibitem[CF18]{cf-18}
	P.~Cascini and E.~Floris.
	\newblock {\em On invariance of plurigenera for foliations on surfaces}.
	\newblock J. Reine Angew. Math., 744:\,201--236, 2018.
	
	\bibitem[Che17]{chen-17}
	X.~Chen.
	\newblock {\em Xiao's conjecture on canonically fibered surfaces}.
	\newblock Adv. Math., 322:\,1033--1084, 2017.
	
	\bibitem[CHL{\etalchar{+}}24]{chlmssx}
	P.~Cascini, J.~Han, J.~Liu, F.~Meng, C.~Spicer, R.~Svaldi, and L.~Xie.
	\newblock {\em Minimal model program for algebraically integrable adjoint foliated
	structures}.
	\newblock Arxiv:\,2408.14258, 2024.
	
	\bibitem[CJ17]{cj-17}
	M.~Chen and Z.~Jiang.
	\newblock {\em A reduction of canonical stability index of 4 and 5 dimensional
		projective varieties with large volume}.
	\newblock Ann. Inst. Fourier (Grenoble), 67(5):\,2043--2082, 2017.
	
	\bibitem[CP19]{cp-19}
	F.~Campana and M.~P\u{a}un.
	\newblock {\em Foliations with positive slopes and birational stability of
		orbifold cotangent bundles}.
	\newblock Publ. Math. Inst. Hautes \'{E}tudes Sci., 129:\,1--49, 2019.
	
	\bibitem[EV92]{ev-92}
	H.~Esnault and E.~Viehweg.
	\newblock {\em Lectures on vanishing theorems}, volume~20 of {\em DMV Seminar}.
	\newblock Birkh\"{a}user Verlag, Basel, 1992.
	
	\bibitem[GH78]{gh-78}
	P.~Griffiths and J.~Harris.
	\newblock {\em Principles of algebraic geometry}.
	\newblock Pure and Applied Mathematics. Wiley-Interscience [John Wiley \&
	Sons], New York, 1978.
	
	\bibitem[Har77]{har-77}
	R.~Hartshorne.
	\newblock {\em Algebraic geometry}, volume No. 52 of {\em Graduate Texts in
		Mathematics}.
	\newblock Springer-Verlag, New York-Heidelberg, 1977.
	
	\bibitem[HM06]{hm-06}
	C.~D. Hacon and J.~McKernan.
	\newblock {\em Boundedness of pluricanonical maps of varieties of general
		type}.
	\newblock Invent. Math., 166(1):\,1--25, 2006.
	
	\bibitem[Hor76]{hor-76-2}
	E.~Horikawa.
	\newblock {\em Algebraic surfaces of general type with small {$c^{2}_{1}$}.
		{II}}.
	\newblock Invent. Math., 37(2):\,121--155, 1976.
	
	\bibitem[KM98]{km-98}
	J.~Koll\'ar and S.~Mori.
	\newblock {\em Birational geometry of algebraic varieties}, volume 134 of {\em
		Cambridge Tracts in Mathematics}.
	\newblock Cambridge University Press, Cambridge, 1998.
	\newblock With the collaboration of C. H. Clemens and A. Corti, Translated from
	the 1998 Japanese original.
	
	\bibitem[Liu17]{liu-19}
	J.~Liu.
	\newblock {\em On invariance of plurigenera for foliated surface pairs}.
	\newblock arXiv:1707.07092, 2017.
	
	\bibitem[LS20]{ls-20}
	A.~{Lins Neto} and B.~Sc\'{a}rdua.
	\newblock {\em Complex algebraic foliations}, volume~67 of {\em De Gruyter
		Expositions in Mathematics}.
	\newblock De Gruyter, Berlin, [2020] \copyright 2020.
	
	\bibitem[LT24]{lt-24}
	X.~L\"{u} and S.~Tan.
	\newblock {\em The poincar\'e problem for a foliated surface}.
	\newblock arXiv:2404.16293v2, 2024.
	
	\bibitem[L{\"u}20]{lu-20}
	X.~L{\"u}.
	\newblock {\em Canonically fibered surfaces of general type}.
	\newblock J. Inst. Math. Jussieu, 19(1):\,209--229, 2020.
	
	\bibitem[McQ08]{mcq-08}
	M.~McQuillan.
	\newblock {\em Canonical models of foliations}.
	\newblock Pure Appl. Math. Q., 4(3, Special Issue: In honor of Fedor Bogomolov.
	Part 2):\,877--1012, 2008.
	
	\bibitem[Men00]{men-00}
	L.~G. Mendes.
	\newblock {\em Kodaira dimension of holomorphic singular foliations}.
	\newblock Bol. Soc. Brasil. Mat. (N.S.), 31(2):\,127--143, 2000.
	
	\bibitem[Miy87]{miy-87}
	Y.~Miyaoka.
	\newblock {\em Deformations of a morphism along a foliation and applications}.
	\newblock In {\em Algebraic geometry, {B}owdoin, 1985 ({B}runswick, {M}aine,
		1985)}, volume~46 of {\em Proc. Sympos. Pure Math.}, pages 245--268. Amer.
	Math. Soc., Providence, RI, 1987.
	
	\bibitem[Noe70]{noether}
	M.~Noether.
	\newblock {\em Zur {T}heorie des eindeutigen {E}ntsprechens algebraischer
		{G}ebilde von beliebig vielen {D}imensionen}.
	\newblock Math. Ann., 2(2):\,293--316, 1870.
	
	\bibitem[Rei88]{rei-88}
	I.~Reider.
	\newblock {\em Vector bundles of rank {$2$} and linear systems on algebraic
		surfaces}.
	\newblock Ann. of Math. (2), 127(2):\,309--316, 1988.
	
	\bibitem[Rit22]{rito-22}
	C.~Rito.
	\newblock {\em Surfaces with canonical map of maximum degree}.
	\newblock J. Algebraic Geom., 31(1):\,127--135, 2022.
	
	\bibitem[SS23]{ss-23}
	C.~Spicer and R.~Svaldi.
	\newblock {\em Effective generation for foliated surfaces: results and
		applications}.
	\newblock J. Reine Angew. Math., 795:\,45--84, 2023.
	
	\bibitem[Sun94]{sun-94}
	X.~T. Sun.
	\newblock {\em On canonical fibrations of algebraic surfaces}.
	\newblock Manuscripta Math., 83(2):\,161--169, 1994.
	
	\bibitem[Tak06]{tak-06}
	S.~Takayama.
	\newblock {\em Pluricanonical systems on algebraic varieties of general type}.
	\newblock Invent. Math., 165(3):\,551--587, 2006.
	
	\bibitem[Xia85a]{xiao-85-1}
	G.~Xiao.
	\newblock {\em L'irr\'{e}gularit\'{e} des surfaces de type g\'{e}n\'{e}ral dont
		le syst\`eme canonique est compos\'{e} d'un pinceau}.
	\newblock Compositio Math., 56(2):\,251--257, 1985.
	
	\bibitem[Xia85b]{xiao-85}
	G.~Xiao.
	\newblock {\em Finitude de l'application bicanonique des surfaces de type
		g\'en\'eral}.
	\newblock Bull. Soc. Math. France, 113(1):\,23--51, 1985.
	
	\bibitem[Xia86]{xiao-86}
	G.~Xiao.
	\newblock {\em Algebraic surfaces with high canonical degree}.
	\newblock Math. Ann., 274(3):\,473--483, 1986.
	
	\bibitem[Yeu17]{yeung-17}
	S.-K. Yeung.
	\newblock {\em A surface of maximal canonical degree}.
	\newblock Math. Ann., 368(3-4):\,1171--1189, 2017.
	
\end{thebibliography}

\newcommand{\etalchar}[1]{$^{#1}$}

\end{document}